\documentclass[draft,reqno]{amsart}
\usepackage{amssymb,amsmath}
\usepackage{verbatim}
\usepackage{color}

\newtheorem{theorem}{Theorem}
\newtheorem{lemma}{Lemma}
\newtheorem{proposition}{Proposition}
\newtheorem{corollary}{Corollary}

\theoremstyle{definition}
\newtheorem{definition}{Definition}
\newtheorem{example}{Example}

\theoremstyle{remark}
\newtheorem{remark}{Remark}

\numberwithin{equation}{section}
\numberwithin{theorem}{section}
\numberwithin{lemma}{section}
\numberwithin{proposition}{section}
\numberwithin{corollary}{section}
\numberwithin{definition}{section}
\numberwithin{example}{section}
\numberwithin{remark}{section}
\newcommand{\thref}[1]{Theorem \ref{#1}}    
\newcommand{\prref}[1]{Proposition \ref{#1}}
\newcommand{\leref}[1]{Lemma \ref{#1}}
\newcommand{\coref}[1]{Corollary \ref{#1}}

\newcommand{\deref}[1]{Definition \ref{#1}}
\newcommand{\exref}[1]{Example \ref{#1}}

\newcommand{\reref}[1]{Remark \ref{#1}}

\newcommand{\seref}[1]{Section \ref{#1}}

\def\as{associative}

\def\psalg{pseudo\-algebra}
\def\psalgs{pseudo\-algebras}

\def\tp{\otimes}                               
\def\bt{\boxtimes}
\def\<{\langle}
\def\>{\rangle}
\def\ce{{e^{-\chi}}}                    
\def\ti{\widetilde}
\def\wti{\widetilde}
\def\what{\widehat}

\def\d{\partial}
\def\bd{\bar{\partial}}

\def\surjto{\twoheadrightarrow}                
\def\injto{\hookrightarrow}                    
\def\st{\; | \;}                               
\def\symm{S}                                   

\def\ext{{\mathrm{ext}}}

\def\mmod{\;\mathrm{mod}\;}

\def\adsp{\ad^\sp}
\def\di{{\mathrm{d}}}

\def\diz{\di_0}
\def\diru{\di^{\mathrm{R}}}
\newcommand{\kk}{\mathbf{k}}       
\newcommand{\ZZ}{\mathbb{Z}}       

\newcommand{\fd}{{\mathfrak d}}

\newcommand{\fg}{\mathfrak{g}}

\def\al{\alpha}                         
\def\be{\beta}
\def\ga{\gamma}

\def\de{\delta}
\def\De{\Delta}
\def\ep{\varepsilon}
\def\io{\iota}

\def\om{\omega}
\def\Om{\Omega}
\def\ph{\varphi}
\def\si{\sigma}

\def\th{\theta}

\def\g{{\mathfrak{g}}}      
\def\h{{\mathfrak{h}}}      
\def\p{{\mathfrak{p}}}
\def\dd{{\mathfrak{d}}}

\def\gl{{\mathfrak{gl}}}
\def\gld{\gl\,\dd}

\def\sp{{\mathfrak{sp}}}
\def\spd{\sp\,\dd}

\def\csp{{\mathfrak{csp}}}
\def\cspd{\csp\,\dd}

\def\Wd{W(\dd)}    
\def\Sd{S(\dd,\chi)}
\def\Hd{H(\dd,\chi,\om)}
\def\Hdzero{H(\dd,0,\om)}
\def\Hdzeta{H(\dd,0, \di \zeta)}
\def\Kd{K(\dd,\th)}

\def\Kdp{K(\dd',\th)}
\def\A{{\mathcal{A}}}
\def\L{{\mathcal{L}}}
\def\V{{\mathcal{V}}}
\def\W{{\mathcal{W}}}

\def\H{{\mathcal{H}}}
\def\P{{\mathcal{P}}}
\def\K{{\mathcal{K}}}

\def\N{\mathcal{N}}

\def\E{\mathcal{E}}

\def\O{\mathcal{O}}

\def\T{\mathcal{T}}           
\def\V{\mathcal{V}}           %
\def\ue{U}                 

\def\bla{{\zeta}}
\def\ddbla{\dd^{\bla}}

\DeclareMathOperator{\sgn}{sgn}
\DeclareMathOperator{\gr}{gr}

\DeclareMathOperator{\Span}{span}

\DeclareMathOperator{\Ind}{Ind}
\DeclareMathOperator{\ad}{ad}

\DeclareMathOperator{\tr}{tr}

\DeclareMathOperator{\sd}{\ltimes}
\DeclareMathOperator{\id}{id}

\def\symp{S}
\DeclareMathOperator{\fil}{F}      

\DeclareMathOperator{\Der}{Der}

\DeclareMathOperator{\Hom}{Hom}
\DeclareMathOperator{\End}{End}

\DeclareMathOperator{\Cur}{Cur}

\renewcommand\Im{{\mathrm{Im}}}

\DeclareMathOperator{\Vir}{Vir}
\DeclareMathOperator{\sing}{sing}
\DeclareMathOperator{\coef}{coeff}
\DeclareMathOperator{\im}{Im}
\newcommand{\ptl}{\partial}

\def\dizerostar
{\di^{2N}_{\Pi_{\phi + N\chi}}}
\def\VPizero
{\V(\Pi, \kk)}
\def\VPione
{\V(\Pi_{\chi/2}, R(\pi_1))}
\def\imdizerostar{\im\dizerostar}

\def\din
{\di^n_{\Pi_{\phi+n\chi/2}}}
\def\prevdin
{\di^{n-1}_{\Pi_{\phi + n\chi/2}}}
\def\nextdin
{\di^{n+1}_{\Pi_{\phi + n\chi/2}}}
\def\dinstar
{\di^{2N-n}_{\Pi_{\phi + (2N-n)\chi/2}}}
\def\dinminusonestar
{\di^{2N-n+1}_{\Pi_{\phi + (2N-n+2)\chi/2}}}
\def\dindinminusonestar
{\din\dinminusonestar}
\def\dinplusone
{\di^{n+1}_{\Pi_{\phi + (n+2)\chi/2}}}
\def\VPin
{\V(\Pi, R(\pi_n))}
\def\VPinminusone
{\V(\Pi_{\chi/2}, R(\pi_{n-1}))}
\def\VPinplusone
{\V(\Pi_{\chi/2}, R(\pi_{n+1}))}
\def\VPichin
{\V(\Pi_\chi, R(\pi_n))}
\def\VPichinminustwo
{\V(\Pi_\chi, R(\pi_{n-2}))}
\def\imdin{\im \din}
\def\imdinstar{\im \dinstar}
\def\imdindinminusonestar{\im\dindinminusonestar}

\def\diN
{\di^N_{\Pi_{\phi+N\chi/2}}}
\def\DiPi
{\diru_{\Pi_{\phi + (N+2)\chi/2}}}
\def\nextDiPi
{\di^{N+1}_{\Pi_{\phi + N\chi/2}}}
\def\diNminusonestar
{\di^{N+1}_{\Pi_{\phi + (N+2)\chi/2}}}
\def\diNdiNminusonestar
{\diN\diNminusonestar}
\def\VPiN
{\V(\Pi, R(\pi_N))}
\def\VPiNminusone
{\V(\Pi_{\chi/2}, R(\pi_{N-1}))}
\def\VPichiN
{\V(\Pi_\chi, R(\pi_N))}
\def\imdiN{\im\diN}
\def\imDiPi{\im\DiPi}
\def\imdiNdiNminusonestar{\im\diNdiNminusonestar}

\def\dinmochizerostar
{\di^{2N-n+1}_{\Pi_\phi}}
\def\dinchizero
{\di^n_{\Pi_{\phi}}}
\def\dinchizerostar
{\di^{2N-n}_{\Pi_{\phi}}}
\def\dindinmochizerostar
{\dinchizero\dinmochizerostar}
\def\VPizerochizero
{\V(\Pi, \kk)}
\def\DiPiminuschi
{\diru_{\Pi_{\phi + N\chi/2}}}
\def\nextDiPidiNalt
{\di^{N-1}_{\Pi_{\phi+(N-2)\chi/2}} \di^{N+2}_{\Pi_{\phi+(N+2)\chi/2}}}
\def\dinminusonestarminuschi
{\di^{2N-n+1}_{\Pi_{\phi + (2N-n)\chi/2}}}
\def\dinplusonestar
{\di^{2N-n-1}_{\Pi_{\phi + (2N-n)\chi/2}}}
\def\dinplusoneminuschi
{\di^{n+1}_{\Pi_{\phi + n\chi/2}}}
\def\dinplusonestarminuschi
{\di^{2N-n-1}_{\Pi_{\phi + (2N-n-2)\chi/2}}}
\def\diNminusone
{\di^{N-1}_{\Pi_{\phi+(N-1)\chi/2}}}
\def\diNminusonestarbla
{\di^{N+1}_{\Pi_{\phi + (N+1)\chi/2}}}
\def\diNminustwostar
{\di^{N+2}_{\Pi_{\phi + (N+3)\chi/2}}}
\def\VPiNNminusone
{\V(\Pi, R(\pi_{N-1}))}
\def\VPichiNminusone
{\V(\Pi_\chi, R(\pi_{N-1}))}
\def\VPiNminusoneminusone
{\V(\Pi_{\chi/2}, R(\pi_{N-2}))}
\def\VPiNminusoneplusone
{\V(\Pi_{\chi/2}, R(\pi_{N}))}
\def\diNchihalf
{\di^N_{\Pi_{\phi+(N+1)\chi/2}}}
\def\diNminustwostarchihalf
{\di^{N+2}_{\Pi_{\phi+(N+3)\chi/2}}}
\def\DIn
{{D_{\Pi'}^n}}
\def\DIN
{{D_{\Pi'}^N}}
\def\DIR
{{D_{\Pi'}^\textit{R}}}
\def\DInplusone
{{D_{\Pi'}^{n+1}}}
\def\DInstar
{{D_{\Pi'}^{2N-n}}}
\def\DInplusonestar
{{D_{\Pi'}^{2N-n-1}}}
\def\DImostar
{{D_{\Pi'}^{2N+1}}}
\def\DInminusonestar
{{D_{\Pi'}^{2N-n+1}}}
\def\DINminusonestar
{{D_{\Pi'}^{N+1}}}
\def\DINminustwostar%
{{D_{\Pi'}^{N+2}}}
\def\DIzero
{{D_{\Pi'}^0}}
\def\DIone%
{{D_{\Pi'}^1}}
\def\DItwo%
{{D_{\Pi'}^2}}
\def\DIzerostar
{{D_{\Pi'}^{2N}}}
\def\VPiprimezero
{{\V(\Pi', \kk)}}
\def\VPiprimen
{{\V(\Pi', R(\pi_n))}}
\def\VPiprimeN
{{\V(\Pi', R(\pi_N))}}
\def\VPiprimenminusone
{{\V(\Pi', R(\pi_{n-1}))}}
\def\VPiprimenplusone
{{\V(\Pi', R(\pi_{n+1}))}}
\def\VPiprimeNminusone
{{\V(\Pi', R(\pi_{N-1}))}}
\begin{document}

\title[Irreducible Modules over Finite Simple Lie Pseudoalgebras III]
{Irreducible Modules over Finite Simple Lie Pseudoalgebras III. \\
Primitive Pseudoalgebras of Type $H$}

\author[B.~Bakalov]{Bojko Bakalov}
\address{Department of Mathematics,
North Carolina State University,
Raleigh, NC 27695, USA}
\email{bojko\_bakalov@ncsu.edu}
\thanks{The first author was supported in part by a Simons Foundation grant 584741}

\author[A.~D'Andrea]{Alessandro D'Andrea}
\address{Dipartimento di Matematica,
Istituto ``Guido Castelnuovo'',
Universit\`a di Roma ``La Sapienza'',
00185 Rome, Italy}
\email{dandrea@mat.uniroma1.it}
\thanks{The second author was supported in part by Ateneo fundings from Sapienza University in Rome.}

\author[V.~G.~Kac]{Victor G.~Kac}
\address{Department of Mathematics, MIT, Cambridge, MA 02139, USA}
\email{kac@math.mit.edu}
\thanks{The first and third authors were supported in part by the Bert and Ann Kostant fund}

\date\today

\begin{abstract}
A Lie conformal algebra is an algebraic structure that encodes the singular part of the operator product expansion of chiral fields in conformal field theory. A \emph{Lie pseudoalgebra} is a generalization of this structure, for which the algebra of polynomials $\kk[\partial] $ in the indeterminate $ \partial $ is replaced by the universal enveloping algebra ${U} (\mathfrak{d}) $ of a finite-dimensional Lie algebra $ \fd $ over the base field $ \kk. $ The finite (i.e., finitely generated over ${U}(\mathfrak{d}) $) simple Lie pseudoalgebras were classified in our 2001 paper \cite{BDK}. The complete list consists of primitive Lie pseudoalgebras of type $ W, S, H, $  and $ K, $ and of current Lie pseudoalgebras over them or over simple finite-dimensional Lie algebras. The present paper is the third in our series on representation theory of simple Lie pseudoalgebras. In the first paper, we showed that any finite irreducible module over a primitive Lie pseudoalgebra of type $ W $ or $ S $ is either an irreducible tensor module or the image of the differential in a member of the \emph{pseudo de Rham complex}. In the second paper, we established a similar result for primitive Lie pseudoalgebras of type $K,$ with the pseudo de Rham complex replaced by a certain reduction, called the \emph{contact pseudo de Rham complex}. This reduction in the context of contact geometry was discovered by M.\ Rumin \cite{Ru}. In the present paper, we show that for primitive Lie pseudoalgebras of type $H$, a similar to type $K$ result holds with the contact pseudo de Rham complex replaced by a suitable complex. However, the type $H$ case in more involved, since the annihilation algebra is not the corresponding Lie-Cartan algebra, as in other cases, but an irreducible central extension. When the action of the center of the annihilation algebra is trivial, this complex is related to work by M.\ Eastwood \cite{E} on conformally symplectic geometry, and we call it \emph{conformally symplectic pseudo de Rham complex}.
\end{abstract}

\maketitle
\tableofcontents


\section{Introduction}\label{sintro}

The present paper is the third in our series of papers on representation theory of simple Lie pseudoalgebras, the first two of which are \cite{BDK1} and \cite{BDK2}.
As in these papers, we will work over an algebraically closed field $\kk$ of
characteristic $0$. Unless otherwise specified, all vector spaces,
linear maps and tensor products will be considered over $\kk$.

Recall that a \emph{Lie pseudoalgebra} is a (left) module $ L $ over a cocommutative  Hopf algebra $ H $, endowed with a pseudobracket
\[ L \otimes L \rightarrow (H \otimes H) \otimes_H L, \qquad a \otimes b \mapsto [a \ast b], \]
which is an $ H $-bilinear map of $ H $-modules, satisfying some analogs of the skew-symmetry and Jacobi identity of a Lie algebra bracket (see \cite{BD}, \cite{BDK}, or \eqref{psss} and \eqref{psjac} in Section \ref{slieps} of the present paper).

In the case when $ H = \kk$, this notion coincides with that of a Lie algebra. Furthermore, any Lie algebra $ \fg $ gives rises to a Lie pseudoalgebra $ \operatorname{Cur} \fg = H \otimes \fg  $ over $ H  $ with  pseudobracket
\[ [(1 \otimes a) \ast (1 \otimes b)] = (1 \otimes 1) \otimes_H [a,b], \]
extended to the whole $ \operatorname{Cur} \fg $ by $H$-bilinearity.  

In the case where $ H = \kk [\partial], $ the algebra of polynomials in an indeterminate $ \ptl $ with the comultiplication $ \Delta (\ptl) = \ptl \otimes 1  + 1  \otimes \ptl, $ the notion of a Lie pseudoalgebra coincides with that of a \emph{Lie conformal algebra} \cite{K}. The main result of \cite{DK} states that in this case any finite (i.e., finitely generated over $ H = \kk [\ptl] $) simple Lie pseudoalgebra is isomorphic to either $ \operatorname{Cur} \fg $ with simple finite-dimensional  $ \fg $, or to the Virasoro pseudoalgebra $ \Vir = \kk [\ptl ] \ell, $ where
\[ [\ell \ast \ell] = (1 \otimes \ptl - \ptl \otimes 1 ) \otimes_{\kk [\ptl]} \ell. \]

In \cite{BDK}, we generalized this result to the case where $ H = U (\fd), $ where $ \fd $ is a finite-dimensional Lie algebra. The generalization of the Virasoro pseudoalgebra is the Lie pseudoalgebra $ W (\fd) = H \otimes \fd$ with the pseudobracket
\begin{align*}
[ (1 \otimes a) \ast (1 \otimes b)]  
&= (1 \otimes 1) \otimes_H (1 \otimes [a,b]) 
\\
&\quad+ (b \otimes 1) \otimes_H (1 \otimes a)
 - (1 \otimes a) \otimes_H (1 \otimes b). 
\end{align*}

The main result of \cite{BDK} is that all non-zero subalgebras of the Lie pseudoalgebra $ W (\fd) $ are simple and non-isomorphic, and along with $ \Cur \fg, $ where $ \fg $ is a simple finite-dimensional Lie algebra, they provide a complete list of finitely generated over $ H $ simple Lie pseudoalgebras.
Furthermore, in \cite{BDK} we gave a description of all subalgebras of $ W (\fd). $ Namely, a complete list consists of the ``primitive'' series: the special Lie pseudoalgebras $ S(\fd, \chi), $ the Hamiltonian Lie pseudoalgebras $ H(\fd, \chi, \omega) $, the contact Lie pseudoalgebras $ K(\fd, \theta), $ and their ``current'' extensions. 

The geometric meaning of the data $ \chi, \omega, $ and $ \theta $ is as follows: $  \chi \in \mathfrak{d}^\ast $ is a closed $1$-form (= trace-form), i.e., $ (d \chi )(a\wedge b) := \chi ([a,b]) =0$; $\omega \in \bigwedge^2 \mathfrak{d}^* $ is a conformally symplectic form, i.e., it is non-degenerate 
and 
\[ d\omega + \chi \wedge \omega = 0; \]
finally, $ \theta \in \mathfrak{d}^* $ is a contact 1-form (cf.\ \cite{BDK2}). This explains why representation theory of the $ K $ and $ H  $ type Lie pseudoalgebras is intimately related to the constructions in contact and conformally symplectic geometry of \cite{Ru} and \cite{E}. 

For every Lie pseudoalgebra $ L $, we have a functor \cite{BDK}
\[ Y \mapsto \mathcal{A}_Y L : = Y \otimes_H L \]
that assigns a Lie algebra $  \mathcal{A}_Y L $ to any commutative associative algebra $ Y $ equipped with compatible left and right actions of the Hopf algebra $ H. $ The Lie algebra bracket on $  \mathcal{A}_Y L $ is given by
\[ [x \otimes_H a, y \otimes_H b] = \sum_{i} (xf_i)(yg_i) \otimes_H c_i, \quad \text{if}\quad
 [a \ast b] = \sum_i (f_i \otimes g_i) \otimes_H c_i. \]
The main tool in the study of Lie pseudoalgebras and their representations is the \emph{annihilation algebra $ \mathcal{A}_X L $}, where $ X = H^\ast $ is the commutative associative algebra dual to the coalgebra $ H. $ In particular, a module over a Lie pseudoalgebra $ L $ is the same as a ``conformal'' module over the extended annihilation Lie algebra $ \fd \sd \mathcal{A}_X L $ (see \cite{BDK1} and Proposition \ref{preplal2} below). 

Note that $ X \simeq \mathcal{O}_M := \kk [[t^1, \dots, t^M  ]]$ where $ M = \dim \fd. $ We define a topology on $ \mathcal{O}_M $ with a fundamental system of neighborhoods of 0 given by the powers of the maximal ideal $ (t^1, \ldots, t^M) $. Then the annihilation algebra of the Lie pseudoalgebra $ W(\fd) $ is isomorphic to the \emph{Lie--Cartan algebra $ W_M $} of continuous derivations of $ \mathcal{O}_M $ (see \cite{BDK1} and Section \ref{subwd} below). 
Similar isomorphisms hold for $S$- and $K$-type Lie pseudoalgebras \cite{BDK1}, \cite{BDK2}. However the annihilation algebra of $ H(\fd, \chi, \omega) $ is isomorphic to the Lie algebra structure $ P_M $ on $ \mathcal{O}_M, \ M = 2N, $ given by 
\[ [f,g] = \sum_{i = 1}^{N} \Bigl( \frac{\ptl f}{\ptl t^i} \frac{\ptl g}{\ptl t^{N+i}} 
- \frac{\ptl f}{\ptl t^{N+i}} \frac{\ptl g}{\ptl t^i} \Bigr),   \]
so that $ P_M / \kk 1 $ is isomorphic to the Hamiltonian Lie--Cartan algebra $ H_{2N} $ of vector fields annihilating the standard symplectic form $ \sum_{i = 1}^{N} dt^i \wedge dt^{N+i} $. This explains the name and notation $ H $ of the corresponding Lie pseudoalgebra. We hope that the reader will not confuse this $ H $ with the Hopf algebra $ H = U (\mathfrak{d}) $.

In \cite{BDK1}, we constructed all \emph{finite} (i.e., finitely generated over $ H = U(\fd) $) irreducible modules over the Lie pseudoalgebras $ W(\fd) $ and $ S(\fd, \chi) $. The simplest
non-zero module over $ W(\fd) $ is $ \Omega^0 (\fd) = H $, with the action given by
\begin{equation}\label{e1.1}
(f \otimes a) \ast g = -(f \otimes ga) \otimes_H 1, \qquad f,g \in H,\; a \in \fd.
\end{equation}
A generalization of this construction, called a tensor $ W(\fd) $-module, is as follows \cite{BDK1}. First, given a Lie algebra $ \fg $, define the semidirect sum $ W (\fd) \ltimes \Cur \fg $ as a direct sum of $ H $-modules, for which $ W(\fd) $ is a subalgebra and $ \Cur \fg $ is an ideal, with the following pseudobracket between them:
\[ \left[(f \otimes a) \ast (g \otimes b) \right] = -(f \otimes ga) \otimes_H (1 \otimes b) , 
\qquad f, g \in H, \; a \in \fd, \; b \in \fg . \]
Given a finite-dimensional $ \fg $-module $ V_0 $, we construct a representation of the Lie pseudoalgebra
$ W(\fd) \sd \Cur \fg $ in $ V = H \otimes V_0 $ by (cf.\ \eqref{e1.1}):
\begin{equation}\label{e1.2}
\bigl( (f \otimes a) \oplus (g \otimes b)   \bigr) \ast (h \otimes v) = -(f \otimes ha) \otimes_H (1 \otimes v) + (g \otimes h) \otimes_H (1 \otimes bv),
\end{equation}
where $ f, g, h \in H$, $a \in \fd$, $b \in \fg$, $v \in V_0 $.
Next, we define an embedding of $ W(\fd) $ in $ W(\fd) \ltimes \Cur (\fd \oplus \gld) $ by
\begin{equation}\label{e1.3}
1 \otimes \ptl_i \mapsto (1 \otimes \ptl_i) \oplus \Bigl(   (1 \otimes \ptl_i ) \oplus (1 \otimes \operatorname{ad} \ptl_i + \sum_j \ptl_j \otimes e^j_i)    \Bigr),
\end{equation}
where $ \{\ptl_i\} $ is a basis of $ \fd $ and $ \{e^j_i\} $ is a basis of $\gld$, defined by $ e^j_i (\ptl_k) = \delta^j_k \ptl_i $. Composing this embedding with the action \eqref{e1.2} of $ W(\fd)\ltimes \Cur (\fd \oplus \gld) $, we obtain a $ W(\fd) $-module $ V = H \otimes V_0 $ for each $ (\fd \oplus \gld) $-module $ V_0 $. This module is called a \emph{tensor $ W(\fd) $-module} and is denoted $ \mathcal{T} (V_0) $.

The main result of \cite{BDK1} states that any finite irreducible $ W(\fd) $-module is a unique quotient of a tensor module $ \mathcal{T} (V_0) $ for some finite-dimensional irreducible $ (\fd \oplus \operatorname{gl} \fd) $-module $ V_0 $. Furthermore, it describes all cases where $ \mathcal{T} (V_0) $ are not irreducible, and provides an explicit construction of their irreducible quotients, called the \emph{degenerate $ W(\fd) $-modules}. Namely, we prove in \cite{BDK1} that all degenerate $ W(\fd) $-modules occur as images of the differential $\di $ in the \emph{$ \Pi $-twisted pseudo de Rham complex} of $ W(\fd) $-modules
\begin{equation}\label{e1.4}
0 \rightarrow \Omega^0_\Pi (\fd) \overset{\di}{\rightarrow} \Omega^1_\Pi \overset{\di}{\rightarrow} \cdots \overset{\di}{\rightarrow} \Omega^{\dim \fd}_\Pi (\fd).
\end{equation}
Here $ \Pi $ is a finite-dimensional irreducible $ \fd $-module and $ \Omega^n_\Pi (\fd) = \mathcal{T} (\Pi \otimes \bigwedge^n \fd^\ast) $ is the space of pseudo $ n $-forms.

In the present paper, we construct all finite irreducible modules over the Hamiltonian Lie pseudoalgebra $ H(\fd, \chi, \omega). $
This Lie pseudoalgebra is constructed as follows. Choose a basis $ \{  \ptl_i  \}^{2N}_{i =1} $ of $ \mathfrak{d}, $ and the dual basis $ \{  \ptl^i  \}^{2N}_{i =1} $ with respect to the (non-degenerate) bilinear form $ \omega $ on $ \mathfrak{d}, $ so that $ \omega (\ptl^i \wedge \ptl_j) = \delta^i_j . $
Let
\begin{equation}\label{e1.5}
r = \sum_{i = 1}^{2N} \ptl_i \otimes \ptl^i= - \sum_{i=1}^{2N}\ptl^i  \otimes \ptl_i = \sum_{i,j = 1}^{2N} r^{ij} \ptl_i \otimes \ptl_j ,
\end{equation} 
and define $ s \in \mathfrak{d} $ by
\[ \chi (a) = \omega (s \wedge a), \qquad a \in \mathfrak{d}. \]
Let, as before, $ H = U (\mathfrak{d}) $ and consider the free $ H $-module rank $1$, $ He $, equipped with the pseudobracket given by 
\begin{equation}\label{e1.6}
[e \ast e] = (r + s \otimes 1 - 1 \otimes s) \otimes_H e,
\end{equation}
and extended to $ He $ by bilinearity. This is a simple Lie pseudoalgebra, denoted by $ H(\mathfrak{d}, \chi, \omega). $
There is a unique pseudoalgebra embedding of $ H (\fd, \chi, \omega) $ in $ W (\fd) = H \otimes \fd, $ defined by \cite{BDK}
\[ e \mapsto -r + 1 \otimes s. \]
We will denote again by $ e $ its image in $ W(\fd). $

Let $\spd$ be the symplectic subalgebra of the Lie algebra $\gld$, defined by the skewsymmetric bilinear form $ \omega. $ Let $ \{e^i_j\} \subset \End \fd $ be the basis of matrix units in the basis
$ \{\ptl_i\} $ of $ \fd,  $ and let $ e^{ij} = \sum_k r^{ik} e^j_k $ be another basis of $ \End \fd $, where the $ r^{ik} $ are defined by \eqref{e1.5}.
Then the elements 
\[ f^{ij} = -\frac{1}{2} (e^{ij} + e^{ji}), \qquad 1 \leq i \leq j \leq 2N, \]
form a basis of $\spd.$

It is important to highlight that repeating the same strategy as in primitive Lie pseudoalgebras of type $W, S, K$ is not straightforward, as in type $H$ there are several issues that force novel and more irregular behaviour.
\begin{itemize}
\item The annihilation algebra $\P$ of $\Hd$ is a graded Lie algebra with a nontrivial one-dimensional center. The above embedding of $\Hd$ inside $\Wd$ does not induce a corresponding embedding of annihilation algebras, since central elements from $\P$ lie in its kernel; as a consequence, a complete family of tensor modules for $\Hd$ cannot be obtained by restriction from tensor modules for $\Wd$, as one will only obtain modules where central elements act trivially.
\item Constructing tensor modules as induced modules does not yield, as with primitive pseudoalgebras of other types, a corresponding grading, but only a filtration. Indeed, even though the annihilation algebra is graded , central elements lie in degree $-2$, whereas Schur's Lemma forces them to act via scalar multiplication which, if homogeneous, should have degree $0$.
\item The proof of existence of a unique maximal submodule in degenerate tensor modules fails, as it relies on a grading. In principle, nonconstant singular vectors (i.e., of positive degree in the filtration) may generate elements of lower degree and they indeed do so in explicit examples.
\item Finally, nonconstant singular vectors in degenerate $\Hd$-tensor modules split in multiple isotypical $\spd$-components. This hints towards a more complicated structure of the lattice of their submodules and the possibility, which indeed occurs, that an irreducible quotient of each tensor module contains more non-isomorphic $\spd$-summands of singular vectors.
\end{itemize}

All issues, but the last one, disappear if one only focuses on $\Hd$-modules with a trivial action of the center, i.e., when the action of $\P \simeq P_{2N}$ factors through $\H \simeq H_{2N}$, and the usual strategy may be employed.
In such case, writing down the action of $ e $ in a tensor $ W (\fd) $-module $ \mathcal{T} (V_0), $ where $ V_0 $ is a $ \fd \oplus \gld$-module, and suitably twisting by $ \chi, $ we obtain a formula for the $ H(\fd, \chi, \omega) $-module $ \mathcal{T} (V_0) = H \otimes V_0, $ where now $ V_0 $ is a $ \dd \oplus \spd$-module. Explicitly, we have for $ v \in V_0 = \kk \otimes V_0 \subset \mathcal{T} (V_0)$:
\begin{equation}\label{e1.7}
\begin{aligned}
e \ast v = \sum_{k=1}^{2N} (\bar{\ptl}_k \otimes \ptl^k) \otimes_H v & - \sum_{k=1}^{2N} (\bar{\ptl}_k \otimes 1) \otimes_H (\ptl^k + \adsp \ptl^k ) v \\
& + \sum_{i,j = 1}^{2N} (\bar{\ptl}_i \bar{\ptl}_j \otimes 1) \otimes_H f^{ij} v
\\
\end{aligned}
\end{equation}
where $ \bar{\ptl}_i = \ptl_i - \chi (\ptl_i), $ and $ \adsp \ptl^k $ is the image of $ \operatorname{ad} \ptl^k + \ptl^k \otimes \chi $ under the projection $\gld \rightarrow \spd $ defined by $ e^{ij} \mapsto - f^{ij}. $

Theorem \ref{irrfacttens} of the present paper, analogous to those in \cite{BDK1} and \cite{BDK2}, implies that every finite irreducible $H(\fd, \chi, \omega) $-module, with a trivial action of the center of $\P$, is a quotient of the tensor module $ \mathcal{T} (V_0), $  where $ V_0 $ is a finite-dimensional irreducible $ \fd \oplus \spd$-module. 
We describe all cases where the $ H (\fd, \chi, \omega) $-modules $ \mathcal{T} (V_0) $ are not irreducible and give an explicit construction of their irreducible quotients, called {\em degenerate} $H(\fd, \chi, \omega)$-modules. 

It turns out that, in analogy with the contact case \cite{BDK2}, all of the above non-trivial degenerate $\Hd $-modules appear as images of composition of two maps (not one as in the contact case) in a certain complex of $\Hd$-modules, which we call the \emph{twisted conformally symplectic pseudo de Rham complex}. This complex is constructed in Section \ref{stdrm} by a certain reduction of the pseudo de Rham complex \eqref{e1.4} (see Theorem \ref{tmodhd} and \thref{exactderhamnontrivial}). The idea of this construction is similar to Eastwood's reduction of the de Rham complex on a conformally symplectic manifold \cite{E}. The structure of the lattice of submodules of members of the twisted conformally symplectic pseudo de Rham complex requires a detailed study of singular vectors. This is carried out in Section \ref{sksing} and \ref{ssubmtm}.

Let us now proceed to the case of $\Hd$-modules with a nontrivial action of central elements in the annihilation algebra $\P$. Since this is a 
central extension of the corresponding simple Lie-Cartan algebra $H_{2N}$, an 
important role in representation theory of the Lie pseudoalgebra 
$\Hd$ is played by the extension $\dd'=\dd+\kk c$ of the Lie algebra 
$\dd$ by a $1$-dimensional abelian ideal $\kk c$, with brackets:
$$[\d,c]'=\chi(\d) c, \qquad [\d_1,\d_2]'=[\d_1,\d_2]+\omega(\d_1 \wedge \d_2)c, \qquad \mbox{where } \d,\d_1,\d_2\in \dd.$$

One may easily check that the central elements in $\P$ may only act by a nontrivial scalar when $\chi = 0$ and $\omega = \di\zeta$ is exact. Lie algebras with a non-degenerate exact $2$-form are known as {\em Frobenius Lie algebras} (the simplest example being the non-abelian $2$-dimensional Lie algebra). The central extension $\dd'$ then splits as a direct sum $\dd^\zeta \oplus \kk c$, where $\dd^\zeta$ is a Lie subalgebra isomorphic to $\dd$ which is a complement to the central ideal $\kk c$.

In this case, the definition of tensor modules $\T(V_0)$ must be modified as follows:
\begin{equation}\label{e1.7bis}
\begin{aligned}
e \ast v = \sum_{k=1}^{2N} (\bar{\ptl}_k \otimes \ptl^k) \otimes_H v & - \sum_{k=1}^{2N} (\bar{\ptl}_k \otimes 1) \otimes_H (\ptl^k + \adsp \ptl^k ) v \\
& + \sum_{i,j = 1}^{2N} (\bar{\ptl}_i \bar{\ptl}_j \otimes 1) \otimes_H f^{ij} v + (1 \otimes 1) \otimes_H cv,
\\
\end{aligned}
\end{equation}
where $c \in \kk$ provides the scalar action of the center of $\P$. Once again Theorem \ref{irrfacttens} shows that every finite irreducible non-trivial $\Hdzeta$-module is a quotient of $\T(V_0)$, where $V_0$ is a finite-dimensional $\dd' \oplus \spd$-module. The complex of degenerate $\Hdzeta$-modules is now constructed by hand, by using the canonical maps between tensor modules whose existence is due to the presence of nonconstant singular vectors. The resulting complex is \eqref{splitexactcomplex}.
The pathological behaviour we anticipated indeed occurs. Namely:
\begin{itemize}
\item each $\VPiprimen, 1 \leq n \leq N,$ contains {\bf two} maximal submodules;
\item each $\im\DIn, 0 \leq n \leq N,$ contains {\bf two} distinct non-isomorphic $\dd'\oplus\spd$-summands of singular vectors, hence it is a quotient of two distinct degenerate tensor modules.
\end{itemize}
Somewhat surprisingly, one may show that the whole complex is split exact, and that each member of the complex decomposes into the direct sum of its maximal submodules, which are also irreducible: this is a special behaviour of tensor modules when $c \neq 0$, which completely falls apart when $c = 0$. Thus, every finite irreducible degenerate $\Hdzeta$-module, which has a nontrivial action as $c \neq 0$, arises as image of one differential in this complex. This complex still lacks a differential geometrical construction, so that its geometrical meaning is as yet unclear.

Irreducibility of tensor modules, not appearing in the above complexes, is proved in Section \ref{sirtm} (see Theorem \ref{irrcriterion}). The resulting complete non-redundant list of finite irreducible $ H(\fd, \chi, \omega) $-modules is given in Theorem \ref{tclass}. 

As a corollary of our results, we obtain a classification and description of all degenerate irreducible modules over the Hamiltonian Lie--Cartan algebra $P_{2N}, $ along with the description of the singular vectors. This result when $c = 0$ was obtained long ago in \cite{Ru2}.


\section{Preliminaries on Lie pseudoalgebras}\label{sprel}

In this section, we review some facts and notation that will be used
throughout the paper. For a more detailed treatment, we refer to our 
previous works \cite{BDK, BDK1, BDK2}.

\subsection{Bases and filtrations of $H$ and $H^*$}\label{sbasfil}
We will denote by $H$ the universal enveloping algebra $\ue(\dd)$
of the Lie algebra $\dd$. Then $H$ is a Hopf algebra with a coproduct 
$\De$, antipode $S$, and counit $\ep$ given by:
\begin{equation}\label{des}
\De(\d) = \d\tp1 +1\tp\d \,, \quad S(\d)=-\d \,, \quad \ep(\d)=0 \,,
\qquad \d\in\dd \,.
\end{equation}
We will use the following notation (cf.\ \cite{Sw}):
\begin{align}
\label{de1}
\De(h) &= h_{(1)} \tp h_{(2)} = h_{(2)} \tp h_{(1)} \,, 
\\
\label{de2}
(\De\tp\id)\De(h) &= (\id\tp\De)\De(h) = h_{(1)} \tp h_{(2)} \tp h_{(3)} \,,
\\
\label{de3}
(S\tp\id)\De(h) &= h_{(-1)} \tp h_{(2)} \,, \qquad\quad h\in H \,.
\end{align}
Then the axioms of antipode and counit can be written
as follows:
\begin{align}
\label{antip}
h_{(-1)} h_{(2)} &= h_{(1)} h_{(-2)} = \ep(h),
\\
\label{cou}
\ep(h_{(1)}) h_{(2)} &= h_{(1)} \ep(h_{(2)}) = h,
\end{align}
while the fact that $\De$ is a homomorphism of algebras
translates as:
\begin{equation}
\label{deprod}
(fg)_{(1)} \tp (fg)_{(2)} = f_{(1)} g_{(1)} \tp f_{(2)} g_{(2)},
\qquad f,g\in H.
\end{equation}
Eqs.\ \eqref{antip}, \eqref{cou} imply the following
useful relations:
\begin{equation}
\label{cou2}
h_{(-1)} h_{(2)} \tp h_{(3)} = 1\tp h
= h_{(1)} h_{(-2)} \tp h_{(3)}.
\end{equation}
%

%

Below it will be convenient to work with a basis $\{\d_1, \dots,
\d_{2N}\}$ of $\dd$ and the dual basis $\{x^1,\dots,x^{2N}\}$ of $\dd^*$.
Denote by $c_{ij}^k$ the structure constants of $\dd$, so that 
\begin{equation}\label{cijk}
[\d_i,\d_j]=\sum_{k=1}^{2N} c_{ij}^k\d_k
\,, \qquad i,j=1,\dots,2N \,.
\end{equation}
Note that $H$ has a basis
\begin{equation}\label{dpbw}
\d^{(I)} = \d_1^{i_1} \dotsm \d_{2N}^{i_{2N}} / i_1! \dotsm
i_{2N}! \,, \qquad I = (i_1,\dots,i_{2N}) \in\ZZ_+^{2N} \,.
\end{equation}
The canonical increasing filtration of $H=\ue(\dd)$ is given by
\begin{equation}\label{filued1}
\fil^n H = \Span_\kk\{ \d^{(I)} \st |I| \le n \} \,, \qquad
\text{where} \quad |I|=i_1+\cdots+i_{2N} \,.
\end{equation}
This filtration does not depend on the choice of basis of $\dd$, and 
is compatible with the Hopf algebra structure of $H$ (see, e.g.,
\cite[Section~2.2]{BDK} for more details). We have: $\fil^{-1} H =
\{0\}$, $\fil^0 H = \kk$, and $\fil^1 H = \kk\oplus\dd$.

The dual $X=H^* := \Hom_\kk(H,\kk)$ is a commutative associative algebra.
We will identify $\dd^*$ as a subspace of $X$ by letting 
$\langle x^i, \d_i \rangle = 1$ and $\langle x^i, \d^{(I)} \rangle = 0$
for all other basis vectors \eqref{dpbw}.
This gives rise to an isomorphism from $X$ 
to the algebra $\O_{2N} = \kk[[t^1,\dots,t^{2N}]]$ 
of formal power series in $2N$ indeterminates, which sends $x^i$ to $t^i$.
The Lie algebra $\dd$ has left and right actions on $X$ by derivations, given by
\begin{align}
\label{dx1}
\langle \d x, h\rangle &= -\langle x, \d h\rangle \,,
\\
\label{dx2}
\langle x \d, h\rangle &= -\langle x, h\d\rangle \,,
\qquad \d\in\dd \,, \; x \in X \,, \; h \in H \,,
\end{align}
where $\d h$ and $h\d$ are the products in $H$. These two actions
coincide only when $\dd$ is abelian. The difference
$\d x - x \d$ gives the coadjoint action of $\d\in\dd$ on $x\in X$.

Throughout the paper, we will be given a 
\emph{trace-form} $\chi\in\dd^*$, so that
$\chi([\d,\d'])=0$ for all $\d,\d' \in\dd$.
Then the assignment 
\begin{equation}\label{dbar}
\d \mapsto \bd = \d - \chi(\d), \qquad \d\in\dd,
\end{equation}
extends to an associative algebra automorphism $h \mapsto \bar h$ of $H$.
Note that its inverse is given in the same way but with $-\chi$ in place of $\chi$.

It is convenient to define another basis of $H$ by applying the bar 
automorphism, thus obtaining elements $\bd^{(I)}$. It is
easily verified that the corresponding bar filtration on $H$ coincides
with $\{\fil^n H\}$.
The following results will be useful in the rest of the paper.

\begin{lemma}\label{checkbar}
For any trace-form\/ $\chi\in\dd^*$ and\/ $\d\in\dd$, we have\/
$\d\chi = \chi\d = -\chi(\d)$.
In particular,
\begin{align}\label{echi}
(x e^{-\chi}) \bd &= (x\d) e^{-\chi},
\\ \label{echi2}
\bd (x e^{-\chi}) &= (\d x) e^{-\chi}, \qquad x\in X, \;  \d\in\dd .
\end{align}
\end{lemma}
\begin{proof}
The first claim follows from \eqref{dx1} and \eqref{dx2}. Then \eqref{echi} and \eqref{echi2} are derived using that the right and left actions of $\d$ are derivations of $X$.
\end{proof}

\begin{lemma}\label{dualtwisted}
{\rm (i)}
The map\/ $h \mapsto S(\overline{S(h)})$ is the inverse to the automorphism\/ $h\mapsto \overline{h}$ of\/ $H$.

{\rm (ii)}
The basis\/ $\{x_I e^{-\chi}\}$ of\/ $X$ is dual to the basis\/ $\{S(\overline{S(\d^{(I)})})\}$ of\/ $H$.
\end{lemma}
\begin{proof}
(i) The map $h \mapsto S(\overline{S(h)})$ is indeed an automorphism, because $h\mapsto \overline{h}$ is an automorphism and $S$ is an anti-automorphism. It is easy to see that for $\d\in\dd$, we have $S(\overline{S(\d)}) = \d+\chi(\d)$, which is the inverse to the map \eqref{dbar}.

To prove (ii), first observe that $\langle x, 1 \rangle = \langle x e^{-\chi}, 1\rangle$ for $x \in X$, since
$\langle x, 1\rangle$ is the constant term of $x$. Then using \eqref{dx1} and \eqref{echi2}, we find
\begin{equation*}
\begin{split}
\delta_I^{J} & = \langle x_I, \d^{(J)}\rangle = \langle S(\d^{(J)})x_I, 1\rangle = \langle (S(\d^{(J)})x_I)e^{-\chi}, 1\rangle\\
& = \langle \overline{\,S(\d^{(J)})}\,(x_I e^{-\chi}), 1\rangle =\langle x_I e^{-\chi}, S(\overline{\,S(\d^{(J)})})\rangle,
\end{split}
\end{equation*}
as claimed.
\end{proof}

We introduce a decreasing filtration of $X$ by letting
$\fil_n X = (\fil^n H)^\perp$ be the set of elements from
$X$ that vanish on $\fil^n H$. Then $\fil_{-1} X = X$,
$X/\fil_0 X \simeq\kk$, and $\fil_0 X/\fil_1 X \simeq\dd^*$.
We define a topology of $X$ by considering $\{\fil_n X\}$ as a
fundamental system of neighborhoods of $0$. We will always consider
$X$ with this topology, while $H$ and $\dd$ are endowed with the
discrete topology. Then $X$ is linearly compact (see
\cite[Chapter~6]{BDK}), and both the multiplication in $X$ and the left
and right actions of $\dd$ on it are continuous.

\subsection{Lie pseudoalgebras and their modules}\label{slieps}
Recall from \cite[Chapter~3]{BDK} that a {\em pseudobracket\/} on a left $H$-module $L$
is an $H$-bilinear map
\begin{equation}\label{psprod}
  L \tp L \to (H \tp H) \tp_H L \,, \quad
  a \tp b \mapsto [a * b] \,,
\end{equation}
where we use the
comultiplication $\Delta\colon H \to H \tp H$ to define
$(H\tp H) \tp_H L$.
We extend the pseudobracket \eqref{psprod} to maps
$(H^{\tp 2} \tp_H L) \tp L \to H^{\tp 3}
\tp_H L$ and $L \tp (H^{\tp 2}
\tp_H L) \to H^{\tp 3}\tp_H L$ by letting:
\begin{align}
\label{psprod1}
[(h\tp_{H}a)*b] &= \sum_i \, (h \tp 1)\, (\Delta\tp\id)(g_i)
 \tp_H c_i \,,
\\
\label{psprod2}
[a*(h\tp_{H}b)] &= \sum_i \, (1 \tp h)\, (\id\tp\Delta)(g_i)
 \tp_H c_i \,,
\intertext{where $h\in H^{\tp 2}$, $a,b\in L$, and}
\label{psprod3}
[a*b] &= \sum_i \, g_i \tp_H c_i
\qquad\text{with \; $g_i\in H^{\tp2}$, $c_i \in L$.}
\end{align}

A {\em Lie \psalg\/} is a left $H$-module equipped with
a pseudobracket satisfying the
following skew-commutativity and Jacobi identity axioms:
\begin{align}
  \label{psss}
  [b*a] &= -(\sigma \tp_H \id)\, [a*b] \,, \\
  \label{psjac}
  [[a*b]*c] &= [a*[b*c]] - ((\sigma \tp \id)\tp_H\id)\,
  [b*[a*c]] \,.
\end{align}
Here, $\sigma\colon H \tp H \to H \tp H$ is the permutation
of factors, and the compositions $[[a*b]*c]$, $[a*[b*c]]$
are defined using \eqref{psprod1}, \eqref{psprod2}.

\begin{example}\label{ecur}
For a Lie algebra $\g$, the \emph{current} Lie pseudoalgebra $\Cur\g=H\tp\g$
has an action of $H$ by left multiplication on the first tensor factor
and a pseudobracket
\begin{equation}\label{curbr*}
[(f\tp a)*(g\tp b)]
= (f\tp g)\tp_H(1\tp [a,b]) \,,
\end{equation}
for $f,g\in H$ and $a,b\in\g$.
\end{example}

A {\em module\/} over a Lie \psalg\ $L$
is a left $H$-module $V$ together with an $H$-bilinear map
\begin{equation}\label{psprod4}
  L \tp V \to (H \tp H) \tp_H V \,, \quad
  a \tp v \mapsto a * v
\end{equation}
that satisfies ($a,b\in L$, $v\in V$):
\begin{equation}\label{psrep}
[a*b]*v = a*(b*v) - ((\si\tp\id)\tp_H\id) \, (b*(a*v)) \,.
\end{equation}
An $L$-module $V$ will be called {\em finite\/}
if it is finitely generated as an $H$-module, and is called {\em trivial} if $a*v = 0$ for all $a \in L, v \in V$, i.e., when the pseudoaction of $L$ on $V$ is trivial.
The {\em zero\/} $L$-module is the set $\{0\}$.

\begin{example}\label{ecur2}
For any module $V_0$ over a Lie algebra $\fg$, we have the $\Cur \fg$-module $V=H \otimes V_0$, with the action given by
\begin{equation}
  \label{curact}
(g \otimes b) * (h \otimes v) = (g \otimes h) \otimes_H (1\otimes  bv)\, ,
\end{equation}
for $g,h \in H$, $b \in \fg$ and $v \in V_0$.
\end{example}

Let $U$ and $V$ be two $L$-modules. A map
$\be\colon U\to V$ is a {\em homomorphism\/} of $L$-modules
if $\be$ is $H$-linear and satisfies
\begin{equation}\label{psprod6}
\bigl( (\id\tp\id)\tp_H \be \bigr) (a*u)
= a * \be(u) \,, \qquad a\in L \,, \; u\in U \,.
\end{equation}
A subspace $W\subset V$ is an {\em $L$-submodule\/} if
it is an $H$-submodule and
$L*W \subset (H\tp H)\tp_H W$,
where $L*W$ is the linear span of all elements $a*w$
with $a\in L$ and $w\in W$.
A submodule $W\subset V$ is called {\em proper\/} if $W\ne V$.
An $L$-module $V$ is {\em irreducible\/} (or {\em simple})
if it does not contain any non-zero proper $L$-submodules and
$L*V \ne \{0\}$.

\begin{remark}\label{rlmod}
{\rm(i)}
Let $V$ be a module over a Lie \psalg\ $L$ and
$W$ be an $H$-submodule of $V$. By  \cite[Lemma 2.3]{BDK2},
for each $a\in L$, $v\in V$, we can write
\begin{equation*}
a*v = \sum_{I\in\ZZ_+^{2N}} (\d^{(I)} \tp 1) \tp_H v'_I \,, \qquad
v'_I \in V \,,
\end{equation*}
where the elements $v'_I$ are uniquely determined by $a$ and $v$.
Then $W\subset V$ is an $L$-submodule iff
it has the property that all $v'_I \in W$ whenever $v \in W$.

{\rm(ii)}
Similarly, for each $a\in L$, $v\in V$, we can write uniquely
\begin{equation*}
a*v = \sum_{I\in\ZZ_+^{2N}} (1 \tp \d^{(I)}) \tp_H v''_I \,,
\qquad v''_I \in V \,,
\end{equation*}
and $W$ is an $L$-submodule iff $v''_I \in W$ whenever $v \in W$.
\end{remark}

\subsection{Twisting of representations}\label{stwrep}
Let $L$ be a Lie \psalg, and $\Pi$ be a finite-dimensional $\dd$-module.
In \cite[Section 4.2]{BDK1}, we introduced a covariant functor $T_\Pi$
from the category of finite $L$-modules to itself. 
We review it in the case when all the $L$-modules are free
as $H$-modules, which will be sufficient for our purposes.

For a finite $L$-module $V=H\tp V_0$, which is free over $H$,
we choose a $\kk$-basis $\{ v_i \}$ of $V_0$, and write the action of $L$
on $V$ in the form
\begin{equation}\label{twrep1}
a*(1\tp v_i) = \sum_j\, (f_{ij} \tp g_{ij}) \tp_H (1 \tp v_j)
\end{equation}
where $a \in L$, $f_{ij}, g_{ij} \in H$.
Then the \emph{twisting of\/ $V$ by\/ $\Pi$} is the $L$-module
$T_\Pi(V) = H\tp \Pi\tp V_0$, where $H$ acts by a left multiplication on
the first factor and
\begin{equation}\label{twrep3}
a*(1\tp u\tp v_i) =
\sum_j\, \bigl( f_{ij} \tp {g_{ij}}_{(1)} \bigr)
\tp_H \bigl( 1 \tp {g_{ij}}_{(-2)} u \tp v_j \bigr)
\end{equation}
for $a\in L$, $u\in\Pi$.
By \cite[Proposition 4.2]{BDK1}, $T_\Pi(V)$ is an $L$-module and
the action of $L$ on it is independent of the choice of basis of $V_0$.

Given another $L$-module $V'=H\tp V'_0$ and a homomorphism $\be\colon V\to V'$, 
we can write
\begin{equation}\label{twrep6}
\be(1\tp v_i) = \sum_j\, h_{ij} \tp v'_j \,,
\qquad h_{ij} \in H \,,
\end{equation}
where $\{ v'_j \}$ is a fixed $\kk$-basis of $V'_0$. 
Then we have a homomorphism of $L$-modules
$T_\Pi(\be)\colon T_\Pi(V)\to T_\Pi(V')$, defined by
\begin{equation}\label{twrep8}
T_\Pi(\be)(1\tp u\tp v_i) =
\sum_j\, {h_{ij}}_{(1)} \tp {h_{ij}}_{(-2)} u \tp v'_j \,.
\end{equation}
Moreover, $T_\Pi(\be)$ is independent of the choice of bases
\cite[Proposition 4.2]{BDK1}.

We showed in \cite[Proposition 3.1]{BDK2}
that the functor $T_\Pi$ is exact on free $H$-modules, i.e., if
$V \xrightarrow{\be} V' \xrightarrow{\be'} V''$
is a short exact sequence of finite free $H$-modules,
then the sequence
$T_\Pi(V) \xrightarrow{T_\Pi(\be)} T_\Pi(V')
\xrightarrow{T_\Pi(\be')} T_\Pi(V'')$
is exact. Another useful property of $T_\Pi$ is that
the image of $T_\Pi(\be)$ has a finite codimension in $T_\Pi(V')$,
whenever the image of $\be\colon V\to V'$ has a finite codimension.

\subsection{Annihilation algebras of Lie pseudoalgebras}\label{spsanih}
For a Lie \psalg\ $L$, we let $\L=\A(L)=X\tp_H L$,
where as before $X=H^*$.
We define a Lie bracket on $\L$ by the formula
(cf.\ \cite[Eq.~(7.2)]{BDK}):
\begin{equation}\label{alliebr}
[x\tp_H a, y\tp_H b] = \sum_i\, (x f_i)(y g_i) \tp_H c_i \,,
\quad\text{if}\quad
[a*b] = \sum_i\, (f_i\tp g_i) \tp_H c_i \,.
\end{equation}
Then $\L$ is a Lie algebra, called the {\em annihilation algebra\/}
of $L$ (see \cite[Section~7.1]{BDK}).
There is an obvious left action of $H$ on $\L$ given by
\begin{equation}\label{hactsonl}
h(x\tp_H a) = hx\tp_H a \,, \qquad h\in H\,, \; x\in X\,, \; a\in L \,;
\end{equation}
in particular, the Lie algebra $\dd$ acts on $\L$ by
derivations. The semidirect sum $\ti\L = \dd\ltimes\L$ is called
the {\em extended annihilation algebra}.

When $L$ is finite, we can define a filtration on
$\L$ as follows (see \cite[Section~7.4]{BDK} for more details).
We fix a finite-dimensional vector
subspace $L_0$ of $L$ such that $L = HL_0$, and set
\begin{equation}\label{fill}
\fil_n\L = \{ x \tp_H a \in\L \st x\in\fil_n X \,, \; a\in L_0 \}
\,, \qquad n\ge -1 \,.
\end{equation}
The subspaces $\fil_n \L$ constitute a decreasing
filtration of $\L$, satisfying
\begin{equation}\label{filbr}
[\fil_m \L, \fil_n \L] \subset \fil_{m+n-\ell} \L \,,
\qquad \dd (\fil_n \L) \subset \fil_{n-1} \L \,,
\end{equation}
where $\ell$ is an integer depending only on the choice of $L_0$.
Notice that this filtration of $\L$ depends on the choice of
$L_0$, but the induced topology does not \cite[Lemma~7.2]{BDK}.
With this topology, $\L$ becomes a linearly-compact Lie algebra (see \cite{BDK}).
We set $\L_n = \fil_{n+\ell} \L$, so that $[\L_m, \L_n] \subset \L_{m+n}$.
In particular, $\L_0$ is a Lie algebra.

We also define a decreasing filtration of $\ti\L$ by letting
$\fil_{-1}\ti\L=\ti\L$, $\fil_n\ti\L = \fil_n\L$ for $n\ge0$,
and we set $\ti\L_n = \fil_{n+\ell} \ti\L$.
An $\ti\L$-module $V$ is called {\em conformal\/}
if every $v\in V$ is annihilated by some $\L_n$; in other words, if $V$ is
a topological $\ti\L$-module when endowed with the discrete topology.
%
%
The next two results from \cite{BDK} play a crucial role in our
study of representations (see \cite{BDK}, Propositions 9.1 and 14.2,
and Lemma 14.4).

\begin{proposition}\label{preplal2}
Any module $V$ over the Lie \psalg\ $L$ has a natural structure of a
conformal $\ti\L$-module, given by the action of\/ $\dd$ on $V$ and by
\begin{equation}\label{axm2}
(x\tp_H a) \cdot v
= \sum_i\, \< x, S(f_i {g_i}_{(-1)}) \> \, {g_i}_{(2)} v_i \,,
\quad\text{if}\quad
a*v = \sum_i\, (f_i\tp g_i) \tp_H v_i
\end{equation}
for $a\in L$, $x\in X$, $v\in V$.
Conversely, any conformal $\ti\L$-module $V$
has a natural structure of an $L$-module, given by
\begin{equation}\label{prpl2}
a*v = \sum_{I\in\ZZ_+^N} \bigl( S(\d^{(I)}) \tp1 \bigr)\tp_H
\bigl( (x_I\tp_H a) \cdot v \bigr) \,.
\end{equation}
%
Moreover, $V$ is irreducible
as an $L$-module iff it is irreducible as an $\ti\L$-module.
\end{proposition}

\begin{remark}\label{pseudoactioncheck}
In the proof of \eqref{prpl2}, one only uses that $\{\d^{(I)}\}$ and $\{x_I\}$ are dual bases of $H$ and $H^*$. Then using \leref{dualtwisted}(ii) also gives
\begin{equation}
a*v = \sum_{I\in\ZZ_+^N} \bigl( \overline{S(\d^{(I)})} \tp1 \bigr)\tp_H
\bigl( (x_I e^{-\chi}\tp_H a) \cdot v \bigr) \,.
\end{equation}
\end{remark}

\begin{lemma}\label{lkey2}
Let $L$ be a finite Lie \psalg\
and $V$ be a finite $L$-module. For $n\ge -1-\ell$, let
\begin{equation*}\label{kernv}
\ker_n V
= \{ v \in V \st \L_n \, v = 0 \},
\end{equation*}
so that, for example, $\ker_{-1-\ell} V = \ker V$
and\/ $V = \bigcup \ker_n V$.
Then all vector spaces\/
$\ker_n V / \ker V$ are finite dimensional.
In particular, if\/ $\ker V=\{0\}$, then
every vector $v\in V$ is contained in a finite-dimensional subspace
invariant under~$\L_0$.
\end{lemma}

\subsection{$\Wd$ and its annihilation algebra}\label{subwd}

One of the most important Lie \psalgs\ is $\Wd=H\tp\dd$, 
with the Lie pseudobracket (see \cite[Section 8.1]{BDK}):
\begin{equation}\label{wdbr*}
\begin{split}
[(f\tp a)*(g\tp b)]
&= (f\tp g)\tp_H(1\tp [a,b])
\\
&- (f\tp ga)\tp_H(1\tp b) + (fb\tp g)\tp_H(1\tp a) \,,
\end{split}
\end{equation}
for $f,g\in H$, $a,b\in\dd$.
The formula
\begin{equation}\label{wdac*}
(f\tp a)*h = -(f\tp ha)\tp_H 1
\end{equation}
defines the structure of a $\Wd$-module on $H$.

We denote the annihilation algebra of $\Wd$ by 
\begin{equation}\label{annihw}
\W = \A(\Wd) = X \tp_H (H \tp \dd) \simeq X \tp \dd \,.
\end{equation}
The Lie bracket in $\W$ is given by ($x,y\in X$, $a,b\in\dd$):
\begin{equation}\label{Wbra}
[x \tp a, y \tp b] = xy \tp [a, b] - x(ya) \tp b + (xb)y \tp a \,.
\end{equation}
The extended annihilation algebra of $\Wd$ is $\ti\W=\dd\ltimes\W$, where
\begin{equation}\label{dactw}
[\d, x \tp a] = \d x \tp a \,, \qquad \d,a\in\dd, \; x\in X \,.
\end{equation}
The Lie algebra $\W$ has a decreasing filtration
\begin{equation}\label{wp}
\W_n = \fil_n \W = \fil_n X\tp\dd \,,
\qquad n\geq-1 \,,
\end{equation}
satisfying $\W_{-1}=\W$ and $[\W_i, \W_j] \subset \W_{i+j}$. 
Note that $\W/\W_0 \simeq \kk\tp\dd \simeq \dd$ and 
$\W_0/\W_1 \simeq \dd^*\tp\dd$.

\begin{lemma}[\cite{BDK1}]\label{cwbra}
The map from\/ $\W_0/\W_1$ to\/ $\dd\tp\dd^*\simeq\gld$, defined by
\begin{equation*}
x\tp a \mod \W_1 \mapsto -a \tp (x \mmod \fil_1 X),
\qquad x\in \fil_0 X, \;\; a\in\dd,
\end{equation*}
is a Lie algebra isomorphism.
Under this isomorphism, the adjoint action of\/ $\W_0/\W_1$ on
$\W/\W_0$ coincides with the standard action of\/
$\gld$ on~$\dd$.
\end{lemma}

The action \eqref{wdac*} of $\Wd$ on $H$ induces a corresponding action of the
annihilation algebra $\W$ on $X$ given by
\begin{equation}\label{xay}
(x\tp a) y = -x (ya), \qquad x,y \in X, \; a\in \dd.
\end{equation}
Since $\dd$ acts on $X$ by continuous derivations, the Lie algebra
$\W$ acts on $X$ by continuous derivations.  The isomorphism $X
\simeq \O_{2N}$ induces a Lie algebra
homomorphism $\ph$ from $\W$ to $W_{2N} = \Der\O_{2N}$, the Lie algebra
of continuous derivations of the algebra $\O_{2N} = \kk[[t^1,\dots,t^{2N}]]$.
In fact, $\ph$ is an isomorphism
compatible with the filtrations (see \cite[Proposition~3.1]{BDK1}).
We recall that the canonical filtration of the Lie--Cartan algebra $W_{2N}$ is given explicitly by
\begin{equation}\label{filpwn2}
\fil_p W_{2N} = \Bigl\{ \sum_{i=1}^{2N} f_i \frac\d{\d t^i}
\; \Big| \; f_i \in \fil_p \O_{2N} \Bigr\} \,,
\qquad p\ge -1\,,
\end{equation}
where $\fil_p \O_{2N}$ is the $(p+1)$-st power of the maximal
ideal $(t^1,\dots,t^{2N})$ of $\O_{2N}$.

It is well known that all continuous derivations of the Lie algebra $W_{2N}$ are inner (see e.g.\ \cite[Proposition 6.4(i)]{BDK}).
Hence, the same is true for $\W$. Since every non-zero $\d\in\dd$ acts as a non-zero continuous derivation of $\W$ by \eqref{dactw},
we obtain an injective Lie algebra homomorphism $\ga\colon\dd\injto\W$ such that the elements
\begin{equation}\label{d-tilde}
\ti\d := \d-\ga(\d) \in \ti\W \qquad (\d\in\dd)
\end{equation}
centralize $\W$. The set $\ti\dd$ of all $\ti\d$ is a subalgebra of $\ti\W$, which is isomorphic to $\dd$ under the map $\d\mapsto\ti\d$
(see \cite[Proposition 3.2]{BDK1}). Moreover, by \cite[Lemma 3.3]{BDK1},
\begin{equation}\label{tilded}
\widetilde \d = \d + 1 \tp \d - \ad \d \mod \W_1, \qquad \d \in\dd \,,
\end{equation}
where $\ad \d$ is understood as an element of $\gl\,\dd \simeq
\W_0/\W_1$ via \leref{cwbra}.

\section{Primitive Lie pseudoalgebras of type $H$}
\label{sprim}
In this section, we introduce the main objects of our study: the Lie
pseudoalgebra $\Hd$ and its annihilation Lie algebra (see
\cite[Chapter~8]{BDK}). We also review the unique embedding of
$\Hd$ into the Lie pseudoalgebra $\Wd$ and the induced homomorphism 
of annihilation algebras.

\subsection{Symplectic Lie algebra}\label{ssymlie}
Let $\om\in\dd^*\wedge\dd^*$ be a nondegenerate skew-symmetric $2$-form, 
so that 
\begin{equation}\label{cont1}
\underbrace{\om\wedge\dots\wedge\om}_{N} \ne 0 \,, \qquad
\dim\dd=2N \,.
\end{equation}
We set $\om_{ij} = \om(\d_i \wedge \d_j)$
and denote by  $(r^{ij})$ the inverse matrix to $(\om_{ij})$:
\begin{equation}\label{cont5}
\sum_{k=1}^{2N} r^{ik} \om_{kj} = \de^i_j \,, \qquad i,j=1,\dots,2N \,.
\end{equation}
We identify $\End\dd$ with $\dd \tp \dd^*$ so that the elementary 
matrix $e_i^j \in\End\dd$ is identified with the element 
$\d_i \tp x^j \in \dd \tp \dd^*$, where $e_i^j (\d_k) = \de^j_k \d_i$. 
Notice that $(\d \tp x) (\d') = \langle x, \d'\rangle \d$, 
and the composition $(\d \tp
x) \circ (\d' \tp x')$ equals $\langle x, \d'\rangle \d \tp x'$. 
We will raise indices using the matrix $(r^{ij})$ and lower them
using $(\om_{ij})$. In particular,
\begin{equation}\label{eij}
e^{ij} = \d^i \tp x^j = \sum_{k=1}^{2N} r^{ik} e_k^j \,,
\end{equation}
where
\begin{equation}\label{cont7}
\d^i = \sum_{k=1}^{2N} r^{ik} \d_k \,, \qquad
\om(\d^i \wedge \d_j) = \de^i_j \,.
\end{equation}
Conversely, we have
\begin{equation}\label{cont8}
\d_k = \sum_{j=1}^{2N} \om_{kj} \d^j \,.
\end{equation}

Denote by $\spd=\sp(\dd,\om)$ the Lie algebra of all $A\in\gld$
such that $A\cdot\om=0$, i.e.,
\begin{equation}\label{aom0}
\om(A\d_i \wedge \d_j) + \om(\d_i \wedge A\d_j) = 0 \,, \qquad
i,j=1,\dots,2N \,.
\end{equation}
The Lie algebra $\spd$ is isomorphic to $\sp_{2N}$
and, in particular, is simple.
It is easy to see that the elements
\begin{equation}\label{fij}
f^{ij} = - \frac{1}{2} (e^{ij} + e^{ji}) = f^{ji} \,,
\qquad 1\le i\le j\le 2N
\end{equation}
form a basis of $\spd$. 
We will denote by $R(\lambda)$ the irreducible $\spd$-module
with highest weight $\lambda$, and by $\pi_n$ the fundamental weights
of $\spd$. For example, $R(\pi_1)\simeq\dd$ is the vector representation.
We set
$R(\pi_0) = \kk$ and $R(\pi_n) = \{0\}$ if $n<0$ or $n>N$.

\begin{example}\label{symbas}
Let us choose the basis of $\dd$ to be \emph{symplectic}, i.e.,
\begin{equation}\label{cont9}
\om(\d_i \wedge \d_{i+N}) = 1 = - \om(\d_{i+N} \wedge \d_i) \,,
\quad \om(\d_i \wedge \d_j) = 0 \quad\text{for}\quad |i-j| \ne N \,.
\end{equation}
Then we have:
\begin{equation}\label{cont10}
\d^i = -\d_{i+N} \,, \quad \d^{i+N} = \d_i 
\,, \qquad 1\le i \le N \,,
\end{equation}
which implies
\begin{equation}\label{cont11}
e^{ij} = -e_{i+N}^j  \,, \quad e^{i+N,j} = e_i^j 
\,, \qquad 1\le i \le N \,, \; 1\le j \le 2N \,.
\end{equation}
\end{example}

\subsection{Definition of $\Hd$}\label{subprim}

Let again $\chi\in\dd^*$ be a trace form and 
$\om\in\dd^*\wedge\dd^*$ be a nondegenerate skew-symmetric
$2$-form on $\dd$. From now on we will assume that $\om$ and $\chi$ satisfy the equation
\begin{equation}\label{omchi3}
\bigl( \om([a_1,a_2] \wedge a_3) - \chi(a_1)\,\om(a_2 \wedge a_3) \bigr)
+\text{cyclic} = 0 \,,
\qquad a_1,a_2,a_3 \in\dd \,,
\end{equation}
where ``cyclic'' here and further means applying the two non-trivial
cyclic permutations to the indices $1,2,3$.
We fix a basis $\{\d_1,\dots,\d_{2N}\}$ of $\dd$ as in \seref{sbasfil},
and let
\begin{equation}\label{cont6}
r = \sum_{i,j=1}^{2N} r^{ij} \d_i\tp\d_j
= \sum_{i=1}^{2N} \d_i\tp\d^i
= -\sum_{i=1}^{2N} \d^i\tp\d_i
\end{equation}
(cf.\ \eqref{cont7}).
Notice that $r$ is skew-symmetric and independent of the choice of basis.
We define $s\in\dd$ by the property 
\begin{equation}\label{chiom}
\chi(a)=(\iota_s\om)(a)=\om(s\wedge a) \,, \qquad a\in\dd \,.
\end{equation}

Then \eqref{omchi3} and $\chi([\dd,\dd])=0$ are
equivalent to the following system of equations for $r$ and $s$
(see \cite[Lemma~8.5]{BDK}):
\begin{align}
\label{cybe3}
[r,\De(s)] &= 0 \,,
\\
\label{cybe4}
([r_{12}, r_{13}] + r_{12} s_3) + \text{{\rm{cyclic}}} &= 0 \,,
\end{align}
where we use the standard notation $r_{12}=r\tp1$, $s_3=1\tp1\tp s$, etc.
It follows from \cite[Lemma 8.7]{BDK} that
\begin{equation}\label{hdps1}
[e * e] = (r + s \tp 1 - 1 \tp s)\tp_H e
\end{equation}
extends to a Lie pseudoalgebra bracket on $He$. 
The obtained Lie \psalg\ is denoted $\Hd$.
There is an injective homomorphism of Lie \psalgs\ 
\begin{equation}\label{kd2}
\io\colon \Hd\to\Wd \,, \qquad e \mapsto -r + 1 \tp s \,,
\end{equation}
where $\Wd=H\tp\dd$ is from \seref{subwd} (see \cite[Lemma 8.3]{BDK}). 
Moreover, this is the unique non-trivial homomorphism from $\Hd$ to $\Wd$,
by \cite[Theorem 13.7]{BDK}. 
{}From now on, we will often identify $\Hd$ with its image in $\Wd$. 

Note that, by \eqref{cont7} and \eqref{chiom}, we have
\begin{equation}\label{schi}
s = \sum_{i=1}^{2N} \chi(\d_i) \d^i = -\sum_{i=1}^{2N} \chi(\d^i) \d_i \,.
\end{equation}
This allows us to rewrite \eqref{hdps1} in the form
\begin{equation}\label{bracket}
[e * e] = \sum_{i,j=1}^{2N} r^{ij} (\bd_i \tp \bd_j) \tp_H e 
= \sum_{i=1}^{2N} (\bd_i \tp \bd^i) \tp_H e \,,
\end{equation}
where $\bd$ is given by \eqref{dbar}.
Thus,
\begin{equation}\label{iota}
\io(e) = -r + 1 \tp s = - \sum_{i=1}^{2N} \bd_i \tp \d^i = \sum_{i=1}^{2N} \bd^i \tp \d_i \,.
\end{equation}
Observe also that $\chi(s)=\om(s\wedge s)=0$ and so $\bar s=s$.

Let us consider the element
\begin{equation}\label{delx}
\rho = \frac{1}{2} \sum_{i,j=1}^{2N} r^{ij} [\d_i, \d_j]
= \frac{1}{2} \sum_{i=1}^{2N} \, [\d_i, \d^i]
\in\dd
\end{equation}
and the linear function
\begin{equation}\label{phirho}
\phi = -\chi + \io_\rho\om = \io_{\rho-s}\om \in\dd^* \,.
\end{equation}
Here $\io_\rho\om \in\dd^*$ is defined by $(\io_\rho\om)(a)=\om(\rho\wedge a)$ for $a\in\dd$,
and we used that $\chi=\io_s\om$. 

\begin{remark}\label{rhins} 
By \cite[Remark 8.5]{BDK},
the above embedding \eqref{kd2} realizes $\Hd$ as a subalgebra of the Lie pseudoalgebra
$S(\dd, \phi) \subset\Wd$.
Note that for $N=1$, we have $\Hd=S(\dd, \phi)=S(\dd,-\chi + \tr\ad)$;
see \cite[Example 8.1]{BDK}.
\end{remark}

In order to explain why $\phi$ is a trace form on $\dd$, we will find another expression for it, which will be useful in the sequel.

\begin{lemma}\label{lemphi}
With the above notation \eqref{delx}, \eqref{phirho}, we have\/
$\phi = -N\chi+\tr\ad$. In particular, $\phi$ is a trace form on\/ $\dd$.
\end{lemma}
\begin{proof}
First observe that, by \eqref{cont7}, for any linear operator $A$ on $\dd$, we have
\begin{equation*}
\tr A = \sum_{i=1}^{2N} \om(A\d^i \wedge \d_i) = -\sum_{i=1}^{2N} \om(A\d_i \wedge \d^i) \,.
\end{equation*}
Then, using \eqref{cont7}, \eqref{omchi3}, \eqref{chiom} and \eqref{schi}, we find for $a\in\dd$:
\begin{align*}
2 \om(\rho\wedge a) &= \sum_{i=1}^{2N} \om([\d_i,\d^i] \wedge a) \\
&= -\sum_{i=1}^{2N} \Bigl( \om([\d^i,a] \wedge\d_i) + \om([a,\d_i] \wedge\d^i) \Bigr) \\
&\quad+ \sum_{i=1}^{2N} \Bigl( \chi(\d_i) \om(\d^i \wedge a) 
+ \chi(\d^i) \om(a \wedge \d_i) + \chi(a) \om(\d_i \wedge \d^i) \Bigr) \\
&= 2 \tr\ad a + 2 \om(s\wedge a) - 2N \chi(a) \\
&= 2 \tr\ad a + (2 - 2N) \chi(a) \,.
\end{align*}
Therefore, $\io_\rho\om = (1-N)\chi + \tr \ad$, as claimed.
\end{proof}

\subsection{Annihilation algebra of $\Hd$}\label{subans}

We will denote by $\P$ the annihilation algebra 
of the Lie \psalg\ $\Hd$ (see \seref{spsanih} and \cite[Section~7.1]{BDK}).
By definition, we have 
\begin{equation}\label{pahd}
\P = \A(\Hd) = X \tp_H \Hd = X \tp_H He \,, 
\end{equation}
which sometimes will be identified with $X$ via the map $x \tp_H he \mapsto xh$. 
The Lie bracket on $\P$ is given by (cf.\ \eqref{alliebr}, \eqref{bracket}):
\begin{equation}\label{brap}
[x,y] = \sum_{i,j=1}^{2N} r^{ij} (x\bd_i)(y\bd_j)
= \sum_{i=1}^{2N} (x\bd_i)(y\bd^i) 
\,, \qquad x,y \in X \,.
\end{equation}

We define a decreasing filtration on $\P$ by
\begin{equation}\label{Pfiltration}
\P_n = \fil_n \P = \fil_{n+1} X \tp_H e \simeq \fil_{n+1} X \,, \qquad n\ge -2 \,,
\end{equation}
which is obtained as in \seref{spsanih} by choosing $L_0 = \kk e$. 
The canonical
injection $\iota$ of the subalgebra $\Hd$ in $\Wd$ induces a Lie
algebra homomorphism $\iota_*\colon \P \to \W$, which is 
given explicitly by (cf.\ \eqref{iota}):
\begin{equation}\label{iota2}
\io_*(x) = - \sum_{i=1}^{2N} x\bd_i \tp \d^i = \sum_{i=1}^{2N} x\bd^i \tp \d_i \,,
\qquad x\in X \simeq \P \,.
\end{equation}
We will also denote by $\iota_*$ the corresponding Lie algebra homomorphism of extended annihilation algebras $\widetilde \P \to \widetilde \W$, defined by \eqref{iota2} and by $\iota_*(\d)=\d$ for $\d\in\dd$.

The map $\io_*$ is not
injective, contrary to what happens with primitive Lie
pseudoalgebras of all other types.

\begin{lemma}\label{lkerio}
The map\/
$\iota_*\colon \P \to \W$ has a\/ $1$-dimensional kernel, which is spanned over\/ $\kk$ by\/
$\ce \equiv e^{-\chi} \tp_H e$ and is contained in the center of\/ $\P$.
\end{lemma}
\begin{proof}
Using \eqref{iota2} and \eqref{echi},
we obtain $\iota_* ( x e^{-\chi} ) = \sum_i
(x\d^i) e^{-\chi} \tp \d_i$. This is zero if
and only if $x\d^i = 0$ for all $i$, which only happens when $x$
lies in $\kk \subset X$. The fact that $e^{-\chi}$ is central
in $\P$ follows immediately from \eqref{brap} and \eqref{echi}.
\end{proof}

From now on, we will denote by $\H:= \iota_* (\P)$ the image of $\P$ in $\W$.
It has two filtrations, induced by the filtrations of $\P$ and of $\W$, which coincide due to the next lemma.

\begin{lemma}\label{lfilp}
The filtrations of\/ $\P$ and\/ $\W$ are compatible, i.e., 
$$
\H_n:=\iota_*(\P_n) = \iota_*(\P) \cap \W_n \,, \qquad n\ge-2 \,.
$$
Moreover, $[\P_m, \P_n] \subset \P_{m+n}$
for all\/ $m,n\in\ZZ$.
\end{lemma}
\begin{proof}
This follows from \eqref{wp}, \eqref{brap}--\eqref{iota2} and the facts that
$(\fil_{n+1} X) \dd \subset \fil_{n} X$ and
$(\fil_{m} X) (\fil_{n} X) \subset \fil_{m+n+1} X$.
\end{proof}

By \leref{lkerio}, there is a central extension of Lie algebras
\begin{equation}
0 \to \kk\ce \to \P \xrightarrow{\iota_*} \H \to 0 \,.
\end{equation}
Composing the isomorphism $\ph\colon\W \to W_{2N}$ with the homomorphism
$\iota_*\colon \P \to \W$, one obtains a Lie algebra homomorphism $\P\to W_{2N}$
with kernel $\kk\ce$.
Its image does not necessarily coincide with the Lie--Cartan algebra $H_{2N} \subset W_{2N}$
of Hamiltonian vector fields, but it does up to a change of variables, as we will see below.

Let us review the definition of $H_{2N}$ and its irreducible central extension $P_{2N}$.
We have $P_{2N}=\O_{2N}$ with the Lie bracket
\begin{equation}\label{brap2}
[f,g] = \sum_{i=1}^{N} 
\frac{\partial f}{\partial t^i} \frac{\partial g}{\partial t^{N+i}} 
-\frac{\partial f}{\partial t^{N+i}} \frac{\partial g}{\partial t^i} \,,
\qquad f,g\in\O_{2N}
\end{equation}
(which is known as the Poisson bracket).
There is a homomorphism $P_{2N}\to W_{2N}$ given by
\begin{equation}\label{brap3}
f \mapsto \sum^N_{i=1} 
\frac{\partial f}{\partial t^i} \frac{\partial}{\partial t^{N+i}} 
- \frac{\partial f}{\partial t^{N+i}} \frac{\partial}{\partial t^i} \,,
\end{equation}
whose image is $H_{2N}$ and whose kernel is $\kk$.
Note that $H_{2N}$ consists of all vector fields from $W_{2N}$ annihilating the standard symplectic form
$\sum^N_{i=1} d t^i \wedge d t^{N+i}$.

\begin{proposition}\label{pioh}
There exists a ring automorphism\/ $\psi$ of\/ $\O_{2N}$, which induces a
 Lie algebra automorphism\/ $\psi$ of\/ $W_{2N}$, such that the image of\/ $\H$ in\/ $W_{2N}$
 under the isomorphism\/ $\ph\colon\W \to W_{2N}$
coincides with\/ $\psi(H_{2N})$. Furthermore, $\psi$ is compatible with the filtrations, and it induces the identity map on the associated graded algebras, i.e.,
$$
(\psi-\id) \fil_p W_{2N} \subset \fil_{p+1} W_{2N} \,, \qquad p\ge-1 \,.
$$
\end{proposition}
\begin{proof}
The same as that of \cite[Proposition 3.6]{BDK1} and \cite[Proposition 4.1]{BDK2}.
\end{proof}

As a consequence of \prref{pioh}, $\H\simeq H_{2N}$ and $\P\simeq P_{2N}$. 
In the following, we will need certain explicit elements of $\H$.

\begin{lemma}\label{lptow1}
The Lie algebra homomorphism\/ $\iota_*\colon \P \to \W$ identifies the following elements{\rm:}

{\rm(i)} 
$e^{-\chi} \tp_H e \mapsto 0$,

{\rm(ii)} 
$x^k e^{-\chi} \tp_H e 
\mapsto e^{-\chi} \tp \d^k - \displaystyle\sum_{1\le i<j\le 2N} c_{ij}^k x^j e^{-\chi} \tp\d^i \mod \W_1$

{\rm(iii)} 
$x^i x^j e^{-\chi} \tp_H e \mapsto 2 f^{ij} \mod \W_1$.
\end{lemma}
\begin{proof}
The proof is straightforward, using \eqref{iota2}. Note that
$f^{ij} \in \gld$ are understood via the identification $\gld \simeq
\W_0/\W_1$ (see \leref{cwbra}).
\end{proof}

\begin{corollary}
We have a Lie algebra isomorphism\/ $\H_0/\H_1 \simeq \spd$.
\end{corollary}
\begin{proof}
The composition $\H_0 \to \W_0 \to \W_0/\W_1$ is injective
on the linear span of elements $x^i x^j \tp_H e$, which generate a
subalgebra isomorphic to $\spd$.
\end{proof}

For future use, we introduce the following linear operators on $\dd$ $(1\le k\le 2N)$:
\begin{equation}\label{adsp1}
\adsp\d^k := \ad\d^k + \d^k \tp \chi + \frac12 \sum_{i,j=1}^{2N} c_{ij}^k e^{ij} - \frac12 \chi(\d^k) I \,,
\end{equation}
where $I$ denotes the identity operator and as before we identify $\dd\otimes\dd^*\simeq\End\dd$.
Recall also that $e^{ij}$ are given by \eqref{eij}, $f^{ij}$ by \eqref{fij},
and $c_{ij}^k$ are the structure constants of $\dd$ given by \eqref{cijk}.
The notation $\adsp$ is explained by the next lemma (note that it differs from the notation used in \cite{BDK2}).

\begin{lemma}\label{lptow2}
For every\/ $k=1,\dots,2N$, we have\/ $\adsp\d^k \in\spd$. Moreover, $\adsp\d^k$ is the image of\/
$\ad\d^k+ \d^k \tp \chi$ under the projection\/ $\pi\colon\gld\to\spd$ defined by\/ $\pi(e^{ij})=-f^{ij}$.
\end{lemma}
\begin{proof}
Denoting the operator \eqref{adsp1} by $A$, we see that
\begin{equation*}
A\d_j = [\d^k,\d_j] + \chi(\d_j) \d^k - \frac12 \chi(\d^k) \d_j + \frac12 \sum_{i=1}^{2N} c_{ij}^k \d^i \,.
\end{equation*}
Then, using \eqref{cont7} and \eqref{omchi3}, we get:
\begin{align*}
\om(A &\d_j \wedge \d_\ell) + \om(\d_j \wedge A\d_\ell) \\
&= \om( [\d^k,\d_j] \wedge \d_\ell) + \chi(\d_j) \om(\d^k \wedge \d_\ell) - \frac12 \chi(\d^k) \om(\d_j \wedge \d_\ell) + \frac12 c_{\ell j}^k \\
&+ \om(\d_j \wedge [\d^k,\d_\ell]) + \chi(\d_\ell) \om(\d_j \wedge \d^k) - \frac12 \chi(\d^k) \om(\d_j \wedge \d_\ell) - \frac12 c_{j \ell}^k \\
&= \om( [\d_\ell,\d_j] \wedge \d^k) + c_{\ell j}^k \\
&= \sum_{i=1}^{2N} c_{\ell j}^i \om( \d_i \wedge \d^k) + c_{\ell j}^k \\
&= 0 \,.
\end{align*}
Therefore, $A\in\spd$. Note that $\pi$ is indeed a projection, since $\pi(f^{ij})=f^{ij}$.
To finish the proof, we need to show that $\pi(A-\ad\d^k-\d^k \tp \chi) = 0$.
By \eqref{cont8}, we have
\begin{equation*}
I = \sum_{i=1}^{2N} \d_i \tp x^i = \sum_{i,j=1}^{2N} \om_{ij} e^{ji} \,,
\end{equation*}
which implies $\pi(I)=0$, because $\pi(e^{ij})=\pi(e^{ji})$ while $\om_{ij} = -\om_{ji}$.
For the same reason,
\begin{equation*}
\pi\Bigl( \sum_{i,j=1}^{2N} c_{ij}^k e^{ij} \Bigr) = 0 \,,
\end{equation*}
thus completing the proof.
\end{proof}

Extending \eqref{adsp1} by linearity, we obtain a linear map $\adsp\colon\dd\to\spd$,
which does not depend on the choice of basis.

\subsection{The normalizers $\N_{\H}$ and $\N_{\P}$}\label{snhp}

Recall the filtration \eqref{Pfiltration} of the annihilation algebra $\P$ of the Lie pseudoalgebra $\Hd$. Our next goal is to compute the normalizer of $\P_0$ inside the extended annihilation algebra $\widetilde \P$. The first step will be to find the normalizer of $\H_0 =\iota_* (\P_0)$ in $\widetilde\H = \dd\ltimes\H \subset \ti\W$.

It is well known that every continuous derivation of the Lie algebra $H_{2N}$ 
is the sum of an inner derivation and a scalar multiple of $\ad E$, where
\begin{equation}\label{euler}
E := \sum_{i=1}^{2N} t^i \frac{\d}{\d t^i}\in \fil_0 W_{2N}
\end{equation}
is the Euler vector field
(see e.g.\ \cite[Proposition 6.4(i)]{BDK}).
Due to \prref{pioh}, the normalizer of $\H$ in $\W$ coincides with $\H + \kk \E$, where
\begin{equation}\label{derhe}
\E = \ph^{-1}\psi(E) \in\W_0 \,.
\end{equation}
Moreover, $\psi(E) = E \mod \fil_1 W_{2N}$, and by \cite[(3.16)]{BDK1},
\begin{equation}\label{derhe2}
\ph^{-1}(E) = - \sum_{i=1}^{2N} x^i \tp \d_i \mod \W_1 \,.
\end{equation}
Therefore, the image of $\E$ in $\W_0/\W_1 \simeq \gld$ is the identity operator $I$ (see \leref{cwbra}).

Every non-zero $\d\in\dd$ acts as a non-zero continuous derivation of $\H$; hence, we get
an injective Lie algebra homomorphism $\ga\colon\dd\injto\H + \kk \E$ such that 
$\ti\d = \d-\ga(\d) \in\ti\W$ centralizes $\H$. Since the centralizer of $\H$ in $\W$ is trivial, 
the elements $\ti\d$ coincide with those constructed from the Lie algebra $\W$ (see \eqref{d-tilde}).
As before, the map $\d\mapsto\ti\d$ is a Lie algebra isomorphism from $\dd$ onto a subalgebra
$\ti\dd$ of $\ti\W$. If we add a suitable scalar multiple of $\E$ to $\ti\d$, we will obtain a unique $\widehat \d \in \widetilde \H$ 
such that $\widehat \d - \d \in \H$ and the adjoint action of $\widehat\d$ on $\H$ is a scalar multiple of $\E$. 
The next lemma will allow us to find $\widehat \d$ explicitly.

\begin{lemma}\label{ldtilde}
For all\/ $k=1,\dots,2N$, we have{\rm:}
\begin{equation*}
\widetilde\d^k = \d^k - \frac{1}{2}\chi(\d^k) \E + \iota_*(x^k e^{-\chi} \otimes_H e) 
- \adsp \d^k + \sum_{1 \leq i<j\leq 2N} c_{ij}^k f^{ij} \mod \W_1 \,,
\end{equation*}
where we use the notation \eqref{cijk}, \eqref{fij} and \eqref{adsp1}.
\end{lemma}
\begin{proof}
Start with \eqref{tilded}:
$$
\widetilde \d^k = \d^k + 1 \tp \d^k - \ad \d^k \mod \W_1 \,.
$$
From \leref{lptow1}, we know that
\begin{equation*}
\begin{split}
\iota_*(x^k e^{-\chi} \otimes_H e) & = e^{-\chi} \otimes \d^k - \sum_{1 \leq i < j \leq 2N} c_{ij}^k x^j e^{-\chi} \otimes \d^i \mod \W_1\\
& = 1 \otimes \d^k - \chi \otimes \d^k - \sum_{1 \leq i < j \leq 2N} c_{ij}^k x^j \otimes \d^i \mod \W_1 \,.
\end{split}
\end{equation*}
By \leref{cwbra}, $\chi \otimes \d^k \mod \W_1$ is identified with $-\d^k \otimes \chi \in \dd\tp\dd^* \simeq \gld$,
and
$x^j \otimes \d^i \mod \W_1$ with $-\d^i \tp x^j = -e^{ij} \in\gld$. 
Hence, 
$$
1 \otimes \d^k = \iota_*(x^k e^{-\chi} \otimes_H e) - \d^k \otimes \chi - \sum_{1 \leq i<j \leq 2N} c_{ij}^k e^{ij}  \mod \W_1\,,
$$
which gives
$$\widetilde\d^k = \d^k  - \ad \d^k
+ \iota_*(x^k e^{-\chi} \otimes_H e) - \d^k \otimes \chi  - \sum_{1 \leq i<j \leq 2N} c_{ij}^k e^{ij} \mod \W_1 \,.$$
Using the definition of $\adsp$ in \eqref{adsp1}, we rewrite the right-hand side as
$$\d^k + \iota_*(x^k e^{-\chi} \otimes_H e) - \adsp \d^k - \sum_{1\leq i<j \leq 2N} c_{ij}^k e^{ij} + \frac{1}{2} \sum_{i,j = 1}^{2N} c_{ij}^k e^{ij} - \frac{1}{2} \chi(\d^k)I\mod \W_1 \,.$$
The proof then follows from the identity
$$\sum_{i,j = 1}^{2N} c_{ij}^k e^{ij} = \sum_{1 \leq i<j\leq 2N} c_{ij}^k (e^{ij} - e^{ji}),$$
the definition \eqref{fij} of $f^{ij}$, and the fact that $\E\mod\W_1$ identifies with $I \in \gld$.
\end{proof}

\begin{corollary}\label{cdtilde}
The above Lie algebra homomorphism\/ $\ga\colon\dd\injto \H + \kk \E$ satisfies
$$
\ga(\d) - \frac{1}{2}\chi(\d) \E \in \H
\,, \qquad \d\in\dd \,.
$$
\end{corollary}
\begin{proof}
Using $\ga(\d)=\d-\ti\d$ and $\spd\simeq \H_0/\H_1\simeq (\H_0+\W_1)/\W_1$,
it follows from \leref{ldtilde} that
$$
\ga(\d) - \frac{1}{2}\chi(\d) \E \in \H+\W_1 \,.
$$
As the projection of $\E$ to $\W_0/\W_1 \simeq\gld$ does not belong to $\spd$,
we see that $(\H+\W_1) \cap (\H + \kk \E) = \H$.
\end{proof}

As a consequence of \coref{cdtilde},
\begin{equation}\label{dhat}
\widehat \d = \widetilde \d + \frac{1}{2}\chi(\d)\E \in \d+\H \subset \widetilde \H \,,
\qquad \d\in \dd \,.
\end{equation}
Since $\d\mapsto\ti\d$ is Lie algebra isomorphism and $\chi$ is a trace-form,
it follows that $\d \mapsto \widehat \d$ is a Lie algebra isomorphism from $\dd$
onto a subalgebra $\widehat \dd$ of $\ti\H$.
The Lie algebra $\widehat \dd$ normalizes every $\H_n$, as so does $\E$ and $[\ti\dd,\H]=\{0\}$.

\begin{remark}\label{dhatsphi}
Recall that $\Hd\subset S(\dd, \phi)$, where $\phi = -N\chi+\tr\ad$
(see \reref{rhins} and \leref{lemphi}). Our elements $\widehat\d$ coincide with those for $S(\dd, \phi)$
defined in \cite[(3.29)]{BDK1}, if we replace there $\chi$ by $\phi$ and $N$ by $2N=\dim\dd$.
\end{remark}

Now we can determine the normalizers of $\H_0$ in $\ti\H$ and of $\P_0$ in $\ti\P$.

\begin{proposition}
For\/ $n\geq 0$, the normalizer\/ $\N_\H$ of\/ $\H_n$ inside\/ $\widetilde \H$ coincides
with\/ $\widehat \dd \sd \H_0$ and is independent of\/ $n$. There
is a decomposition as a direct sum of vector spaces\/ $\wti \H = \dd \oplus \N_\H$.
\end{proposition}
\begin{proof}
Both $\H_0$ and $\widehat \dd$
normalize $\H_n$. If we denote by $\N_\H$ the subspace $\widehat \dd
+ \H_0 = \widehat\dd \sd \H_0$, then we obtain a direct sum of
vector spaces $\widetilde \H = \dd \oplus \N_\H$.
Since $\H_n \subsetneqq \d(\H_n) \subset \H_{n-1}$ for all $0 \neq \d
\in \dd$, we see that no element outside $\N_\H$ is
contained in the normalizer of $\H_n$ in $\widetilde \H$, thus
proving the statement.
\end{proof}


Recall that the Lie algebra homomorphism $\iota_*\colon\ti\P\to\ti\H$ has a kernel $\kk\ce$, which is the center of $\P$. Hence, $\iota_*^{-1}(\widehat \dd)$ is a subalgebra of $\ti\P$, as $\widehat \dd$ is a subalgebra of $\ti\H$. Since $\P_0$ is stabilized by $\iota_*^{-1}(\widehat \dd)$, 
we have the semidirect sum 
\begin{equation}\label{np}
\N_\P := \iota_*^{-1}(\N_\H) = \iota_*^{-1}(\widehat \dd) \sd \P_0 \,.
\end{equation}

\begin{proposition}
For\/ $n \geq 0$, the normalizer of\/ $\P_n$ inside\/ $\widetilde \P$ coincides with\/ $\N_\P$
and is independent of\/ $n$. There is a decomposition as a direct sum
of vector spaces\/ $\widetilde \P = \dd \oplus \N_\P$.
\end{proposition}
\begin{proof}
Since it is a Lie algebra homomorphism,
$\iota_*$ sends elements of $\widetilde \P$
normalizing $\P_n$ to elements of $\widetilde \H$ normalizing
$\H_n$. Therefore, the normalizer of $\P_n$ inside $\widetilde \P$
is contained in $\iota_*^{-1}(\N_\H) = \N_\P$.

We already know that $[\P_0, \P_n] \subset \P_n$.
Moreover, from $[\dd, \P_n] \subset \P_{n-1}$ and $[\P, \P_n]
\subset \P_{n-1}$ we deduce $[\iota_*^{-1}(\widehat \dd), \P_n] \subset \P_{n-1}$.
Since $[\widehat \dd, \H_n] \subset \H_n$, we also obtain $[\iota_*^{-1}(\widehat
\dd), \P_n] \subset \P_n + \kk\ce$. However, 
 $\ce\not\in\P_m$ for $m \geq -1$,
which implies $[\iota_*^{-1}(\widehat \dd), \P_n] \subset \P_n$ if $n\geq 0$.
Finally, 
the fact that 
$\widetilde \P = \dd \oplus \N_\P$ is clear.
\end{proof}

\begin{lemma}\label{ldprime}
The Lie algebra\/ $\iota_*^{-1}(\widehat\dd)$
is an abelian extension of\/ $\widehat\dd \simeq \dd$ by the ideal\/ $\kk\ce$ of character\/ $\chi$ with cocycle\/
$\omega$. This means that in\/ $\iota_*^{-1}(\widehat\dd){:}$
\begin{equation}\label{abext}
[\widehat \d,\ce] = \chi(\d)\ce \,, \qquad 
[\widehat \d,\widehat \d'] = \widehat{[\d,\d']} + \om(\d\wedge\d')\ce \,,
\qquad \d,\d'\in\dd \,.
\end{equation}
\end{lemma}
\begin{proof}
Let us identify elements $\widehat \d \in \widehat\dd \subset\widetilde \H$ with their unique liftings to 
$\dd \sd \P_{-1} \subset\ti\P$.
By \eqref{echi2}, we have in $\ti\P$:
\begin{equation*}
[\d, \ce] = \d(e^{-\chi}) = \chi(\d) \ce \,.
\end{equation*}
Since $\widehat \d-\d\in\H$ and $[\H,\ce]=0$, we get $[\widehat \d,\ce] = \chi(\d)\ce$.
By \eqref{dhat} and \leref{ldtilde},
\begin{equation*}
[\widehat \d^i, \widehat \d^j] \in [\d^i + x^i e^{-\chi} + \P_0, \d^j + x^j e^{-\chi} + \P_0] \,.
\end{equation*}
We are only interested in the coefficient of $\ce$ in this bracket, 
as we already know the result in $\H=\iota_* \P$. The only terms
giving contributions in $\kk\ce$ are:
\begin{equation*}
[\d^i, x^j e^{-\chi}] \,, \quad
[\d^j, x^i e^{-\chi}] \,, \quad
[x^i e^{-\chi}, x^j e^{-\chi}] \,,
\end{equation*}
and the coefficients of $\ce$ can be found by taking images in $X/\fil_0 X \simeq \kk$.
All three are easy to compute. 
For instance, using again \eqref{echi2}, we obtain
\begin{align*}
[\d^i, x^j e^{-\chi}] &= \d^i(x^j e^{-\chi}) 
= \bigl( \d^i x^j + \chi(\d^i) x^j \bigr) e^{-\chi}
\\
& = \d^i x^j \mod \fil_0 X
= -r^{ij} \mod \fil_0 X \,,
\end{align*}
by \eqref{cont7} and \eqref{dx1}.
Similarly, using \eqref{brap}, \eqref{echi} and \eqref{dx2},
\begin{align*}
[x^i e^{-\chi}, x^j e^{-\chi}] 
&= \sum_{k=1}^{2N} \bigl( (x^i e^{-\chi}) \bd_k) \bigr) \bigl( (x^j e^{-\chi})  \bd^k \bigr)
\\
&= \sum_{k=1}^{2N} (x^i \d_k) (x^j \d^k) e^{-2\chi}
\\
&= r^{ij} \mod \fil_0 X \,.
\end{align*}
Adding all contributions, we obtain
\begin{equation*}
[\widehat \d^i, \widehat \d^j] = - r^{ij} \ce
\mod \widehat \dd \,,
\end{equation*}
which gives
\begin{equation*}
[\widehat \d_i, \widehat \d_j] = \omega_{ij} \ce \mod \widehat \dd
\end{equation*}
after lowering indices.
\end{proof}

\subsection{Finite dimensional irreducible representations of $\iota_*^{-1}(\widehat\dd)$}\label{sfdird}

We will later parametrize irreducible representations of $\Hd$ (also) in terms of irreducible representations of the abelian extension $\dd' := \iota_*^{-1}(\widehat \dd)$. In this section we briefly recall what finite-dimensional irreducible representations of $\dd'$ look like.

\begin{lemma}
Let $V$ be a finite-dimensional irreducible representation of $\dd'$. Then $\ce$ acts trivially on $V$ unless $\chi = 0$ and $\omega$ is a $2$-coboundary.
\end{lemma} 
\begin{proof}
First of all, $\ce$ spans an abelian ideal of $\dd'$, hence it lies in the radical. By the Cartan-Jacobson theorem, $\ce$ must then act via scalar multiplication by some element $\lambda \in \kk$ on each finite-dimensional irreducible representation of $\dd'$.


In conclusion, the only possibility for $\ce$ to act nontrivially on $V$ happens when $\chi = 0$ and $\omega$ is a $2$-coboundary.
\end{proof}

When $\omega$ is a $2$-coboundary, we may find a unique $1$-form $\bla$ of $\dd$ such that $\di \bla = \omega$. Then elements $\widehat \d + \bla(\d)\ce, \d \in \dd,$ span a Lie subalgebra $\ddbla \subset \dd'$ isomorphic to $\dd$ which provides a splitting $\dd' = \kk \ce \oplus \ddbla$.

\begin{proposition}
Let $V$ be a finite-dimensional irreducible representation of $\dd'$. Then:
\begin{itemize}
\item either $V$ is obtained by lifting to $\dd'$ an irreducible representation of $\dd \simeq \dd'/\kk\ce$, in which case $\ce$ acts trivially on $V$;
\item or $\chi = 0, \omega = \di \bla$, so that $\dd'$ splits as a direct sum of Lie algebras $\dd' = \kk 1 \oplus \ddbla$, and $V$ is obtained from an irreducible representation of $\ddbla \simeq \dd$ by supplementing it with a (possibly nontrivial) scalar action of the center $\kk 1$.
\end{itemize}
\end{proposition}

\section{Tensor modules for $\Hd$}\label{stmhd}

In our classification of irreducible finite $\Wd$-modules \cite{BDK1}, a crucial role was played by the so-called tensor modules. In this section, by restricting them to the subalgebra $\Hd$ of $\Wd$, we obtain the notion of tensor
$\Hd$-modules.

\subsection{Tensor modules for $\Wd$}\label{stwd}

Let us review the construction of tensor modules for $\Wd$ from \cite{BDK1}.
First, given a Lie algebra $\fg$, we define the semidirect sum $W (\dd) \sd \Cur \fg$ as
a direct sum  as $H$-modules, for which $W (\dd)$ is a subalgebra
and $\Cur \fg$ is an ideal, with the following pseudobracket
between them:
\begin{equation*}
    [(f \otimes a) * (g \otimes b)]=-(f \otimes ga)\otimes_H
       (1 \otimes b) \, ,
\end{equation*}
for $f,g \in H$, $a \in \dd$, $b \in \fg$ (see \exref{ecur} and \seref{subwd}).
For any finite-dimensional $\fg$-module $V_0$, we construct a
representation of $W (\dd) \sd \Cur \fg$ in $V=H \otimes V_0$
by
%
%
\begin{equation}\label{eq:1.6}
\bigl( (f \otimes a) \oplus (g \otimes b) \bigr) * v =- (f \otimes a)
    \otimes_H v + (g \otimes 1) \otimes_H bv \,,
\end{equation}
where $f,g \in H$, $a \in \dd$, $b \in \fg$, $v \in V_0$.
For convenience, in \eqref{eq:1.6}, we identify $v\in V_0$ with its image $1\tp v \in \kk\tp V_0 \subset V$.
The action on $V_0$ is then extended to the whole $V$ by $H$-linearity.
Note that  \eqref{eq:1.6} combines the usual action of $\Cur\g$ on $V$ from \exref{ecur2} with the $\Wd$-action
on $H$ given by \eqref{wdac*}.

Next, we define an embedding of $W (\dd)$ in $W (\dd) \sd
\Cur (\dd \oplus \gl\, \dd)$ by
\begin{equation}\label{wdgld2}
    1 \otimes \partial_i \mapsto (1 \otimes \partial_i) \oplus
    \Bigl( (1 \otimes \partial_i) \oplus \bigl(1 \otimes \ad \partial_i+
      \sum_j \partial_j \otimes e^j_i \bigr)\Bigr)\,.
\end{equation}
%
%
Composing this embedding with the action \eqref{eq:1.6} of $W (\dd)
\sd \Cur \fg$ for $\fg = \dd \oplus \gl\, \dd$, we
obtain a $W (\dd)$-module $V=H \otimes V_0$ for each $(\dd \oplus
\gl\, \dd)$-module $V_0$.  This module is called a {\em tensor
$W (\dd)$-module} and is denoted $\T (V_0)$.
Explicitly, the action of $\Wd$ on $\T(V_0)$ is given by \cite[Eq. (4.30)]{BDK1}:
\begin{equation}\label{wdgcd3}
\begin{split}
(1\tp \d_i)*v &= (1 \tp 1) \tp_H (\ad\d_i)v
+ \sum_{j}\, (\d_j \tp 1) \tp_H e_i^j v
\\
&- (1 \tp \d_i) \tp_H v
+ (1 \tp 1) \tp_H (\d_i \cdot v) \,,
\end{split}
\end{equation}
for $v \in V_0$. Here we denoted by $\d_i \cdot v \in V_0$ the action of $\d_i\in\dd$ on $v\in V_0$, which should be distinguished from the element $\d_i\tp v_0\in V$.
%

If $\Pi$ is a finite-dimensional $\dd$-module
and $U$ is a finite-dimensional $\gld$-module, 
their exterior tensor product $\Pi \bt U$ is defined as
the $(\dd\oplus\gld)$-module $\Pi \tp U$, where $\dd$
acts on the first factor and $\gld$ acts on the second one.
In this case, the tensor module $\T(\Pi \bt U)$ will also be denoted as $\T(\Pi,U)$.
Notice that
\begin{equation}\label{tpiv0}
\T(\Pi,U) = T_\Pi(\T(\kk,U)) \,,
\end{equation}
where $T_\Pi$ is the twisting functor from \seref{stwrep}.
More generally, given two $\dd$-modules $\Pi$, $\overline{\Pi}$ and a $\gld$-module $U$, we have
\begin{equation}\label{tpiv1}
T_\Pi(\T(\overline{\Pi},U)) = \T(\Pi\tp\overline{\Pi},U) \,.
\end{equation}

\subsection{Restriction of tensor modules from $\Wd$ to $\Hd$}\label{subtres}
We will identify $\Hd$ as a subalgebra of $\Wd$ via the
embedding \eqref{kd2}. Then $\Hd=He$ where $e\in\Wd$ is given by
\eqref{iota}. In view of \eqref{wdgld2}, we introduce the $H$-linear map (but not a homomorphism of Lie \psalgs) 
$\tau\colon\Wd\to\Cur\gld$ given by
\begin{equation}\label{wdgld3}
\tau(h\tp\d_i) = h\tp\ad\d_i + \sum_{j=1}^{2N} h\d_j \tp e_i^j \,,
\qquad h\in H \,.
\end{equation}
Then the image of $e$ under the map \eqref{wdgld2} is
$e \oplus (e \oplus \tau(e))$, where the first $e$ is considered an element of $\Wd=H\tp\dd$, while the second $e$ is in $\Cur\dd=H\tp\dd$.

Recall also the notation $\bd=\d-\chi(\d)$ and the
linear map $\adsp\colon\dd\to\spd$ defined by \eqref{adsp1}.
Then we have the following analog of \cite[Lemma 5.1]{BDK2}.

\begin{lemma}\label{leee}
We have
\begin{equation}\label{eee1}
\begin{split}
\tau(e) &= (\id\tp\adsp)(e) + \frac12 s \tp I + \sum_{i,j=1}^{2N} \bd_i\bd_j \tp f^{ij} 
\\
&= -\sum_{k=1}^{2N} \bd_k \tp \Bigl( \adsp\d^k +\frac12 \chi(\d^k) I \Bigr) + \sum_{i,j=1}^{2N} \bd_i\bd_j \tp f^{ij} 
\,.
\end{split}
\end{equation}
\end{lemma}
\begin{proof}
First note that, by \eqref{eij}, \eqref{cont7} and \eqref{wdgld3},
\begin{equation*}
\tau(h\tp\d^i) = h\tp\ad\d^i + \sum_{j=1}^{2N} h\d_j \tp e^{ij} \,,
\qquad h\in H \,.
\end{equation*}
Applying this to \eqref{iota}, we find
\begin{align*}
\tau(e)
&= -\sum_{i=1}^{2N} \bd_i\tp\ad\d^i - \sum_{i,j=1}^{2N} \bd_i\d_j \tp e^{ij}
\\
&= -\sum_{i=1}^{2N} \bd_i\tp \Bigl( \ad\d^i + \sum_{j=1}^{2N} \chi(\d_j) e^{ij} \Bigr) 
- \sum_{i,j=1}^{2N} \bd_i\bd_j \tp e^{ij}
\\
&= -\sum_{k=1}^{2N} \bd_k\tp ( \ad\d^k + \d^k \tp \chi ) 
- \sum_{i,j=1}^{2N} \bd_i\bd_j \tp e^{ij}
\,.
\end{align*}
Now we write
\begin{align*}
-\sum_{i,j=1}^{2N} &\, \bd_i\bd_j \tp e^{ij}
= -\frac12 \sum_{i,j=1}^{2N} \bd_i\bd_j \tp e^{ij} - \frac12 \sum_{i,j=1}^{2N} \bd_j\bd_i \tp e^{ji}
\\
&= -\frac12 \sum_{i,j=1}^{2N} \bd_i\bd_j \tp e^{ij} - \frac12 \sum_{i,j=1}^{2N} \bd_i\bd_j \tp e^{ji}
+\frac12 \sum_{i,j=1}^{2N} [\bd_i,\bd_j] \tp e^{ji}
\\
&=  \sum_{i,j=1}^{2N} \bd_i\bd_j \tp f^{ij} - \frac12 \sum_{i,j=1}^{2N} [\bd_i,\bd_j] \tp e^{ij}
\,.
\end{align*}
Since the map \eqref{dbar} is a Lie algebra homomorphism, we have
\begin{equation*}
[\bd_i,\bd_j] = \sum_{k=1}^{2N} c_{ij}^k \bd_k \,.
\end{equation*}
Finally, from \eqref{schi}, we obtain $(\id\tp\chi)(e) = \bar s = s$.
Plugging all of these in the above formula for $\tau(e)$, we derive \eqref{eee1}.
\end{proof}

As a consequence, we obtain the image of $e$ under the map \eqref{wdgld2}.
Then, by \eqref{eq:1.6}, the action of $e$ on a tensor $W (\dd)$-module $\T (V_0)$ is given explicitly by:
\begin{equation}\label{eee2}
\begin{split}
e * v &= \sum_{k=1}^{2N} (\bd_k \tp \d^k) \otimes_H v 
-\sum_{k=1}^{2N} (\bd_k \tp 1) \otimes_H \Bigl( \d^k + \adsp\d^k + \frac12 \chi(\d^k) I \Bigr) \cdot v
\\
&+ \sum_{i,j=1}^{2N} (\bd_i\bd_j \tp 1) \otimes_H f^{ij} v
\,,
\end{split}
\end{equation}
for $v\in V_0 \equiv \kk\tp V_0 \subset\T (V_0)$.

\begin{remark}\label{nontrivialc}
We already know that the unique inclusion $\Hd \hookrightarrow \Wd$ induces a non-injective homomorphism of the corresponding annihilation algebras, its kernel being spanned over $\kk$ by the central element $\ce \in \P$. Consequently, every module obtained from a $\Wd$-tensor module by restricting it to $\Hd$ necessarily exhibits a trivial action of $\ce$. One may provide an alternate construction which allows for nonzero scalar actions of the centre of $\P$ as follows.

Assume that $\chi = 0, \omega = \di \bla$. Set $H' = \ue(\dd')$, $X' = (H')^*$, where $\dd'=\iota_*^{-1}(\widehat \dd)$ as before. Recall that we may construct, as in \cite[Remark 8.23]{BDK}, a Lie pseudoalgebra $\Kdp = H' e'$, whose annihilation algebra we denote by $\K$. Then the canonical projection of Lie algebras $\pi: \dd' \to \dd'/\kk c\simeq \dd$ induces corresponding algebra homomorphisms $\pi_*: H' \to H$ and $\pi^*: X \to X'$ so that the map $j: \P \to \K$ defined as
$$\P \ni x \otimes_H e \mapsto \pi^*(x) \otimes_{H'} e' \in \K$$
is an injective Lie algebra homomorphism. As both the left and the right action of $\d_0 = c = \ce$ are trivial on $\pi^*(X)$, one obtains that all Fourier coefficients $\pi^*(x) \otimes_{H'} \d_0 e'$ vanish, and that the semidirect product $\dd' \sd \P$ induced from the extended annihilation algebra $\dd' \sd \K$ projects to $\dd \sd \P$.

Due to the correspondence between representations of a Lie pseudoalgebra and conformal representations of the corresponding (extended) annihilation algebra, one may consider tensor modules for $\Kdp$, introduced in \cite{BDK2}, and restrict the action of $\K$ to the subalgebra $\P$. One thus obtains a representation $T$ of the semidirect product $\dd' \sd \P$, and it is not difficult to check that $T/\d_0 T$ then admits an action of $\dd \sd \P$, thus yielding a representation of $\Hd$ with a possibly nontrivial action of $\ce \in \P$.
\end{remark}

\subsection{Tensor modules for $\Hd$ }\label{subthd}

For any trace form $\ph$ on $\dd$, the corresponding $1$-dimensional $\dd$-module $\kk_\ph$ is spanned by an element $1_\ph$ such that $\d\cdot 1_\ph = \ph(\d)1_\ph$ for $\d\in\dd$. Similarly, for any $\dd$-module $\Pi$, we may form the tensor product $\dd$-module $\Pi_\ph := \Pi \otimes \kk_\ph$, which is isomorphic to $\Pi$ as a vector space, but is endowed with the action
\begin{equation}\label{tpiv2}
\d\cdot (u\tp 1_\ph) = \bigl( (\d+\ph(\d))\cdot u \bigr) \tp 1_\ph \,,
\qquad \d\in\dd \,, \;\; u\in\Pi \,.
\end{equation}
This observation implies that, by changing the action of $\dd$ on a $(\dd \oplus\gl\,\dd)$-module $V_0$, we can eliminate the term $\frac12 \chi(\d^k) I$ in \eqref{eee2}.
Furthermore, note that both $\adsp\d^k$ and $f^{ij}$ lie in the subalgebra $\sp\,\dd\subset\gl\,\dd$.
This motivates the following definition.

\begin{definition}\label{dtmodh}
Let $V_0$ be a finite-dimensional representation of $\dd'\oplus\spd$.
Then the \emph{tensor module} $\T(V_0,\lambda)$ over $\Hd$ is defined as
$\T(V_0,\lambda)=H\tp V_0$ with the action:
\begin{equation}\label{eee3}
\begin{split}
e * v &= \sum_{k=1}^{2N} (\bd_k \tp \d^k) \otimes_H v 
-\sum_{k=1}^{2N} (\bd_k \tp 1) \otimes_H ( \d^k + \adsp\d^k ) \cdot v + (1 \otimes 1) \otimes_H c. v
\\
&+ \sum_{i,j=1}^{2N} (\bd_i\bd_j \tp 1) \otimes_H f^{ij} v
\,, \qquad v\in V_0 \equiv \kk\tp V_0 \subset\T (V_0) \,.
\end{split}
\end{equation}
For $V_0 = \Pi'\boxtimes U$, where $\Pi'$ is a
finite-dimensional $\dd'$-module and $U$ is a finite-dimensional
$\spd$-module, we will write $\T(V_0)=\T(\Pi',U)$. When $V_0 = \Pi \boxtimes U$ is a $\dd\oplus \spd$-module, we will abuse the notation and set $\T(V_0) = \T(\Pi, U):= \T(\pi^* \Pi, U)$, where $\pi^* \Pi$ is viewed as a $\dd'$-module via composition with the canonical projection $\pi: \dd' \to \dd'/\kk c \simeq \dd$.
\end{definition}

As usual, the action of $\Hd=He$ on $\T(V_0)$ is determined uniquely from \eqref{eee3} and $H$-linearity. Notice that when the action of $c \in \dd'$ is trivial and $\Pi'$ gets identified with $\pi^* \Pi$ as above, \eqref{eee3} coincides with the expression from \eqref{eee2}, thus showing that $\T(V_0)$ is indeed an $\Hd$-module. One may in principle use \reref{nontrivialc} to generalize this fact to nontrivial actions of $c$; we will instead postpone this to \seref{skten}, where it will be derived in terms of singular vectors.
Using \eqref{eee2}--\eqref{eee3}, we obtain the next proposition.

\begin{proposition}\label{phtm}
Let\/ $\Pi$ be a finite-dimensional\/ $\dd$-module, and\/ $U$ be a finite-dimensional\/ $\gld$-module, which is also an\/ $\spd$-module by restriction. Suppose that\/ $I\in\gld$ acts as\/ $k\id$ on\/ $U$ for some\/ $k\in\kk$. Then the restriction of the tensor\/ $\Wd$-module\/ $\T(\Pi,U)$ to the subalgebra\/ $\Hd$ is isomorphic to the tensor\/ $\Hd$-module\/ $\T(\Pi_{k\chi/2},U)$.
\end{proposition}

\section{Tensor modules of de Rham type}\label{stdrm}
The main goal of this section is to define an important complex of
$\Hd$-modules, called the {\em conformally symplectic pseudo de Rham complex}. 
We start by reviewing the pseudo de Rham complex of $\Wd$-modules from \cite{BDK,BDK1}.

\subsection{Forms with constant coefficients}\label{sfcc}
%
Consider the cohomology complex of the Lie algebra $\dd$
with trivial coefficients:
\begin{equation}\label{domcc}
0 \to \Om^0 \xrightarrow{\diz}\Om^1\xrightarrow{\diz} \cdots \xrightarrow{\diz} \Om^{2N} \to 0
\,, \qquad \dim\dd=2N \,,
\end{equation}
where
$\Om^n = \textstyle\bigwedge^n\dd^*$.
Set
$\Om = \textstyle\bigwedge^\bullet\dd^*
= \textstyle\bigoplus_{n=0}^{2N} \Om^n$
and $\Om^n = \{0\}$ if $n<0$ or $n>2N$.
We will think of the elements of $\Om^n$ as skew-symmetric
\emph{$n$-forms}, i.e., linear maps from $\bigwedge^n \dd$ to $\kk$.
The \emph{differential} $\diz$ is given by the formula
($\al\in\Om^n$, $a_i\in\dd$):
\begin{equation}\label{d0al}
\begin{split}
(&\diz\al)(a_1 \wedge \dots \wedge a_{n+1})
\\
&= \sum_{i<j} (-1)^{i+j} \al([a_i,a_j] \wedge a_1 \wedge \dots \wedge
\what a_i \wedge \dots \wedge \what a_j \wedge \dots \wedge a_{n+1})
\,, \qquad n\ge1 \,,
\end{split}
\end{equation}
and $\diz\al = 0$ if $\al\in\Om^0=\kk$.
Here, as usual, a hat over a term means that it is omitted in the wedge
product.
Recall also that the \emph{wedge product} of two forms
$\al\in\Om^n$ and $\be\in\Om^p$
is defined by:
\begin{equation}\label{wedge}
\begin{split}
(\al&\wedge\be)(a_1 \wedge \dots \wedge a_{n+p})
\\
&= \frac1{n!p!} \sum_{\pi\in\symm_{n+p}} (\sgn\pi) \,
\al(a_{\pi(1)} \wedge \dots \wedge a_{\pi(n)}) \,
\be(a_{\pi(n+1)} \wedge \dots \wedge a_{\pi(n+p)}) \,,
\end{split}
\end{equation}
where $\symm_{n+p}$ denotes the symmetric group on $n+p$ letters
and $\sgn\pi$ is the sign of the permutation $\pi$.

The wedge product \eqref{wedge} makes $\Om$
a $\ZZ_+$-graded \as\ commutative superalgebra with parity $p(\al)=n \mod 2\ZZ$
for $\al\in\Om^n$:
\begin{equation}\label{omasgc}
\al\wedge\be = (-1)^{p(\al)p(\be)} \be\wedge\al \,, \qquad
(\al\wedge\be)\wedge\ga = \al\wedge(\be\wedge\ga) \,.
\end{equation}
The differential $\diz$ is an odd derivation of $\Om$\,:
\begin{equation}\label{dizder}
\diz(\al\wedge\be) = \diz\al\wedge\be + (-1)^{p(\al)} \al\wedge\diz\be \,.
\end{equation}
For $a\in\dd$, we have operators $\io_a\colon\Om^n\to\Om^{n-1}$ defined by
\begin{equation}\label{ioal}
(\io_a\al)(a_1 \wedge \dots \wedge a_{n-1})
=\al(a \wedge a_1 \wedge \dots \wedge a_{n-1}) \,,
\qquad\quad a_i\in\dd \,,
\end{equation}
which are also odd derivations of $\Om$.
For $A\in\gld$, denote by $A\cdot$ its action on $\Om$\,:
\begin{equation}\label{acdotal}
(A\cdot\al)(a_1 \wedge \dots \wedge a_n)
= \sum_{i=1}^n\, (-1)^i \al(A a_i \wedge a_1 \wedge \dots \wedge \what a_i
\wedge \dots \wedge a_n) \,.
\end{equation}
Then $A\cdot$  is an even derivation of $\Om$\,:
\begin{equation}\label{acder}
A\cdot(\al\wedge\be) = (A\cdot\al)\wedge\be + \al\wedge(A\cdot\be) \,.
\end{equation}

As before, let us fix a $1$-form $\chi\in\Om^1=\dd^*$ and a nondegenerate $2$-form $\om\in\Om^2=\dd^*\wedge\dd^*$.
Then equation \eqref{omchi3} and the condition that $\chi$ is a trace form are equivalent to:
\begin{equation}\label{omchi4}
\diz\chi=0 \,, \qquad \diz\om+\chi\wedge\om=0 \,.
\end{equation}
Consider the operator $\Psi$ of wedge multiplication with $\om$, i.e., $\Psi(\al) = \om\wedge\al$ for $\al\in\Om$.
Denote by $I^n\subset\Om^n$ the image of $\Psi|_{\Om^{n-2}}$
and by $J^n\subset\Om^n$ the kernel of $\Psi|_{\Om^n}$.
Since $\om$ is nondegenerate, we have:
\begin{equation}\label{injn}
I^k=\Om^k  \;\;\text{for}\;\; k\ge N+1 \,, \qquad  J^n=\{0\} \;\;\text{for}\;\; n\le N-1 \,.
\end{equation}
In particular, $\Psi$ gives an isomorphism $\Om^{N-1}\to\Om^{N+1}$.
More generally, $\Psi^m$ gives an isomorphism
$\Om^{N-m}\to\Om^{N+m}$ for $0\le m\le N$.
By \cite[Lemma 6.1]{BDK2}, the composition 
\begin{equation}\label{injn2}
J^{N+m} \injto \Om^{N+m} 
\xrightarrow{\,\simeq\,} \Om^{N-m}
\surjto \Om^{N-m}/I^{N-m}
\end{equation}
is an isomorphism for $0\le m\le N$.

By definition, $A\cdot\om=0$ for $A\in\spd$. Hence, by \eqref{acder}, $A\cdot$ commutes with $\Psi$.
This implies that $I^n$ and $J^n$ are $\spd$-modules and \eqref{injn2} is an isomorphism of $\spd$-modules.
In fact, one has isomorphisms of $\spd$-modules
\begin{equation}\label{injn3}
\Om^{n}/I^{n} \simeq J^{2N-n} \simeq R(\pi_n) \,, \qquad
0\leq n \leq N \,,
\end{equation}
where $R(\pi_n)$ denotes the $n$-th fundamental
representation of $\spd$ and $R(\pi_0) = \kk$
(see, e.g., \cite[Lecture 17]{FH}).

\subsection{Pseudo de Rham complex}\label{sdrm}
Recall from \cite{BDK} the spaces of \emph{pseudoforms}
\begin{equation}\label{omndd}
\Om^n(\dd) = H\tp\Om^n \,, \qquad
\Om(\dd) = H\tp\Om = \bigoplus_{n=0}^{2N} \Om^n(\dd) \,.
\end{equation}
These are free $H$-modules, where $H$ acts by left multiplication on the first factor.
We will often identify $n$-pseudoforms $\al\in\Om^n(\dd)$ with $\kk$-linear maps from
$\bigwedge^n \dd^*$ to $H$.
The \emph{differential} $\di\colon\Om^n(\dd)\to\Om^{n+1}(\dd)$ is given by
($\al\in\Om^n(\dd)$, $a_i\in\dd$):
\begin{equation}\label{dw}
\begin{split}
&\begin{split} (\di \al)&(a_1 \wedge \dots \wedge a_{n+1})
\\
&= \sum_{i<j} (-1)^{i+j} \al([a_i,a_j] \wedge a_1 \wedge \dots
\wedge \what a_i \wedge \dots \wedge \what a_j \wedge \dots \wedge
a_{n+1})
\\
&+ \sum_i (-1)^i \al(a_1 \wedge \dots \wedge \what a_i \wedge
\dots \wedge a_{n+1}) \, a_i \in H
\qquad\text{if}\;\; n\ge1,
\end{split}
\\
&(\di \al)(a_1) = -\al \, a_1 \in H \qquad\text{if}\;\; \al\in\Om^0(\dd)=H.
\end{split}
\end{equation}
Notice that $\di$ is $H$-linear and $\di^2=0$. 
The sequence
\begin{equation}\label{domd}
0 \to \Om^0(\dd) \xrightarrow{\di} \Om^1(\dd) \xrightarrow{\di}
\cdots \xrightarrow{\di} \Om^{2N-1}(\dd) \xrightarrow{\di} \Om^{2N}(\dd)
\end{equation}
is called the \emph{pseudo de Rham complex}.
By \cite[Remark 8.1]{BDK}, the
$n$-th cohomology of the complex $(\Om(\dd), \di)$ is trivial
for $n\ne 2N=\dim\dd$ and is $1$-dimensional for $n=2N$. 
Hence, the sequence \eqref{domd} is exact. 
We extend the wedge product in $\Om$ to a product in $\Om(\dd)$ by setting
\begin{equation}\label{wedge*}
(f\tp\al)\wedge(g\tp\be) = (fg) \tp (\al\wedge\be) \,, \qquad
\al,\be\in\Om \,, \;\; f,g\in H \,.
\end{equation}
Then, by \cite[Lemma 6.3]{BDK2}, we have:
\begin{align}
\label{dider}
\di(\al\wedge\be) = \diz\al\wedge\be + (-1)^n
\al\wedge\di\be \,,
\qquad \al\in\Om^n \,, \;\; \be\in\Om \,.
\end{align}

In \cite{BDK,BDK1}, we introduced a $\Wd$-module structure on each $\Om^n(\dd)$, so that the map $\di$
is a module homomorphism. Furthermore, $\Om^n(\dd)$ is a tensor $\Wd$-module,
namely $\Om^n(\dd)=\T(\kk,\Om^n)$ (see \seref{stwd}).
Explicitly, if we identify the elements of $H^{\tp2} \tp_H \Om^n(\dd) \simeq H^{\tp2} \tp\Om^n$ with $\kk$-linear maps
from $\bigwedge^n \dd^*$ to $H^{\tp2}$, the action of $\Wd$ on $\Om^n(\dd)$ is given by:
\begin{equation}\label{fa*om}
\begin{split}
\bigl((f\tp a)&*\al\bigr)(a_1 \wedge \dots \wedge a_n) 
= - \al(a_1\wedge \dots \wedge a_n) \, (f \tp a)
\\
&+ \sum_{i=1}^n (-1)^i  
\al(a \wedge a_1 \wedge\dots \wedge \what a_i \wedge \dots \wedge a_n) \, (f a_i \tp 1)
\\
&+ \sum_{i=1}^n (-1)^i  
\al([a,a_i] \wedge a_1 \wedge\dots \wedge \what a_i \wedge \dots \wedge a_n) \, (f \tp 1)
\in H^{\tp2} \,,
\end{split}
\end{equation}
for $n\ge1$, $f\tp a\in\Wd$ and $\al\in\Om^n\equiv\kk\tp\Om^n\subset\Om^n(\dd)$.
For $n=0$, the action of $\Wd$ on $\Om^0(\dd)=H$ coincides with the action \eqref{wdac*}.

\subsection{Conformally symplectic pseudo de Rham complex for $\chi=0$}\label{sdrmh0}
In this subsection, we will assume that $\chi=0$. The general case will be considered in the next subsection.
Let us extend $\Psi$ by $H$-linearity to a homomorphism of $H$-modules $\Psi\colon\Om^n(\dd)\to\Om^{n+2}(\dd)$.
Then the image of $\Psi |_{\Om^{n-2}(\dd)}$ is $I^n(\dd) = H\tp I^n \subset \Om^n(\dd)$, 
and the kernel of $\Psi |_{\Om^{n}(\dd)}$ is $J^n(\dd)  = H\tp J^n \subset \Om^n(\dd)$. 
By \eqref{injn2}, we have isomorphisms
$J^{N+m}(\dd) \simeq \Om^{N-m}(\dd) / I^{N-m}(\dd)$
for all $0\le m\le N$.

It follows from \eqref{dider} and \eqref{omchi4} with $\chi=0$ that
$\di\Psi = \Psi\di$.
Therefore, the complex \eqref{domd}
induces complexes
\begin{align}
\label{domd1}
0 \to \Om^0(\dd) / I^0(\dd) &\xrightarrow{\di} 
\Om^1(\dd) / I^1(\dd) 
\xrightarrow{\di} \cdots \xrightarrow{\di} \Om^{N}(\dd) / I^{N}(\dd)
\intertext{and}
\label{domd2}
J^{N}(\dd) &\xrightarrow{\di} J^{N+1}(\dd) 
\xrightarrow{\di} \cdots \xrightarrow{\di} J^{2N}(\dd) \to 0 \,.
\end{align}

\begin{lemma}\label{ldomd2}
The sequences \eqref{domd1} and 
\begin{equation*}
J^{N}(\dd) \xrightarrow{\di} J^{N+1}(\dd) 
\xrightarrow{\di} \cdots \xrightarrow{\di} J^{2N-2}(\dd) \xrightarrow{\di} J^{2N-1}(\dd)
\,, \qquad N\ge 3\,,
\end{equation*}
are exact. The cohomology of the complex \eqref{domd2} is finite dimensional.
\end{lemma}
\begin{proof}
To show exactness at the term $\Om^n(\dd) / I^n(\dd)$ in \eqref{domd1}
for $0\le n\le N-1$, take $\al\in\Om^n(\dd)$ such that
$\di\al\in I^{n+1}(\dd)$. This means $\di\al = \Psi\be$
for some $\be\in\Om^{n-1}(\dd)$. Then 
$0 = \di^2\al = \di\Psi\be = \Psi\di\be$\,; hence,
$\di\be=0$, because $\Psi$ is injective on $\Om^n(\dd)$ for $n\le N-1$.
By the exactness of \eqref{domd}, $\be=\di\ga$ for some $\ga\in\Om^{n-2}(\dd)$.
Then 
$\di(\al-\Psi\ga) = \Psi\be - \Psi\di\ga = 0$,
which implies $\al-\Psi\ga = \di\rho$
for some $\rho\in\Om^{n-1}(\dd)$. Therefore,
$\al+I^{n}(\dd) = \di(\rho+I^{n-1}(\dd))$ is a coboundary.

To prove exactness at the term $J^n(\dd)$ in \eqref{domd2}
for $N+1\le n\le 2N-2$, take $\al\in J^n(\dd)$ such that $\di\al=0$.
Then $\al=\di\be$ for some $\be\in\Om^{n-1}(\dd)$, by the exactness of \eqref{domd}.
We have $\di\Psi\be = \Psi\di\be = \Psi\al = 0$.
Hence, $\Psi\be = \di\ga$ for some $\ga\in\Om^{n}(\dd)$, since $\Psi\be\in\Om^{n+1}(\dd)$
with $n+1\le 2N-1$ and \eqref{domd} is exact.
Since $\Psi$ maps $\Om^{n-2}(\dd)$ onto $\Om^{n}(\dd)$ for $n\ge N+1$, 
we can find $\rho\in\Om^{n-2}(\dd)$ such that
$\Psi\rho = \ga$. Then
$\Psi(\be-\di\rho) = \Psi\be - \di\Psi\rho
= \Psi\be - \di\ga = 0$. 
We have found an element $\be-\di\rho \in J^{n-1}(\dd)$ such that
$\di(\be-\di\rho) = \al$.
This completes the proof of exactness.

The finite-dimensionality of the cohomology of \eqref{domd2} follows from the fact that
$\dim\Om^{2N}(\dd) / \di\Om^{2N-1}(\dd) = 1$.
\end{proof}

Now we will construct a map
$\diru\colon \Om^{N}(\dd) / I^{N}(\dd) \to J^{N}(\dd)$,
analogous to Rumin's map from \cite{Ru} and \cite{BDK2}.
Take any $\al\in\Om^{N}(\dd)$. Since the restriction of
$\Psi$ gives an isomorphism $\Om^{N-1}(\dd) \to \Om^{N+1}(\dd)$, we can write $\di\al = \Psi\be$
for a unique $\be\in\Om^{N-1}(\dd)$.
Then 
we have
$0 = \di^2\al = \Psi\di\be$\,; hence, $\di\be \in J^{N}(\dd)$.
We let $\diru\al=\di\be$. To show that $\diru$ is well defined,
we need to check that $\diru\al=0$ when $\al\in I^{N}(\dd)$.
Indeed, if $\al=\Psi\ga$ for some $\ga\in\Om^{N-2}(\dd)$, then
$\di\al = \Psi\di\ga$, which implies $\be=\di\ga$ and $\di\be=0$.
Connecting \eqref{domd1} and \eqref{domd2}, we obtain a sequence
\begin{equation}\label{domd6}
0 \to \Om^0(\dd) / I^0(\dd) \xrightarrow{\di} 
\cdots \xrightarrow{\di} \Om^{N}(\dd) / I^{N}(\dd) 
\xrightarrow{\diru}
J^{N}(\dd) \xrightarrow{\di} \cdots \xrightarrow{\di} J^{2N}(\dd) \,,
\end{equation}
which we will call the \emph{conformally symplectic pseudo de Rham complex} for $\chi=0$.
(Since for $\chi=0$ the form $\om$ is symplectic, one could also call this the symplectic pseudo de Rham complex.)

\begin{proposition}\label{pdomd5}
The sequence \eqref{domd6} is a complex, which is exact at all terms except\/ $J^{2N-1}(\dd)$ and\/ $J^{2N}(\dd)$. 
The map\/ $\diru$ is defined as follows{\rm:}
take any representative\/ $\al\in\Om^{N}(\dd)$ and write\/
$\di\al = \Psi\be;$ then\/ $\diru\al = \di\be$.
\end{proposition}
\begin{proof}
We have already shown that $\diru$ is well defined.
Next, it is clear by construction that 
$\diru\di=0$ and $\di\diru=0$.
Due to \leref{ldomd2}, it remains only to check exactness at the terms
$\Om^{N}(\dd) / I^{N}(\dd)$ and $J^{N}(\dd)$, the latter only for $N\ge 2$.

Let $\al\in\Om^{N}(\dd)$ be such that $\diru\al = 0$.
Then $\di\be = \diru\al = 0$\,; hence, $\be=\di\ga$
for some $\ga\in\Om^{N-2}(\dd)$ by the exactness of \eqref{domd}.
Then $\di\al=\Psi\be=\di\Psi\ga$
implies that $\al-\Psi\ga = \di\rho$ for some $\rho\in\Om^{N-1}(\dd)$.
Therefore, $\al+I^{N}(\dd) = \di(\rho+I^{N-1}(\dd))$.

Now suppose that $N\ge2$ and $\al\in J^{N}(\dd)$ satisfies $\di\al = 0$.
Then $\al=\di\be$ for some $\be\in\Om^{N-1}(\dd)$, again by the exactness of \eqref{domd}.
We have $\di\Psi\be=\Psi\al=0$. Since $\Psi\be\in\Om^{N+1}(\dd)$ and $N+1\le 2N-1$, there is
$\ga\in\Om^{N}(\dd)$ such that $\Psi\be=\di\ga$. 
Then $\diru\ga=\di\be=\al$, as desired.
\end{proof}

Recall that, by \cite{BDK1}, $\Om^n(\dd)=\T(\kk,\Om^n)$ is a tensor $\Wd$-module (see \seref{stwd}). Since $I\in\gld$ acts on $\Om^n$ as $-n\id$, \prref{phtm} tells us that, considered as an $\Hd$-module,
$\Om^n(\dd)$ is isomorphic to the tensor module $\T(\kk_{-n\chi/2},\Om^n)$.
Then its quotient $\Om^{n}(\dd) / I^{n}(\dd)$ and submodule $J^n(\dd)$ are also tensor $\Hd$-modules, because
$\Om^{n}(\dd) / I^{n}(\dd) = H \tp (\Om^n/I^n)$ and $J^n(\dd) = H \tp J^n$.
{}From \eqref{injn3}, we obtain isomorphisms of $\Hd$-modules:
\begin{align}\label{injn4}
\Om^{n}(\dd) / I^{n}(\dd) &\simeq \T(\kk_{-n\chi/2},R(\pi_n))
\,, 
\\ \label{injn4j}
J^{2N-n}(\dd) &\simeq \T(\kk_{-(2N-n)\chi/2},R(\pi_n)) 
\,, \qquad 0\leq n \leq N \,.
\end{align}
Furthermore, the differential $\di$ in \eqref{domd} is a homomorphism of $\Wd$-modules.
Hence, $\di$ in \eqref{domd6} is a homomorphism of $\Hd$-modules. We will prove in the next subsection that $\diru$ is a homomorphism too, and will generalize all these results to the case $\chi\ne0$.

\subsection{Twisted conformally symplectic pseudo de Rham complex}\label{sdrmh}

Now let us consider the general case, i.e., $\chi$ is not necessarily $0$.
When $\chi\ne0$, the maps $\di$ and $\Psi$ no longer commute. 
Accordingly, we will modify the definition of $\Psi$ using the twisting functors from \seref{stwrep}.

Fix a finite-dimensional $\dd$-module $\Pi$. 
Applying the twisting functor $T_\Pi$ to the pseudo de Rham complex \eqref{domd},
we obtain the exact complex
\begin{equation}\label{domd7}
0 \to \Om^0_\Pi(\dd) \xrightarrow{\di_\Pi} \Om^1_\Pi(\dd) \xrightarrow{\di_\Pi}
\cdots \xrightarrow{\di_\Pi} \Om^{2N}_\Pi(\dd) \,,
\end{equation}
where
\begin{equation}\label{dipi}
\di_\Pi = T_\Pi(\di) \,, \qquad  \Om^n_\Pi(\dd) = T_\Pi(\Om^n(\dd)) = H\tp\Pi\tp\Om^n
\,.
\end{equation}
As before, $H$ acts on $\Om^n_\Pi(\dd)$ by left multiplication on the first factor. As a $\Wd$-module,
$\Om^n_\Pi(\dd)=\T(\Pi,\Om^n)$ is a tensor module.

Recall that $\kk_\chi = \kk\,1_\chi$ where $\d\cdot 1_\chi = \chi(\d)1_\chi$ for $\d\in\dd$, and $\Pi_\chi$ denotes the $\dd$-module $\Pi \otimes \kk_\chi$.
We introduce $H$-linear maps
\begin{equation}\label{psichi}
\Psi_\chi\colon \Om^n_\Pi(\dd) \to \Om^{n+2}_{\Pi_\chi}(\dd)
\,, \qquad
h\tp u\tp\al \mapsto h\tp u\tp 1_\chi \tp (\om\wedge\al) \,,
\end{equation}
where $h\in H$, $u\in\Pi$, $\al\in\Om^n$.
We claim that these maps are compatible with the differentials.

\begin{lemma}\label{ldomd3}
With the above notation \eqref{psichi}, we have\/ $\Psi_\chi \di_\Pi = \di_{\Pi_\chi} \Psi_\chi$.
\end{lemma}
\begin{proof}
By $H$-linearity, it is enough to show that both sides of the desired identity are equal when applied to 
$1\tp u\tp\al\in\Om^n_\Pi(\dd)$ for $u\in\Pi$, $\al\in\Om^n$. Furthermore, we will assume that $\al=\io_{\d_i} \be$ for some $\be\in\Om^{n+1}$ and $1\le i\le 2N$, as such elements span $\Om^n$.
Note that, by \cite[Lemma 5.1]{BDK1}, 
\begin{equation}\label{dipi2}
\di\al \in \sum_{k=1}^{2N} \d_k \tp (e^k_i \cdot\be) + \kk\tp\Om^{n+1} \subset \Om^{n+1}(\dd) \,.
\end{equation}
Then from the definition \eqref{twrep8} of $T_\Pi$, we obtain
\begin{equation*}
\di_\Pi (1\tp u\tp\al) = \si_{12} (u\tp\di\al) - \sum_{k=1}^{2N} 1\tp (\d_k\cdot u) \tp (e^k_i \cdot\be) \,,
\end{equation*}
where $\si_{12}\colon \Pi\tp H\tp\Om^n \to H\tp\Pi\tp\Om^n$ is the transposition of the first two factors.
(This equation can also be derived from \cite[Lemma 5.3]{BDK1}.)
Hence,
\begin{align*}
\Psi_\chi\di_\Pi &(1\tp u\tp\al) \\
&= \si_{12} \bigl( (u\tp 1_\chi) \tp (\om\wedge\di\al) \bigr)
- \sum_{k=1}^{2N} 1\tp (\d_k\cdot u) \tp 1_\chi \tp \bigl(\om\wedge(e^k_i \cdot\be)\bigr) \,,
\end{align*}
where now $\si_{12}\colon \Pi_\chi \tp H\tp\Om^n \to H\tp \Pi_\chi\tp\Om^n$.

On the other hand,
\begin{equation*}
\Psi_\chi(1\tp u\tp\al) = 1\tp u\tp 1_\chi \tp (\om\wedge\al) \,,
\end{equation*}
and by \eqref{omchi4}, \eqref{dider}, we have
\begin{equation}\label{dipi3}
\di(\om\wedge\al) = \diz\om\wedge\al + \om\wedge\di\al
= -\chi\wedge\om\wedge\al + \om\wedge\di\al
\,.
\end{equation}
This and \eqref{dipi2} imply
\begin{equation*}
\di(\om\wedge\al) \in \sum_{k=1}^{2N} \d_k \tp (\om\wedge e^k_i \cdot\be) 
+ \kk\tp\Om^{n+3} \subset \Om^{n+3}(\dd) \,.
\end{equation*}
Using again \eqref{twrep8} and \eqref{dipi3}, we get
\begin{align*}
\di_{\Pi_\chi} &(1\tp u\tp 1_\chi \tp (\om\wedge\al)) \\
&= \si_{12} \bigl( (u\tp 1_\chi) \tp \di(\om\wedge\al) \bigr)
- \sum_{k=1}^{2N} 1\tp \d_k\cdot(u\tp 1_\chi) \tp \bigl(\om\wedge(e^k_i \cdot\be)\bigr) \\
&= \si_{12} \bigl( (u\tp 1_\chi) \tp (\om\wedge\di\al) \bigr)
- \sum_{k=1}^{2N} 1\tp (\d_k\cdot u) \tp 1_\chi \tp \bigl(\om\wedge(e^k_i \cdot\be)\bigr) \\
&- \si_{12} \bigl( (u\tp 1_\chi) \tp (\chi\wedge\om\wedge\al) \bigr)
- \sum_{k=1}^{2N} 1\tp  u \tp (\d_k\cdot 1_\chi) \tp \bigl(\om\wedge(e^k_i \cdot\be)\bigr)
\,.
\end{align*}

In order to finish the proof of the lemma, we need to show that the last two terms in the right-hand side cancel.
In other words, we claim that
\begin{equation*}
\sum_{k=1}^{2N} \chi(\d_k) e^k_i \cdot\be = -\chi\wedge (\io_{\d_i} \be) 
\,, \qquad \be\in\Om^{n+1} \,.
\end{equation*}
Indeed, evaluating either side on $\d_{j_1} \wedge\cdots\wedge \d_{j_{n+1}}$ gives
\begin{equation*}
\sum_{t=1}^{n+1} (-1)^t \chi(\d_{j_t}) \,
\be(\d_i \wedge \d_{j_1} \wedge\cdots\wedge \widehat{\d_{j_t}} \wedge\cdots\wedge \d_{j_{n+1}}) 
\,,
\end{equation*}
by using \eqref{wedge}, \eqref{ioal} and \eqref{acdotal}.
\end{proof}

\begin{lemma}\label{ldomd4}
The maps\/ $\Psi_\chi\colon \Om^n_\Pi(\dd) \to \Om^{n+2}_{\Pi_\chi}(\dd)$, defined by \eqref{psichi}, are
homomorphisms of\/ $\Hd$-modules.
\end{lemma}
\begin{proof}
By $H$-linearity, it is enough to show that $\Psi_\chi$ commutes with the action of $e$ on $1\tp u\tp\al$,
where $u\in\Pi$, $\al\in\Om^n$ and $e\in\Hd\subset\Wd$ is the generator \eqref{iota}.
Since $\Om^n_\Pi(\dd)=\T(\Pi,\Om^n)=\T(V_0)$ is a tensor $\Wd$-module, the action of $e$ on it is given by \eqref{eee2} 
for $v=u\tp\al \in V_0 = \Pi\boxtimes\Om^n$.

We have:
\begin{equation*}
I\cdot\al=-n\al \,, \qquad I\cdot(\om\wedge\al)=-(n+2)(\om\wedge\al)
\,, \qquad \al\in\Om^n \,,
\end{equation*}
\begin{equation*}
\d\cdot(u\tp 1_\chi) = (\d\cdot u) \tp 1_\chi +\chi(\d) u\tp 1_\chi
\,, \qquad \d\in\dd \,, \;\; u\in\Pi \,,
\end{equation*}
and, by \eqref{acder},
\begin{equation*}
A\cdot(\om\wedge\al) = (A\cdot\om) \wedge\al + \om\wedge(A\cdot\al) = \om\wedge(A\cdot\al)
\,, \qquad A\in\spd \,.
\end{equation*}
Plugging these into \eqref{eee2}, we get:
\begin{align*}
e* \bigl(& \Psi_\chi(1\tp u\tp\al)\bigr)
-\bigl( (\id\tp\id)\tp_H \Psi_\chi \bigr) \bigl(e*(1\tp u\tp\al)\bigr) \\
&= e* \bigl( 1\tp u\tp 1_\chi \tp (\om\wedge\al) \bigr)
-\bigl( (\id\tp\id)\tp_H \Psi_\chi \bigr) \bigl(e*(1\tp u\tp\al)\bigr) \\
&=-\sum_{k=1}^{2N} (\bd_k \tp 1) \otimes_H \Bigl( \chi(\d^k) + \frac12 \chi(\d^k) (-2) \Bigr) 
\bigl(1\tp u\tp 1_\chi \tp (\om\wedge\al)\bigr) \\
& = 0\,.
\end{align*}
Therefore, $\Psi_\chi$ is an $\Hd$-module homomorphism.
\end{proof}

Notice that the kernel of $\Psi_\chi\colon \Om^n_\Pi(\dd) \to \Om^{n+2}_{\Pi_\chi}(\dd)$ is
\begin{equation}\label{jnpi}
J^n_\Pi(\dd) = T_\Pi(J^n(\dd)) = H\tp\Pi\tp J^n \subset \Om^n_\Pi(\dd) \,,
\end{equation}
and the image of $\Psi_{-\chi}\colon \Om^{n-2}_{\Pi_\chi}(\dd) \to \Om^{n}_\Pi(\dd)$ is
\begin{equation}\label{inpi}
I^n_\Pi(\dd) = T_\Pi(I^n(\dd)) = H\tp\Pi\tp I^n \subset \Om^n_\Pi(\dd) \,.
\end{equation}
As a consequence of Lemmas \ref{ldomd3} and \ref{ldomd4}, we obtain complexes of $\Hd$-modules
\begin{align}
\label{domdpi1}
0 \to \Om^0_\Pi(\dd) / I^0_\Pi(\dd) &\xrightarrow{\di_\Pi} 
\Om^1(\dd_\Pi) / I^1_\Pi(\dd) 
\xrightarrow{\di_\Pi} \cdots \xrightarrow{\di_\Pi} \Om^{N}_\Pi(\dd) / I^{N}_\Pi(\dd)
\intertext{and}
\label{domdpi2}
J^{N}_\Pi(\dd) &\xrightarrow{\di_\Pi} J^{N+1}_\Pi(\dd) 
\xrightarrow{\di_\Pi} \cdots \xrightarrow{\di_\Pi} J^{2N}_\Pi(\dd) \,.
\end{align}
Note that \eqref{domdpi1} and \eqref{domdpi2} are obtained from the complexes of $\Hd$-modules
\eqref{domd1} and \eqref{domd2} by applying the twisting functor $T_\Pi$.
Due to \eqref{injn4} and \eqref{injn4j}, these are tensor modules:
\begin{align}\label{injn5}
\Om^{n}_\Pi(\dd) / I^{n}_\Pi(\dd) &\simeq \T(\Pi_{-n\chi/2},R(\pi_n)) \,, 
\\ \label{injn5j}
J^{2N-n}_\Pi(\dd) &\simeq \T(\Pi_{-(2N-n)\chi/2},R(\pi_n)) \,, 
\qquad 0\leq n \leq N \,.
\end{align}
In the next lemma, we construct a twisted analog of Rumin's map $\diru$.

\begin{lemma}\label{ldomd5}
There is a well-defined\/ $\Hd$-module homomorphism
\begin{equation*}
\diru_\Pi\colon \Om^{N}_\Pi(\dd) / I^{N}_\Pi(\dd) \to J^{N}_{\Pi_{-\chi}}(\dd)
\,,
\end{equation*}
such that\/ $\diru_\Pi \, \di_\Pi = 0$ and\/ $\di_{\Pi_{-\chi}} \, \diru_\Pi = 0$.
\end{lemma}
\begin{proof}
Pick any $\al\in\Om^{N}_\Pi(\dd)$. 
Since $\Psi$ gives an isomorphism $\Om^{N-1} \simeq \Om^{N+1}$, 
we can write $\di_\Pi\al = \Psi_\chi\be$
for a unique $\be\in\Om^{N-1}_{\Pi_{-\chi}}(\dd)$.
By \leref{ldomd3}, we have
\begin{equation*}
0 = (\di_\Pi)^2 \al = \di_\Pi \Psi_\chi \be = \Psi_\chi \di_{\Pi_{-\chi}} \be
\,,
\end{equation*}
which means that $\diru_\Pi\al := \di_{\Pi_{-\chi}} \be \in J^{N}_{\Pi_{-\chi}}(\dd)$.
To show that $\diru_\Pi$ is well defined, we need to check that $\diru_\Pi\al=0$ when $\al\in I^{N}_\Pi(\dd)$.
Indeed, if $\al=\Psi_\chi\ga$ for some $\ga\in\Om^{N-2}_{\Pi_{-\chi}}(\dd)$, then
\begin{equation*}
\di_\Pi \al = \di_\Pi \Psi_\chi \ga = \Psi_\chi \di_{\Pi_{-\chi}} \ga 
\,.
\end{equation*}
Hence, we can take $\be=\di_{\Pi_{-\chi}} \ga$, and we get $\diru_\Pi\al = \di_{\Pi_{-\chi}} \be=0$.

The map $\diru_\Pi$ is a homomorphism of $\Hd$-modules, because it is a composition of homomorphisms:
\begin{equation}\label{dirupi}
\diru_\Pi = \di_{\Pi_{-\chi}} \bigl( \Psi_\chi \big|_{\Om^{N-1}_{\Pi_{-\chi}}(\dd)} \bigr)^{-1} \di_\Pi
\,.
\end{equation}
Finally, the identities $\diru_\Pi \, \di_\Pi = 0$ and $\di_{\Pi_{-\chi}} \, \diru_\Pi = 0$ are obvious.
\end{proof}

Using \leref{ldomd5}, we can connect the complexes \eqref{domdpi1} and \eqref{domdpi2} to obtain
the \emph{twisted conformally symplectic pseudo de Rham complex}:
\begin{equation}\label{domd8}
\begin{split}
0 &\to \Om^0_\Pi(\dd) / I^0_\Pi(\dd) \xrightarrow{\di^1_\Pi} 
\Om^1(\dd_\Pi) / I^1_\Pi(\dd) 
\xrightarrow{\di^2_\Pi} \cdots \xrightarrow{\di^N_\Pi} \Om^{N}_\Pi(\dd) / I^{N}_\Pi(\dd)
\\
&\xrightarrow{\diru_\Pi}
J^{N}_{\Pi_{-\chi}}(\dd) \xrightarrow{\di^{N+1}_{\Pi_{-\chi}}} 
J^{N+1}_{\Pi_{-\chi}}(\dd) \xrightarrow{\di^{N+2}_{\Pi_{-\chi}}}
\cdots \xrightarrow{\di^{2N}_{\Pi_{-\chi}}} 
J^{2N}_{\Pi_{-\chi}}(\dd) \,,
\end{split}
\end{equation}
where we decorated the differentials with superscripts that will turn out to be convenient later.

\begin{lemma}\label{ldomd6}
The twisted conformally symplectic pseudo de Rham complex \eqref{domd8} is exact at all terms except for\/ 
$J^{2N-1}_{\Pi_{-\chi}}(\dd)$ and\/ $J^{2N}_{\Pi_{-\chi}}(\dd)$. 

\end{lemma}
\begin{proof}
Recall that the twisting functor $T_\Pi$ preserves exactness. Hence, the complexes
\eqref{domd7} and \eqref{domdpi1} are exact, while \eqref{domdpi2} is exact at all terms except the last two
(cf.\ \leref{ldomd2}). The rest of the proof is the same as in \prref{pdomd5}.
\end{proof}

We summarize the above results in the following theorem.

\begin{theorem}\label{tmodhd}
For every finite-dimensional\/ $\dd$-module\/ $\Pi$, we have a complex of tensor $\Hd$-modules
\begin{equation}\label{domd9}
\begin{split}
0 &\to \T(\Pi,R(\pi_0)) \xrightarrow{\di^1_\Pi} \T(\Pi_{-\chi/2},R(\pi_1))
\xrightarrow{\di^2_\Pi} \cdots \xrightarrow{\di^N_\Pi} \T(\Pi_{-N\chi/2},R(\pi_N))
\\
&\xrightarrow{\diru_\Pi}
\T(\Pi_{-(N+2)\chi/2},R(\pi_N)) \xrightarrow{\di^{N+1}_{\Pi_{-\chi}}} 
\T(\Pi_{-(N+3)\chi/2},R(\pi_{N-1})) 
\\
&\xrightarrow{\di^{N+2}_{\Pi_{-\chi}}} \cdots \xrightarrow{\di^{2N}_{\Pi_{-\chi}}} 
\T(\Pi_{-(2N+2)\chi/2},R(\pi_0)) \to 0 \,,
\end{split}
\end{equation}
where\/ $\diru_\Pi$ is given by \eqref{dirupi}. The complex \eqref{domd9} is exact at all terms except for the last two,
and its cohomology is finite dimensional.
\end{theorem}
\begin{proof}
The only remaining statement is the finite-dimensionality of the cohomology. This follows from the fact that the cohomology of the complex \eqref{domd2} is finite dimensional and the twisting functor $T_\Pi$ preserves this property (see \seref{stwrep}).
\end{proof}


\section{Singular vectors and tensor modules}\label{skten}
\subsection{Conformal representations of $\widetilde{\P}$}\label{confrep}

Due to \prref{preplal2}, any $\Hd$-module has a natural structure of a
conformal $\widetilde \P$-module, and vice versa. If $V$ is an irreducible non-trivial $\Hd$-module, then each element $v \in V$ lies in some $\ker_p V$, as from \leref{lkey2}, which is finite-dimensional and stabilized by the action of $\N_\P$.
As usual, one has
\begin{proposition}\label{pnhconformal}
The subalgebra\/ $\P_1 \subset \N_\P$ acts trivially on every
irreducible finite-dimensional conformal\/ $\N_\P$-module. Irreducible
finite-dimensional conformal\/ $\N_\P$-modules are in one-to-one
correspondence with irreducible finite-dimensional modules over the
Lie algebra\/ $\N_\P / \P_1 = (\iota_*^{-1}(\widehat \dd) \sd \P_0)/\P_1 \simeq \dd'\oplus\spd$.
\end{proposition}
\begin{proof}
Similar to the proof of \cite[Proposition 3.4]{BDK1}.
\end{proof}

\subsection{Singular vectors}

\begin{definition}\label{dhsing}
A {\em singular vector\/} in $V$ is an element
$v\in V$ such that $\P_1\cdot v = 0$. The space of singular vectors
in $V$ will be denoted by $\sing V$, and is stabilized by the action of $\N_\P$, as it normalizes $\P_1$. Then
$$\rho_{\sing}
\colon \dd'\oplus\cspd \to \gl(\sing V)$$
will indicate the representation obtained
from the $\N_\P$-action on $\sing V\equiv\ker_1 V$ via the
isomorphism $\N_\P / \P_1 \simeq \dd'\oplus\spd$, so that
\begin{equation}\label{rhosing}
\rho_{\sing}(f^{ij}) v =
\frac{1}{2}(x^i x^j \tp_H e)v\,, \qquad v\in \sing
V \,,
\end{equation}
whereas $\dd'$ is generated over $\kk$ by elements $\widehat \d, \d \in \dd,$ and $c:= \ce$ as in the proof of \leref{ldprime}.
\end{definition}
A vector $v\in V$ is singular if and only if
\begin{equation}\label{vsingf1}
e*v \in (\fil^2 H \tp \kk) \tp_H V \,,
\end{equation}
or equivalently
\begin{equation}\label{vsingf2}
e*v \in (\kk \tp \fil^2 H) \tp_H V \,.
\end{equation}


\begin{proposition}\label{phsing1}
Let $V$ be a finite non-zero $\Hd$-module. Then the vector space $\sing V$ is non-zero and the space $\sing V/\ker V$ is finite dimensional.
\end{proposition}
\begin{proof}
Finite dimensionality of $\ker_n V/\ker V$ for all $n$ was stated in \leref{lkey2} above.

In order to prove the other claim, we may assume without loss of
generality that $\ker V = \{0\}$. As the $\P$-module $V$ is
conformal, then $\ker_n V$ is non-zero for some $n\geq 0$. Note that
$\ker_n V$ is killed by $\P_n$, so it is preserved by its normalizer
$\N_\P$. Choose now an irreducible $\N_\P$-submodule $U \subset
\ker_n V$. As $U$ is finite dimensional, \prref{pnhconformal} shows
that the action of $\P_1$ on $U$ is trivial, so that $U \subset
\sing V$.
\end{proof}

\begin{lemma}\label{singtosing}
Let\/ $f\colon M \to N$ be a homomorphism of\/ $\Hd$-modules. Then\/
$f(\sing M) \subset \sing N$.
\end{lemma}
\begin{proof}
The proof is obvious, since $f$ commutes with the action of $\P_1$.
\end{proof}

\begin{lemma}
Let\/ $f\colon M \to N$ be a homomorphism of\/ $\Hd$-modules. If\/ $U
\subset \sing M$ is an irreducible\/ $\dd' \oplus \spd$-submodule, then
either\/ $U \subset \ker f$ or $f|_U$ is an isomorphism.
\end{lemma}
\begin{proof}
Follows from Schur's Lemma and the fact that
$f|_U$ is a homomorphism of $\N_\P/\P_1$-modules.
\end{proof}

\begin{remark}
The above lemma applies whenever $M$ is a tensor module, as such
modules are generated in degree zero. See Section \ref{filtrationtm} below for the definition of the degree.
\end{remark}

We would like to give a description of the pseudoaction of $\Hd$ on singular vectors in terms of $\rho_{\sing}$. However, when using \prref{preplal2}, it is better to express \eqref{prpl2} in terms of Fourier coefficients of the form $(x^i e^{-\chi}) \otimes_H e$. In order to do so, we need the following result.

\begin{proposition}\label{phsing}
If $v\in \sing V$, then the action of\/ $\Hd$ on $v$ is given by
\begin{equation}\label{Hactionsing}
\begin{split}
e * v = \sum_{i,j=1}^{2N}  & \, (\bar\d_i \bar\d_j \tp 1) \tp_H
\rho_{\sing}(f^{ij})v\\
- \sum_{k=1}^{2N} & \, (\bar\d_k \tp 1) \tp_H \bigl(\rho_{\sing} (\d^k +
\adsp \d^k) v - \d^k v\bigr)\\
+ \quad\,\,\,\, & \,(1 \tp 1) \tp_H \rho_{\sing} (c)v.
\end{split}
\end{equation}\end{proposition}
\begin{proof}
As $\P_1$ acts trivially on singular vectors, \reref{pseudoactioncheck} shows that the action of $e$ on a singular vector $v$ is
given by the following expression
\begin{equation*}
\begin{split}
e * v = \sum_{1 \leq i<j \leq 2N} &
(\overline{S(\d_i \d_j)} \tp 1) \tp_H (x^i x^j e^{-\chi} \tp_H e).v\\
+ \sum_{i=1}^{2N} \quad\,\, & (\overline{S(\d_i^2)} \tp 1) \tp_H
\left(\frac{1}{2}(x^i)^2 e^{-\chi} \tp_H e\right).v\\
 + \sum_{k=1}^{2N} \quad\,\, & (\overline{S(\d_k)} \tp 1) \tp_H (x^k e^{-\chi} \tp_H e).v\\
 + \qquad \quad  & (\overline{S(1)} \otimes 1) \tp_H (e^{-\chi} \tp_H e).v.
\end{split}
\end{equation*}
This rewrites as
\begin{equation*}
\begin{split}
e * v = \sum_{1\leq i<j\leq 2N} &
({\bar\d}_j \bar\d_i \tp 1) \tp_H (x^i x^j \tp_H e).v\\
+ \sum_{i=1}^{2N} \quad\,\, & (\bar \d_i^2 \tp 1) \tp_H
\left(\frac{1}{2}(x^i)^2 \tp_H e\right).v\\
 - \sum_{k=1}^{2N} \quad\,\, & (\bar\d_k \tp 1) \tp_H (x^k e^{-\chi} \tp_H e).v\\
 + \qquad \quad & (1 \tp 1) \tp_H (e^{-\chi} \tp_H e).v\,,
\end{split}
\end{equation*}
where we used that $x^i x^j e^{-\chi}$ coincides with $x^i x^j$ up to higher degree terms.
We may use \leref{ldtilde} to express the third summand on the right-hand side as
 \begin{equation*}
\begin{split}
 & \sum_{k=1}^{2N} \, (\bar\d_k \tp 1) \tp_H (x^k e^{-\chi} \tp_H e).v
 \\
=& \sum_{k=1}^{2N} \, (\bar\d_k \tp 1) \tp_H \Bigl(\rho_{\sing}(\d^k + \adsp \d^k)v - \d^k v - \sum_{1\leq i<j\leq 2N} c_{ij}^k \rho_{\sing}(f^{ij}) v\Bigr)
\end{split}
\end{equation*}
Now we rewrite the first summand in the right-hand side
using \eqref{rhosing} and the fact that $2\bar \d_j \bar \d_i =
(\bar \d_i \bar \d_j + \bar \d_j \bar \d_i) - [\bar \d_i, \bar
\d_j]$. Then the first two summands together yield

\begin{equation*}
\sum_{i,j=1}^{2N} (\bar\d_i\bar\d_j \tp 1) \tp_H (\rho_{\sing}(f^{ij})v)
- \sum_{k=1}^{2N} (\bar \d_k\tp 1)\tp_H \sum_{1\leq i<j\leq 2N} c_{ij}^k \rho_{\sing}(f^{ij})v.
\end{equation*}
The two summations over $i<j$ then cancel each other, giving \eqref{Hactionsing}.
\end{proof}

\begin{corollary}\label{chsing}
Let $V$ be an $\Hd$-module and let $R$ be a
$\dd'\oplus\spd$-submodule of\/ $\sing V$. Denote by $HR$ the
$H$-submodule of\/ $V$ generated by $R$. Then $HR$ is a
$\Hd$-submodule of\/ $V$. In particular, if\/ $V$ is irreducible and $R$ is non-zero,
then $V=HR$.
\end{corollary}
\begin{proof}
By \eqref{Hactionsing}, we have $\Hd * R \subset (H \tp H) \tp_H HR$. Then
$H$-bilinearity implies that 
$HR$ is a $\Hd$-submodule of $V$.
\end{proof}

\subsection{More tensor modules}

Let $R$ be a finite-dimensional $\dd'\oplus\spd$-module, with the
action denoted as $\rho_R$. Let $V=H\tp R$ be the free $H$-module
generated by $R$, where $H$ acts by a left multiplication on the
first factor. Mimicking \eqref{Hactionsing}, we define a
pseudoaction
\begin{equation}\label{hsing1}
\begin{split}
e * v = \sum_{i,j=1}^{2N}  & \, (\bar \d_i \bar \d_j \tp 1) \tp_H
(1 \tp \rho_R (f^{ij})v)\\
- \sum_{k=1}^{2N} & \, (\bar \d_k \tp 1) \tp_H \bigl(1 \tp \rho_R (\d^k + \adsp
\d^k) v - \d^k \tp v\bigr)\\
+ \quad\,\,\,\, & \,(1 \tp 1) \tp_H (1 \tp \rho_{R} (c)v),\qquad v \in R,
\end{split}
\end{equation}
and then extend it by $H$-bilinearity to a map $* \colon \Hd\tp V
\to (H\tp H)\tp_H V$.

\begin{lemma}\label{lhsing2}
Let\/ $R$ be a finite-dimensional $\dd\oplus\spd$-module with an
action $\rho_R$. Then formula \eqref{hsing1} defines a structure of
a $\Hd$-module on $V=H\tp R$. We have $\kk \tp R \subset \sing V$
and
\begin{equation}\label{rhosingA}
\rho_{\sing}(A)(1 \tp u) = 1 \tp \rho_R(A)u, \qquad A \in \dd \oplus
\spd, \; u \in R.
\end{equation}
\end{lemma}
\begin{proof}
It can be showed by means of a lengthy but straightforward
computation.
Alternatively, the identification $\dd'\oplus\spd= \N_\P / \P_1$
makes $R$ into a conformal $\N_\P$-module with a trivial
$\P_1$-action. As $\widetilde \P = \dd \oplus \N_\P$ as vector
spaces, then the induced module $\Ind_{\N_\P}^{\widetilde \P} R$ is
isomorphic to $\ue(\dd) \tp R = H \tp R$. The $H$-module structure
is then given by multiplication on the first factor, whereas the
$\N_\P/\P_1$-action on $\kk \tp R$ is as given in \eqref{rhosingA}.
Thus we have built a conformal $\widetilde \P$-module -- i.e., a
$\Hd$-module -- structure on $H\tp R$.
\end{proof}

\begin{definition}\label{dvmodw}
{\rm(i)} Let $R$ be a finite-dimensional $\dd'\oplus\spd$-module with
an action $\rho_R$. Then the $\Hd$-module $H\tp R$, with the action
of $\Hd$ given by \eqref{hsing1}, will be denoted as~$\V(R)$.

{\rm(ii)} Let $R = \Pi'\boxtimes U$, where $\Pi'$ is a
finite-dimensional $\dd'$-module and $U$ is a finite-dimensional
$\spd$-module. Then the module $\V(R)$ will also be denoted
as~$\V(\Pi',U)$.

{\rm(iii)} As before, when $R = \Pi\boxtimes U$, where $\Pi$ is a
finite-dimensional $\dd$-module and $U$ is a finite-dimensional
$\spd$-module, we will abuse the notation and write $\V(R) = \V(\Pi, U):= \V(\pi^* \Pi, U)$, where $\pi^*\Pi$ is the $\dd'$-module obtained by lifting the $\dd$-action on $\Pi$ to $\dd'$ via composing with the canonical projection $\pi: \dd' \to \dd'/\kk c \simeq \dd$.
\end{definition}

It is possible to translate Definitions \ref{dtmodh} and \ref{dvmodw} into one another. 

\begin{proposition}
Let\/ $R = \Pi' \boxtimes U$ be a finite-dimensional\/ $\dd'\oplus \spd$-module and $\pi: \dd' \to \dd'/\kk c \simeq \dd$ denote the canonical projection. Then\/
$\V(\Pi', U) = \T(\Pi'_{\pi^*\phi}, U)$ and\/
$\T(\Pi', U) = \V(\Pi'_{-\pi^*\phi}, U)$,
where\/ $\phi = - N\chi+\tr\ad$ and $\pi^*\phi = \phi \circ \pi$ is the corresponding trace form on $\dd'$.
\end{proposition}
\begin{proof}
The two claims are clearly equivalent, so it suffices to prove the latter. Take the definition of $\T(\Pi', U)$ given in \eqref{eee3} and rewrite the first summand in the right-hand side as follows
$$\sum_{k=1}^{2N} (\bd_k \tp \d^k) \otimes_H v = \sum_{k=1}^{2N} (\bd_k \tp 1) \otimes_H \d^k v - \sum_{k=1}^{2N} (\bd_k (\bd^k + \chi(\d^k)) \tp 1 ) \tp_H v.$$

One has 
$$\sum_{k=1}^{2N} (\bd_k (\bd^k + \chi(\d^k)) \tp 1 ) \tp_H v = \sum_{k=1}^{2N} (\bd_k \bd^k \tp 1) \tp_H v - \sum_{k=1}^{2N}(\bd_k \tp 1) \tp_H \chi(\d^k) v.$$
One may one use \eqref{delx} to rewrite the first summand as
$$\sum_{k=1}^{2N} (\bd_k \bd^k \tp 1) \tp_H v = \sum_{k = 1}^{2N} (\bd_k \tp 1) \tp_H \iota_\rho \omega(\d^k)v$$ so that 
$$\sum_{k=1}^{2N} (\bd_k \tp \d^k) \otimes_H v = \sum_{k=1}^{2N} (\bd_k \tp 1) \otimes_H \d^k v - \sum_{k=1}^{2N} (\bd_k \tp 1 ) \tp_H \phi(\d^k)v.$$

Now \eqref{eee3} rewrites as
\begin{equation}\label{translatettov}
\begin{split}
e * v = \sum_{i,j=1}^{2N}  & \, (\bar \d_i \bar \d_j \tp 1) \tp_H
(1 \tp f^{ij} \cdot v)\\
- \sum_{k=1}^{2N} & \, (\bar \d_k \tp 1) \tp_H (1 \tp (\d^k + \adsp
\d^k)\cdot v - \phi(\d^k) \tp v - \d^k \tp v)\\
+ \quad\,\,\,\, & \,(1 \tp 1) \tp_H c\cdot v,
\end{split}
\end{equation}
for all $v \in R$.
\end{proof}

The following statements are analogous to those proved in \cite{BDK, BDK1, BDK2}.

\begin{theorem}\label{irrfacttens}
Let\/ $V$ be an irreducible finite\/ $\Hd$-module, and\/ $R$ be an
irreducible\/ $\dd' \oplus \spd$-submodule of\/ $\sing V$. Then\/ $V$ is a
homomorphic image of\/ $\V(R)$. In particular, every irreducible finite\/
$\Hd$-module is a quotient of a tensor module.
\end{theorem}
\begin{proof}
By \eqref{Hactionsing} and \eqref{hsing1} the canonical projection
$\V(R) = H \tp R \to HR$ is a homomorphism of $\Hd$-modules.
However, $HR=V$ by \coref{chsing}.
\end{proof}

We will now investigate how reducibility of a tensor module depends on
existence of non-constant singular vectors. The picture is more involved than with primitive Lie pseudoalgebras of the other types.

\begin{lemma}\label{noconstants}
If $R$ is an irreducible representation of $\dd' \oplus \spd$, then any non-zero proper $\Hd$-submodule $M$ of $\V(R) = H \tp R$ satisfies $M
\cap \fil^0 \V(R) = \{0\}$.
\end{lemma}
\begin{proof}
Both $M$ and $\kk \tp R$ are $\dd' \oplus \spd$-stable, and the
same is true of their intersection. The claim now follows from
irreducibility of $R$.
\end{proof}

\begin{corollary}\label{constirr}
If $R$ is an irreducible $\dd' \oplus \spd$-module and $\sing \V(R) = \fil^0 \V(R)$, then the $\Hd$-module $\V(R)$ is
irreducible.
\end{corollary}
\begin{proof}
Assume there is a non-zero proper submodule $M$. Then $M$ must contain some non-zero singular vector. However, $M \cap \sing \V(R) = \{0\}$
by \leref{noconstants}.
\end{proof}

A converse to \coref{constirr} only holds for those tensor modules $\V(R)$ in which the $\P$-action factors through $\H \simeq \P/\kk c$. Recall 
that $\P$ is graded by the eigenspace decomposition of $\ad
\E$, which also provides a grading for $\widetilde \P$ as $\E$ commutes with elements $\widehat \d, \d \in \dd$. If $\p_n$ denotes the graded summand of eigenvalue $n$, then
one has the semidirect sum decomposition of Lie algebras
\begin{align}
\N_\P & = \dd' \sd \prod_{j\geq 0} \p_j\\
\intertext{and the decomposition of vector spaces} \widetilde \P & =
\p_{-1} \oplus \N_\P.
\end{align}
Notice that $\p_{-1}$ is not a subalgebra, as the Lie bracket of suitably chosen elements yields a nonzero multiple of the central element $c = \ce \in \p_{-2}$, which may act nontrivially on $R$. 

However, if the action of $c$ is trivial on $R$ then the $\P$-action factors through $\H$, and we may repeat the above argument with the Lie algebra $\widetilde\H$ so as to obtain $\widetilde \H =
\h_{-1} \oplus \N_\H$. Then $\h_{-1}$ is a graded (abelian) Lie algebra and its
universal enveloping algebra is a symmetric algebra generated in
degree $-1$. Then $\V(R) = \Ind_{\N_\H}^{\wti\H} R$ is isomorphic to
$\symp(\h_{-1})\tp R$, which can be endowed with a $-\ZZ_+$-grading
by setting elements from $R$ to have degree zero, and elements from
$\h_{-i}$ to have degree $-i$.  A vector $v \in \V(R)$ is then {\em homogeneous} if it has a single graded component with respect to the above grading. Since $\N_\H/\H_1$ may be identified with the degree $0$ part $\widehat \dd \oplus \h_0 \simeq \dd \oplus \spd$, its action on singular vectors will map homogeneous vectors to homogeneous vectors of the same degree, thus showing that homogeneous components of singular vectors are singular vectors.

The above argument, used in the case where $c \in \p_{-2}$ acts nontrivially on $R$, shows that $\V(R)$ is only $-\ZZ_+$-filtered, but certainly not graded, as the central element, which has degree $-2$, acts via a nonzero scalar multiplication so preserving the degree of vectors in the representation.

\begin{remark}
Let $V$ be an irreducible $\Hd$-module on which the $\P$-action factors through $\H$. 
As $\E$ stabilizes $\H_1$ and commutes with $\h_0$, isotypic components of $\sing V$ are stabilized by the action of $\E$. In particular, if an isotypic component of $\sing V$ is a single $\dd \oplus \spd$-irreducible summand, then $\E$ acts as multiplication by a scalar, and this summand is homogeneous. We will later see that in the framework of $\Hd$-tensor modules there is a single exception to homogeneity of irreducible summands in $\sing V$, which happens when $\chi = 0$.
\end{remark}

\begin{proposition}\label{constiff}
Let $R$ be an irreducible $\dd' \oplus \spd$-module on which $c\in \dd'$ acts trivially. Then
every noncostant homogeneous singular vector in $\V(R)$ is contained in a non-zero proper submodule. In particular, $\V(R)$ is irreducible if and only if $\sing \V(R) = \fil^0 \V(R)$.
\end{proposition}
\begin{proof}
Every non-constant homogeneous singular vector $s$ is of negative
degree $d<0$. Then $\ue(\widetilde \H) s = S(\h_{-1})
s$ is a non-zero $\Hd$-submodule of $\V(R)$ lying in degrees $\leq d$, and
intersects $R$ trivially, as $R$ lies in degree zero.
\end{proof}

\subsection{Filtration of tensor modules}\label{filtrationtm}

$\Hd$-tensor modules possess a filtration induced by the canonical
increasing filtration $\{\fil^n H\}$. We may in fact define
\begin{equation}
\fil^n \V(R) = \fil^n H \tp R \,, \qquad n = -1, 0, \dots\, .
\end{equation}
As usual, $\fil^{-1} \V(R) = \{0\}$ and $\fil^0 \V(R) = \kk \tp R$, and elements lying in $\fil^{n} \V(R) \setminus \fil^{n-1} \V(R)$ are called
{\em vectors of degree $n$}.

The
associated graded space is defined accordingly, and we have
isomorphisms of vector spaces
\begin{equation}
\gr^n \V(R) \simeq S^n \dd \tp R \,,
\end{equation}
where $S^n \dd$ is the $n$-th symmetric power of the vector space $\dd$.
%
Note that we can make $\dd$ into an $\spd$-module via the inclusion $\spd
\subset \gld.$ Then $\dd$ is an irreducible $\spd$-module. In a
similar way $S^n \dd$ will be understood as an $\spd$-module.

\begin{lemma}\label{lfilact}
For every $p,n \geq 0$, we have:
\begin{align}
\label{dfp}\dd \cdot \fil^p \V(R) & = \fil^{p+1} \V(R),\\
\label{nhfp}\N_\P \cdot \fil^p \V(R) & \subset \fil^p \V(R),\\
\label{hfp}\P \cdot \fil^p \V(R) & \subset \fil^{p+1} \V(R),\\
\label{h1fp}\P_n \cdot \fil^p \V(R) & \subset \fil^{p-n} \V(R).
\end{align}
\end{lemma}
\begin{proof}
The proof of \eqref{dfp}-\eqref{hfp} is the same as in the case of
$\Wd$ (see \cite[Lemma 6.3]{BDK1}). Similarly, \eqref{h1fp} follows
using \eqref{hfp} and the fact that $[\dd, \P_n] \subset \P_{n-1}$, $[\dd, \P_1] \subset \P_0 \subset \N_\P$.
\end{proof}
%
As a consequence of \leref{lfilact}, we see that both $\N_\P$ and its quotient $\N_\P/\P_1
\simeq \dd' \oplus \spd$ act on each space $\gr^p \V(R)$. 
%

\begin{lemma}\label{legraction}
The action of\/ $\dd'$ and\/ $\P_0/\P_1 \simeq \spd$ on
the space\/ $\gr^p \V(R) \simeq S^p \dd \tp R$ satisfies
\begin{align}
\widehat \d \cdot (f \tp u) & = f \tp \Bigl(\rho_R(\d) u + \frac{p}{2}\chi(\d) u\Bigr) ,\\
A \cdot (f \tp u) & = A f \tp u + f \tp \rho_R(A) u,
\end{align}
for\/ $A \in \spd$, $f\in S^p \dd$, $u \in R$.
\end{lemma}
\begin{proof}
One only needs to pay attention to the fact that, by \eqref{dhat}, the 
action of $\widehat \d$ on $\dd$ is given by scalar multiplication by $\chi(\d)/2$. 
\end{proof}

\begin{corollary}\label{grdecomposition}
Let $\Pi', U$ be irreducible representations of $\dd', \spd$ respectively. There is an isomorphism of $\dd' \oplus \P_0/\P_1 \simeq \dd\oplus\spd$-modules
$$\gr^p \V(\Pi', U) \simeq \Pi'_{p\,\pi^*\chi/2} \boxtimes
(S^p \dd \tp U).$$
In particular, when the action of $\P$ factors through the quotient $\H$ and $\Pi' = \pi^* \Pi$, one has
$$\gr^p \V(\Pi, U) \simeq  \Pi_{p\chi/2} \boxtimes
(S^p \dd \tp U),$$
whereas, when the action of $c \in \dd'$ on $\Pi'$ is nontrivial (whence $\chi$ vanishes) one has
$$\gr^p \V(\Pi', U) \simeq \Pi' \boxtimes
(S^p \dd \tp U).$$
\end{corollary}
\begin{proof}
Follows immediately by \leref{legraction}.
\end{proof}

\begin{remark}
Notice that when $\chi \neq 0$ one has $$\fil^p \V(\Pi, U) \simeq \bigoplus_{i=0}^p \Pi_{i\chi/2} \boxtimes
(S^i \dd \tp U),$$
by complete reducibility of the $\spd$-action, as each  $\gr^i \V(\Pi, U), 0 \leq i\leq p,$ sits in distinct isotypical components.
\end{remark}

\begin{remark}\label{degreeone}
We will later use the fact that all $\Hd$-linear homomorphisms showing up in the twisted conformally symplectic pseudo de Rham complexes, with the exception of $\di^R_\Pi$, map constant singular vectors to vectors of degree $1$.
This is due to the very definition of such maps, and the fact that the differential homomorphisms in the pseudo de Rham complex for $\Wd$ satisfy the same property.
\end{remark}

\section{Irreducibility of tensor modules}\label{sirtm}

We have seen that all irreducible finite representations of $\Hd$
arise as quotients of some special $\Hd$-modules called tensor
modules. A classification of irreducible $\Hd$-modules will then
follow from a description of irreducible quotients of tensor
modules.

\subsection{Coefficients of elements and submodules}\label{shirten}

Throughout this section, $R$ will be an irreducible $\dd' \oplus
\spd$-module, with the action denoted by $\rho_R$, and $\V(R)$ the corresponding tensor module
introduced in Definition \ref{dvmodw}.
First, let us rewrite the expression \eqref{hsing1} for the
pseudoaction of $e$ on singular vectors in a tensor module in a
different fashion. Let 
$$\psi(u) = \sum_{i,j=1}^{2N} \d_i \d_j \tp \rho_R(f^{ij})u.$$
Then we have:
\begin{equation}\label{hsing2}
\begin{split}
e * (1 \tp u) &= (1 \tp 1) \tp_H (\psi(u) - \sum_{k=1}^{2N} \d_k
\tp \rho_R(\d^k)u)\\
&+ \text{ terms in } (1 \tp H) \tp_H (\dd\tp (\kk + \rho_R(\spd)) u \\
&+ \kk \tp (\kk + \rho_R(\spd + \dd')) u).
\end{split}
\end{equation}
On the other hand, if $v \in \sing \V(R)$, then \eqref{Hactionsing}
implies
\begin{equation}\label{hsing3}
\begin{split}
e * v = \sum_{i,j=1}^{2N} (1 \tp \d_i\d_j) & \tp_H \rho_R(f^{ij}) v\\
+ \mbox{ terms in } (1 & \tp \fil^1 H) \tp_H \V(R).
\end{split}
\end{equation}

\begin{lemma}\label{lcoefficients}
If\/ $v= \sum_I \d^{(I)} \tp v_I$, then
\begin{equation}
\begin{split}
e*v = \sum_I (1 \tp \d^{(I)}) & \tp_H \psi(v_I)\\
+ \text{\rm{ terms in }} (1 & \tp \d^{(I)}H) \tp_H (\fil^1 H \tp (\kk +
\rho_R(\spd + \dd')) v_I).
\end{split}
\end{equation}
In particular, the coefficient multiplying\/ $1 \tp \d^{(I)}$ equals\/
$\psi(v_I)$ modulo\/ $\fil^1 \V(R)$.
\end{lemma}
\begin{proof}
This follows from \eqref{hsing2}.
\end{proof}

Recall that every element $v$ of $\V(R)$ can be uniquely expressed
in the form
$$v = \sum_I \d^{(I)} \tp v_I.$$
\begin{definition}
The non-zero elements $v_I$ in the above expression are called {\em
coefficients} of $v\in \V(R)$. For a submodule $M \subset \V(R)$, we
denote by $\coef M$ the subspace of $R$ spanned over $\kk$ by all
coefficients of elements from $M$.
\end{definition}

\begin{lemma}
Let\/ $R$ be an irreducible\/ $\dd'\oplus\spd$-module. 
Then, for any non-zero proper\/ $\Hd$-submodule\/ $M$ of\/ $\V(R)$, we have\/
$\coef M = R$.
\end{lemma}
\begin{proof}
Pick a non-zero element $v = \sum_I \d^{(I)} \tp v_I$ contained in
$M$. Then \leref{lcoefficients} shows that $M$ contains an element
congruent to $\psi(v_I)$ modulo $\fil^1 \V(R)$, thus coefficients of
$\psi(v_I)$ all lie in $\coef M$ for all $I$. This proves that
$\rho_R(\spd)\coef M \subset \coef M$. Similarly, one can write
\begin{equation}
\begin{split}
e*v = \sum_I (1 \tp \d^{(I)}) \tp_H (\psi(v_I) - \sum_{k}
\d_k \tp (\rho_R(& \d^k)v_I)) \\
+ \mbox{ terms in } (1 \tp \d^{(I)} H) \tp_H (\dd & \tp
\rho_R(\spd + \kk))v_I\\
 + \kk & \tp (\kk + \rho_R(\spd + \dd'))v_I),
\end{split}
\end{equation}
showing that $\rho_R(\d^k)v_I \in \coef M$ for all $I$, $k=1, \dots,
2N$. Thus, $\rho_R(\dd')\coef M \subset \coef M$, as the central element $c$ acts via multiplication by a scalar. Then $\coef M$ is a
non-zero $\dd' \oplus \spd$-submodule of $R$. Irreducibility
of $R$ now gives $\coef M = R$ as soon as $M \neq 0$.
\end{proof}

\begin{corollary}\label{psiinm}
Let $M$ be a non-zero proper $\Hd$-submodule of $\V(R)$. Then for
every $u \in R$ there is some element in $M$ coinciding with $\psi(u)$
modulo $\fil^1 \V(R)$.
\end{corollary}
\begin{proof}
As $\coef M = R$, it is enough to prove the statement for
coefficients of elements lying in $M$. In order to do so, pick $v
\in M$ and proceed as in \leref{lcoefficients}. As $M$ is a
$\Hd$-submodule of $\V(R)$, the coefficient multiplying $1 \tp
\d^{(I)}$ still lies in $M$, and equals $\psi(v_I)$ modulo $\fil^1
\V(R)$.
\end{proof}

The following irreducibility criterion is proved as in \cite[Section 7.2]{BDK2}.

\begin{theorem}\label{irrcriterion}
Let\/ $R = \Pi' \boxtimes U$ be an irreducible\/ $\dd' \oplus \spd$-module. If\/ $U$ is neither trivial nor isomorphic to one of the fundamental representations\/ $R(\pi_i)$, $i=1, \dots, N$, then the\/
$\Hd$-tensor module\/ $\V(\Pi, U)$ is irreducible.
\end{theorem}
Notice that the converse certainly holds if the $\P$-action factors through $\H$: in this case, if $U$ is either trivial or a fundamental representation, then $\V(\Pi, U) = \V(\pi^* \Pi, U) = \V(\Pi', U)$ appears in a twisted conformally symplectic pseudo de Rham complex (see \thref{tmodhd}) and it is reducible, as the image or the kernel of each differential provides a non-trivial submodule.

Our final goal is to extend this fact to tensor modules in which $c$ acts nontrivially and to provide a complete classification of submodules. All this will be accomplished in next section by providing a description of all singular vectors.

\section{Computation of singular vectors}\label{sksing}

\begin{proposition}\label{pboundeddegree}
Let $v \in \V(R)$ be a singular vector contained in a non-zero
proper $\Hd$-submodule $M$, and assume that the $\spd$-action on $R$
is non-trivial. Then $v$ is of degree at most two.
\end{proposition}
\begin{proof}
Write $v=\sum_I \d^{(I)} \tp v_I$. Then \leref{lcoefficients},
together with \eqref{hsing3}, show that $\psi(v_I) =0$ whenever
$|I|\geq 2$. As the $\spd$-action on $R$ is non-trivial, from
$\psi(v_I) = 0$ follows $v_I=0$.
\end{proof}

Our next step is to describe singular vectors and submodules contained in
tensor modules of the form $\V(\Pi', U)$, where $U=R(\pi_i)$ for some $i = 0, \dots, N$ and $\Pi'$ is an irreducible $\dd'$-module. Our final result
states that when the $\P$-action factors through $\H$, then such a module contains singular vectors if and only if
it is a twist of a module contained in the conformally symplectic pseudo de Rham
complex, and that in such cases singular vectors are 
described in terms of the differential maps. The case where the action of $\P$ does not factor through $\H$ can then also be deduced easily.

The case $U \simeq \kk$ is of a somewhat different flavor, so we
treat it separately.

\subsection{Singular vectors in $\VPiprimezero$}\label{shsing}
First, we find all singular vectors in $\VPiprimezero$.
\begin{lemma}\label{spdtrivial}
A vector in $\VPiprimezero$ is singular if and only if it lies in $\fil^1 \VPiprimezero$.
\end{lemma}
\begin{proof}
Let us first prove $\sing \VPizero \subset \fil^1 \VPiprimezero$.
The $\spd$-action is trivial, so \eqref{hsing2} can be rewritten as
\begin{equation}\label{spdtrivial1}
\begin{split}
e * (1 \tp u) & = \sum_{k=1}^{2N} (\bar \d_k \tp 1) \tp_H (\d^k \tp u - 1 \tp
\rho_R(\d^k)u)\\
& \qquad + (1 \tp 1) \tp_H 1 \tp \rho_R(c)u,
\end{split}
\end{equation}
and also as
\begin{equation}\label{spdtrivial2}
\begin{split}
e * (1 \tp u) = - \sum_{k} (1 \tp \d_k) & \tp_H (\d^k \tp u -
1 \tp
\rho_R(\d^k)u)\\
+ \mbox{ terms in } (1 & \tp 1) \tp_H \fil^1 \VPiprimezero.
\end{split}
\end{equation}
Let $v = \sum_I \d^{(I)} \tp v_I$ be a singular vector in
$\VPiprimezero$, and assume there is some $I$, $|I|>1$, such that $v_I \neq
0$. If $n$ is the maximal value of $|I|$ for such $I$, choose among
all $I = (i_1, \dots, i_{2N})$ with $|I| = n$, one with the highest
value of $i_1$. Then the coefficient multiplying $1 \tp
\d_1 \d^{(I)}$ in $e * v$ equals $\d^1 \tp v_I - 1 \tp
\rho_R(\d^1)v_I$. As $v$ is singular, this must vanish if $|I|
>1$, a contradiction with $v_I \neq 0$.

The reverse inclusion $\sing \VPiprimezero \supset \fil^1 \VPiprimezero$ is now easy. If $v \in \fil^1 \VPiprimezero$, plugging it into \eqref{spdtrivial1} and left-straightening the resulting expression only yields terms in $(\fil^2 H \tp \kk) \tp_H \VPiprimezero$.
\end{proof}

\begin{proposition}\label{irreduciblecnonzero}
The\/ $\Hd$-module\/
$\V(R) = \VPiprimezero$ is reducible if and only if the action of $c \in \dd'$ on $\Pi'$ is trivial.
\end{proposition}
\begin{proof}
Let $M$ be a non-zero proper submodule of $\VPiprimezero$. We know $M$ that must contain non-constant singular vectors and that $\gr^1 \VPiprimezero \simeq \Pi'_{\pi^*\chi/2} \boxtimes \dd$ is $\dd' \oplus \spd$-irreducible. Thus, for every choice of $\d \in \dd, u \in  
1 \tp R$, $M$ must contain some element whose degree one term is $\d \otimes u$.

Consequently, as $M\subset \VPiprimezero$ is an $H$-submodule, $M$ must be of finite codimension in $\VPiprimezero$, hence $\VPiprimezero/M$ is a torsion $H$-module. However, the action of a pseudoalgebra on torsion elements is trivial, so $e*\VPiprimezero = e*M \subset (H \tp H) \tp_H M$. Direct inspection of \eqref{spdtrivial1} now shows that $1 \tp \rho_R(c) u$ belongs to $M$ for all $u \in U$. Then, if the action of $u$ is non-trivial, $M$ must contain non-zero constant vectors, which is a contradiction.

If instead the action of $c$ is trivial, then vectors $\d \tp
u - 1 \tp \rho_R(\d)u,$ where $\d \in \dd$, $u \in R,$ all belong to $M$, and they indeed span over $H$ a proper $\Hd$-submodule of $\VPiprimezero$.
\end{proof}

\begin{remark}
\prref{irreduciblecnonzero} shows, by providing a counterexample, that requiring $c\in \dd'$ to act trivially on $R$ in \prref{constiff} is a mandatory hypothesis.
\end{remark}

\subsection{Non-trivial submodules of $\VPiprimezero$}\label{ssubvpik}

We are only interested in the case where the action of $c \in \dd'$ on $\Pi'$ is trivial, as we have seen that $\VPiprimezero$ is otherwise irreducible. We thus focus on tensor modules $\VPizero:= \V(\pi^*\Pi, \kk) = \VPiprimezero$, where $\Pi$ is a $\dd$-module, see \deref{dvmodw}(iii).

First of all, recall that
\begin{equation*}
\begin{split}
\Omega^{2N}_{\Pi}(\dd)  = J^{2N}_{\Pi}(\dd) & \simeq \T(\Pi_{-N\chi}, \kk) = \V(\Pi_{-N\chi - \phi}, \kk),\\
\Omega^{2N-1}_{\Pi}(\dd)  = J^{2N-1}_{\Pi}(\dd) & \simeq \T(\Pi_{-(2N-1)\chi/2}, R(\pi_1)) = \V(\Pi_{-(2N-1)\chi/2 - \phi}, R(\pi_1)),
\end{split}
\end{equation*}
so that
\begin{equation*}
\begin{split}
\VPizero & \simeq J^{2N}_{\Pi_{N\chi + \phi}}(\dd) = \Omega^{2N}_{\Pi_{N\chi + \phi}}(\dd),\\ 
\VPione &  \simeq J^{2N-1}_{\Pi_{N\chi + \phi}}(\dd) = \Omega^{2N-1}_{\Pi_{N\chi + \phi}}(\dd).
\end{split}
\end{equation*}
Recall (cf.\ \cite{BDK}) that the image of the de Rham differential $\di_{\Pi_{N\chi + \phi}}\colon\Omega^{2N-1}_{\Pi_{N\chi + \phi}}(\dd) \to \Omega^{2N}_{\Pi_{N\chi + \phi}}(\dd)$ is a maximal $H$-submodule which has finite codimension as a $\kk$-vector subspace. This means that the image of the map
$$\dizerostar\colon \VPione \to \VPizero,$$ is a maximal $\Hd$-submodule of $\VPizero$, of finite $\kk$-codimension. Notice that $\dizerostar \fil^0 \VPione$ is a $\dd\oplus \spd$-submodule of $\sing \VPizero$ which has a non-zero projection of $\gr^1 \VPizero \simeq \Pi_{\chi/2} \boxtimes R(\pi_1)$, so that it coincides with the $(\dd\oplus\spd)$-stable complement of $\fil^0 \VPizero$ in $\sing \VPizero$, which is uniquely determined as it coincides with the $R(\pi_1)$-isotypical component of $\fil^1 \VPizero$.

\begin{proposition}
The unique non-zero proper $\Hd-$submodule of $\VPizero$ is $\imdizerostar$.
\end{proposition}
\begin{proof}
We have proved in the previous section that
$$\sing \VPizero =\fil^1 \VPizero \simeq \fil^0 \VPizero \oplus \gr^1 \VPizero$$ is a direct sum of two non-isomorphic homogeneous irreducible $\dd \oplus \spd$-submodules. If $M \subset \VPizero$ is a non-zero proper submodule, then it must contain non-zero homogeneous singular vectors, which cannot be constant. Thus it contains all of $\dizerostar \fil^0 \VPione$, which spans $\imdizerostar$ over $H$. Thus $M$ contains the maximal submodule $\imdizerostar$, hence it coincides with it. 
\end{proof}

We will now separately classify singular vectors of degree one and
two in tensor modules of the form $\VPiprimen$, for $n= 1, \dots, N$.

\subsection{Singular vectors of degree one in $\VPiprimen, 1 \leq n \leq N$}

We will proceed as follows: we will first treat the case $\VPiprimen$, where the action of $c \in \dd'$ is trivial; we will then show that the description of singular vectors of degree one stays unchanged when the action of $c$ is via multiplication by a non-zero scalar.

\subsubsection{Singular vectors of degree one, graded case}

Our setting is the following: $\Pi$ and $U$ are irreducible representations of $\dd$ and $\spd$ respectively, $R = \pi^*\Pi \boxtimes U$ is a $\dd' \oplus \spd$ module and $V = \V(R) = \V(\Pi, U)$, as in \deref{dvmodw}(iii), is a reducible $\Hd$ tensor module. Furthermore, $U = R(\pi_n)$ for some $1 \leq n \leq N$. We are also
given a non-zero proper submodule $M \subset V$. We first look for singular vectors of degree one, i.e., of the form
$$S = \sum_{i=1}^{2N} \d_i \tp s_i + 1 \tp s$$
that are contained in $M$.

\begin{lemma}
A non-zero singular vector of degree one contained in\/ $M\subset V$
is uniquely determined by its degree one part.
\end{lemma}
\begin{proof}
If $s$ and $s'$ are two such vectors agreeing in degree one, then
$s-s'$ is a singular vector contained in $M\cap (\kk\tp R) = \{0\}$.
\end{proof}

\begin{proposition}\label{psingunique}
Let\/ $s, s'$ be non-zero singular vectors of degree one contained in
(possibly distinct) non-zero proper submodules\/ $M, M'$ of a tensor module\/ $\V(R) =
\VPin$ as above. If\/ $s = s' \mod \fil^0 \V(R)$,
then\/ $s$ and\/ $s'$ coincide.
\end{proposition}
\begin{proof}
We already know that each element in the irreducible $\dd \oplus \spd$-submodule $\fil^0 \V(R)\simeq \Pi \boxtimes
 R(\pi_n)$ is a singular vector. Moreover, $\gr^1 \V(R)$ decomposes in the direct sum of (at most)
three irreducible components isomorphic to $\Pi_{\chi/2} \boxtimes R(\pi_{n-1})$,
$\Pi_{\chi/2} \boxtimes R(\pi_{n+1})$, and $\Pi_{\chi/2} \boxtimes R(\pi_n + \pi_1)$, respectively. As no such representation, viewed as an $\spd$-module, has a non-zero $R(\pi_n)$-isotypical component, $\gr^1 \V(R)$ splits off canonically and uniquely as a (homogeneous) complement of $\fil^0 \V(R)$ in $\fil^1 \V(R)$. If a singular vector $s \in \fil^1 \V(R)$ has a non-zero projection to $\gr^1 \V(R)$ and is not contained in this canonical complement, then each $\Hd$-submodule of $\V(R)$ containing $s$ will also contain constant singular vectors, and will fail to be proper.
\end{proof}

Recall that there are $\Hd$-module isomorphisms
\begin{equation*}
\begin{split}
\Omega^n_\Pi(\dd)/I^n_\Pi(\dd) & \simeq \T(\Pi_{-n\chi/2}, R(\pi_n)) = \V(\Pi_{-\phi - n \chi/2}, R(\pi_n)),\\
J^{2N-n}_\Pi(\dd) & \simeq \T(\Pi_{-(2N-n)\chi/2}, R(\pi_n)) = 
\V(\Pi_{-\phi -(2N-n)\chi/2}, R(\pi_n)),
\end{split}
\end{equation*}
so that
$$\VPin \simeq \Omega^n_{\Pi_{\phi + n\chi/2}}(\dd)/I^n_{\Pi_{\phi + n\chi/2}}(\dd) \simeq J^{2N-n}_{\Pi_{\phi + (2N-n)\chi/2}}(\dd).$$

The conformally symplectic pseudo de Rham differentials
\begin{equation*}
\begin{split}
\din\colon& \Omega_{\Pi_{\phi + n\chi/2}}^{n-1}(\dd)/I_{\Pi_{\phi + n\chi/2}}^{n-1}(\dd) \to \Omega_{\Pi_{\phi + n\chi/2}}^{n}(\dd)/I_{\Pi_{\phi + n\chi/2}}^{n}(\dd),\\
\dinstar\colon& J_{\Pi_{\phi + (2N-n)\chi/2}}^{2N-n-1}(\dd) \to J_{\Pi_{\phi + (2N-n)\chi/2}}^{2N-n}(\dd),
\end{split}
\end{equation*}
that we denote as in \eqref{domd8}, \eqref{domd9},
then identify to corresponding $\Hd$-tensor module homomorphisms 
\begin{equation*}
\begin{split}
\din\colon& \VPinminusone \to \VPin,\qquad n = 1, \dots, N,\\
\dinstar\colon& \VPinplusone \to \VPin,\qquad n = 0, \dots, N-1.
\end{split}
\end{equation*}
We similarly obtain an $\Hd$-module homomorphism 
$$\DiPi\colon\VPichiN \to \VPiN.$$

\begin{theorem}\label{onlysymplectic}
One has{\rm:}
\begin{align*}
\sing \fil^1 \VPin &= \fil^0 \VPin + \din \fil^0 \VPinminusone 
\\
&+ \dinstar \fil^0 \VPinplusone, \qquad 1 \leq n < N,
\intertext{and}
\sing \fil^1 \VPiN &= \fil^0 \VPiN + \diN \fil^0 \VPiNminusone .
\end{align*}
\end{theorem}
\begin{proof}
By the proof of \prref{psingunique}, the $\dd \oplus \spd$-submodule $\fil^1 \VPin$ is isomorphic to the direct sum
$$\fil^0 \VPin\,\, \oplus \,\,\Pi_{\chi/2} \boxtimes R(\pi_{n-1})\,\,\oplus \,\,\Pi_{\chi/2} \boxtimes R(\pi_{n+1})\,\,\oplus\,\,\Pi_{\chi/2} \boxtimes R(\pi_n + \pi_1),$$
where non-zero vectors in the last three summands have degree one.

The $\spd$-module $R(\pi_n + \pi_1)$ satisfies the irreducibility
criterion stated in \thref{irrcriterion}, and its dimension is
larger than $\dim R(\pi_n)$. We can therefore proceed as in
\cite[Lemma 7.8]{BDK1} to conclude that no singular vectors of degree one will have a non-zero projection to this summand.

However, by exactness of the conformally symplectic pseudo de Rham
complex, the differentials $\din, n>1,$ map $\fil^0 \VPinminusone$ to non-zero (otherwise $\din$ would vanish) non-constant (otherwise $\din$ would be surjective) irreducible $\dd \oplus \spd$-summands of singular vectors, isomorphic to $\Pi_{\chi/2}\boxtimes R(\pi_{n-1})$, and a similar argument applies to $\dinstar$. These summand must lie in $\fil^1 \VPin$, due to \reref{degreeone}.
Now \prref{psingunique} shows that those are the
only singular vectors of degree one, up to costant elements.
\end{proof}

\subsubsection{Singular vectors of degree one, non-graded case}\label{ssnongraded}

Our setting is now the following: we have a $\Hd$-tensor module $\VPiprimen, 1 \leq n \leq N,$ on which $c \in \dd'$ acts nontrivially; here $n$ is fixed once and for all. Consequently, $\chi = 0$, $\omega = \di \bla$, and $\dd'$ --- as a central extension of $\dd$ --- may be trivialized as $\dd' = \ddbla \oplus \kk c$, where $\ddbla$ is spanned by elements $\d_{\bla} := \d + \bla(\d), \d \in \dd$. Then $\Pi'$ is isomorphic to $\Pi \boxtimes \kk_\lambda$, where $\Pi$ is the restriction of $\Pi'$ to $\ddbla$, and $\kk_\lambda$ is the $1$-dimensional representation of the abelian Lie algebra $\kk c$ on which $c$ acts via multiplication by $\lambda$. We may rewrite \eqref{hsing1} as follows
\begin{equation}\label{hsinglambda}
\begin{split}
e * v = \sum_{i,j=1}^{2N}  & \, (\d_i \d_j \tp 1) \tp_H
(1 \tp \rho_R (f^{ij})v)\\
- \sum_{k=1}^{2N} & \, (\d_k \tp 1) \tp_H \bigl(1 \tp \rho_R (\d^k_{\bla} + \adsp
\d^k) v - \d^k \tp v\bigr)\\
+ & ((\ell + 1) \tp 1) \tp_H (1 \tp \lambda v),\qquad v \in R,
\end{split}
\end{equation}
where $\ell := \sum_k \bla(\d^k)\d_k$.

We now identify all $\Pi \boxtimes \kk_\lambda$ as vector spaces, and consequently all $\V(\Pi \boxtimes \kk_\lambda, R(\pi_n))$ with each other. Notice that $\V(\Pi \boxtimes \kk_0, R(\pi_n))$ is nothing but $\V(\Pi, R(\pi_n)):= \V(\pi^*\Pi, R(\pi_n))$.

\begin{proposition}\label{nongradedsingular}
Let $V = \V(\Pi \boxtimes \kk_\lambda, R(\pi_n))$.
The intersection $\sing V \cap \fil^1 V$ is independent of $\lambda$.
\end{proposition}
\begin{proof}
We show that each singular vector of degree at most one in $\V(\Pi \boxtimes \kk_0, R(\pi_n))$ is also a singular vector in $\V(\Pi \boxtimes \kk_\lambda, R(\pi_n))$, as the converse is completely analogous.

The first two terms in the right-hand side of \eqref{hsinglambda} describe $e*v$ when $\lambda = 0$. Assume $s$ is a singular vector of degree one for $\V(\Pi \boxtimes \kk_0, R(\pi_n))$. We need to compute $e*s$ in $\V(\Pi \boxtimes \kk_\lambda, R(\pi_n))$ and check that it lies in $(\fil^2 H \tp H) \tp_H V$.

However, the contribution in $e*s$ from the first two summands in the right-hand side of 
\eqref{hsinglambda} belongs to $(\fil^2 H \tp H) \tp_H V$ as $s$ is a singular vector for the $\lambda = 0$ case; furthermore, the third summand left-straightens to $(\fil^2 H \tp H) \tp_H V$ because $s$ is of degree one.
\end{proof}

\subsection{Singular vectors of degree two}

Throughout this section, we will focus on an $\Hd$-tensor module of the form $\V(R) = \VPiprimen, n = 1, \dots, N$. Recall the definition of $\psi(u)$ given
in \seref{shirten}. Then, if
$$S= \sum_{i,j=1}^{2N} \d_i \d_j \tp s^{ij} + \sum_{k=1}^{2N} \d_k \tp s^k + 1 \tp
s$$
is a vector of degree two lying in $\sing \V(R)$, one has:
\begin{lemma}\label{spactiondeg2}
$f^{\alpha\beta}(S) = \psi(s^{\alpha\beta}) \mod \fil^1 \V(R)$.
\end{lemma}
\begin{proof}
Use \eqref{hsing2} to compute $e*S$ and compare it with
\eqref{hsing3}.
\end{proof}
As we are assuming $S$ to be of degree two, some $s^{ij}$ must be non-zero. Then $\psi(s^{ij})$ is also non-zero, as the action of $\spd$ on $R$ is non-trivial, hence $f^{ij}(S) = \psi(s^{ij}) \mod \fil^1 \V(R)$ is non-zero.

This shows, in one shot, that the action of $\spd$ on each singular vector $S$ of degree two is non-trivial, and that for some $0 \neq u\in R$ there exists a singular vector
coinciding with $\psi(u)$ modulo $\fil^1 \V(R)$, as the $\spd$-action preserves singular vectors, hence $f^{ij}(S)$ is singular too.

\begin{lemma}\label{lphiproperties}
If $u \in R$, let $\phi(u)$ denote a singular vector, if there exists one, coinciding with $\psi(u) \mod \fil^1 \V(R)$. Then $f^{\alpha \beta} (\phi(u))$ and $\widehat\d (\phi(u))$ are also singular vectors, which coincide, respectively, with $\psi(f^{\alpha \beta}(u))$ and $\psi(\widehat \d.u + \chi(\d)u) $ modulo $\fil^1 \V(R)$.
\end{lemma}
\begin{proof}
The claims follow from \leref{spactiondeg2} and \leref{legraction}, along with the definition of $\psi(u)$.
\end{proof}

\begin{lemma}\label{phiproperties}
For every $u \in R$ there exists some singular vector $\phi(u)$ coinciding with $\psi(u)$ modulo $\fil^1 \V(R)$. The corresponding map $\phi: R \to \gr^2 \V(R)$ is a linear homomorphism and satisfies:
\begin{align}
\label{fabphi}
f^{\alpha\beta}(\phi(u)) & = \phi(f^{\alpha\beta}(u)),\\
\label{fabs}
f^{\alpha\beta}(S) & = \phi(s^{\alpha\beta}),\\
\label{dphi} \widehat\d.\phi(u) & = \phi(\widehat\d.u+ \chi(\d)u),\\
\label{cphi} c.	\phi(u) & = \phi(c.u).
\end{align}
Thus, $\phi$ is a well defined injective homomorphism which commutes with the $\spd$-action.
\end{lemma}
\begin{proof}
The space of elements $u \in R$ for which we can locate a singular vector equal to $\psi(u)$ modulo $\fil^1 \V(R)$ is non-zero and $\dd' \oplus \spd$-stable. By irreducibility of $R$, it must be all of $R$. Injectivity follows from $\eqref{fabs}$, whereas the other properties have already been proved.
\end{proof}

\begin{corollary}
The space of singular vectors of degree two does not
contain trivial $\spd$-summands.
\end{corollary}
\begin{proof}
If $S$ is a homogeneous singular vector of degree two lying in a
trivial summand, then $0 = f^{\alpha\beta}(S) =
\phi(s^{\alpha\beta})$ for all $\alpha, \beta$. But $\phi$ is
injective, hence $s^{\alpha\beta} = 0$, a contradiction with $S$
being of degree two, hence $S=0$.
\end{proof}

We are now ready to classify singular vectors of degree two. As before, we will first proceed with tensor modules on which the action of $c \in \dd'$ is trivial, and then extend our results to the general case.

\subsubsection{Singular vectors of degree two, graded case}

Our setting is now as follows: $\Pi$ is a $\dd$-module, and $\VPin:= \V(\pi^* \Pi, R(\pi_n)) = \VPiprimen$ denotes the $\Hd$-tensor module as in \deref{dvmodw}(iii).

\begin{proposition}\label{degreetwoisirreducible}
Non-zero singular vectors of degree two in $\VPin$ project to a unique $\dd \oplus \spd$-irreducible summand in $\gr^2 \VPin$, which is isomorphic to $\Pi_{\chi} \boxtimes R(\pi_n)$.
Moreover, the homomorphism $\dindinminusonestar$  maps injectively $\fil^0 \VPichin$ to an irreducible summand of singular vectors of degree two in $\VPin$.

In particular, the space of non-zero homogeneous singular vectors of degree two in $\VPin$ coincides with $\dindinminusonestar \fil^0 \VPichin$.
\end{proposition}
\begin{proof}
The first claim is clear.

As for the second one, an $\Hd$-homomorphism maps singular vectors to singular vectors, and its restriction to an irreducible $\dd \oplus \spd$-summand is either $0$ or an isomorphism. Recall that, by exactness of the conformally symplectic pseudo de Rham complex, $\din \prevdin$ equals $0$, and $\ker \din = \Im\,\, \prevdin$.
Now, $\prevdin$ must be injective on 
$\dinminusonestar \fil^0 \VPichin$. If it were not, this homogeneous summand of degree $1$ singular vectors would have to lie in the image of $\prevdin$, hence it would be obtained by applying $\prevdin$ to constant singular vectors. However $\prevdin \fil^0 \VPichinminustwo$ is  isomorphic to $\Pi_{\chi} \boxtimes R(\pi_{n-2})$ as $\dd \oplus \spd$-module, yielding a contradiction.
We conclude that $\dindinminusonestar \fil^0 \VPichin$ is a homogeneous summand of degree $2$ singular vectors, which is $\dd\oplus \spd$-isomorphic to $\Pi_\chi \boxtimes R(\pi_n)$.
\end{proof}

\begin{remark}
Note that if $\chi=0$ a proper submodule $M$ of a tensor module $\V(R)$ can
in principle contain non homogeneous singular vectors of degree two. Indeed,
homogeneous singular vectors of degree zero and two span in such a case isomorphic $\dd \oplus \spd$-submodules of $\V(R)$, so that suitable linear combinations will yield possibly non homogeneous singular vectors.

We will see below that singular vectors of degree two contained in a proper submodule are always homogenous. We will show this by constructing, when $\chi = 0$, non-homogeneous automorphisms of reducible tensor modules.
\end{remark}

\subsubsection{Singular vectors of degree two, non-graded case}\label{lambdavectors}

Our setting is as above in \seref{ssnongraded} and we identify, again, all $\V(\Pi \boxtimes \kk_\lambda, R(\pi_n))$ with each other. 

\begin{proposition}
Let $S = \sum_{ij} \d_i \d_j \otimes f^{ij}(u) + \sum_k \d_k \otimes s^k + 1 \otimes s$.
Then $S$ is singular in $\V(\Pi \boxtimes \kk_0, R(\pi_n))$ if and only if $S_\lambda := S + \lambda\, \ell \otimes u$ is singular in $V = \V(\Pi \boxtimes \kk_\lambda, R(\pi_n))$.
\end{proposition}
\begin{proof}
Proceed as in \prref{nongradedsingular}. The only two terms that do not obviously lie in $(\fil^2 H \tp H) \tp_H V$ cancel with each other modulo $(\fil^2 H \tp H) \tp_H V$.
\end{proof}

If $W \subset \V(R)$ is a subset of singular vectors of degree two, it makes sense to denote by $W_\lambda$ the set of all $S_\lambda$, where $S \in W$.

\subsection{Classification of singular vectors}
As a result of the above discussion, we obtain the following 
classification of singular vectors.

\begin{theorem}\label{singclassification}
Let $V= \V(\Pi', U)$ be a tensor module over $\Hd$, where $\Pi' \boxtimes U$ is an irreducible $\dd' \oplus \spd$-module in which $c\in \dd'$ acts via multiplication by $\lambda \in \kk$. When $\lambda \neq 0$, denote by $\Pi$ the unique $\dd$-module such that $\Pi' \simeq \Pi \boxtimes \kk_\lambda$ as $\dd' = \ddbla \oplus \kk c$-modules; when $\lambda = 0$, set $\Pi$ to satisfy $\Pi' = \pi^* \Pi$.
\begin{itemize}
\item If $U \simeq R(\pi_0) = \kk$, then
$$\sing V = \fil^1 V = \fil^0 V \oplus \dizerostar  \fil^0 \VPione;$$
\item if $U \simeq R(\pi_n), 1
\leq n < N,$ then
\begin{equation*}
\begin{split}
\sing V  = \fil^0 V & \oplus \,\din \fil^0\VPinminusone\\
& \oplus\dinstar \fil^0\VPinplusone\\
& \oplus \left(\dindinminusonestar  \fil^0 \VPichin\right)_\lambda;
\end{split}
\end{equation*}
\item if $U \simeq R(\pi_N)$ then
\begin{equation*}
\begin{split}
\sing V = \fil^0 V & \oplus \diN  \fil^0 \VPiNminusone\\
& \oplus \left(\diN \diNminusonestar \fil^0 \VPichiN\right)_\lambda;
\end{split}
\end{equation*}
\item $\sing V = \fil^0 V$ in all other cases.
\end{itemize}
\end{theorem}

\begin{remark}\label{uptomultiplication}
It is not difficult to show that singular vectors of degree two may also be obtained as $\dinstar \dinplusone \fil^0 \VPichin$ in the tensor module $\VPin, 1 \leq n < N,$ and as $\DiPi \fil^0 \VPichiN$ in $\VPichiN$.

This shows that $\dindinminusonestar$ and $\dinstar \dinplusone$ (resp. $\diNdiNminusonestar$ and $\DiPi$) coincide up to multiplication by a non-zero scalar.
\end{remark}

\section{Submodules of tensor modules, graded case}\label{ssubmtm}

In this section we will show that the only non-zero proper submodules of tensor modules appearing in a twisted conformally symplectic pseudo de Rham complex are either the image of a differential or the sum/intersection of two such images. Notice that the action of $\P$ on all such modules factors through $\H$, so that it makes sense to talk of {\em homogeneous} singular vectors. Throughout this section, $\Pi$ is an irreducible finite-dimensional $\dd$-module and $\V(\Pi, U)$ denotes the $\Hd$-tensor module $\V(\pi^* \Pi, U)$.

\subsection{Automorphisms of tensor modules}

Let $f\colon V \to V$ be a homomorphism of tensor modules. Then $f$
maps constant vectors in $V$, which are singular, to $\sing V$.
If $f$ is non-zero, then $f|_{\fil^0 V}$ must be an
isomorphism. This is typically possible only when $f$ is a
multiple of the identity. However, the case of reducible tensor
modules over $\Hd$ is a remarkable exception, when $\chi=0$.

\begin{theorem}
Let $V = \VPin$, $1 \leq n \leq N$, be an $\Hdzero$-module. Then each
$\Hdzero$-endomorphism of $V$ is of the form $f_{\alpha
\beta} = \alpha \id_V + \beta \dindinmochizerostar$ for some choice of $\alpha, \beta \in \kk$.
The endomorphism $f_{\alpha \beta}$ is invertible if
and only if $\alpha \neq 0$.
\end{theorem}
\begin{proof}
The maps $\id_V$ and $\dindinmochizerostar$ are commuting $\Hdzero$-endomorphisms, and so are
their linear combinations. As $\id_V$ is invertible and $\dindinmochizerostar$ is nilpotent, the invertibility claim follows.

As for showing that an $\Hdzero$-homomorphism $f: V \to V$ is necessarily of this form, notice that if $f$ is non-zero, then it
must map $\fil^0 V \simeq \Pi \boxtimes R(\pi_n)$ to an
irreducible $\dd \oplus \spd$-component of $\sing V$. However, since $\chi=0$, the constant and the degree two summands are isomorphic.

Composing $f$ with projections to components of degree zero and
two yields $\dd \oplus \spd$-homomorphisms that must be scalar
multiples of $\id$ and $\dd^*$ by irreducibility of $\fil^0 V$ and Schur's Lemma.
\end{proof}

When $\chi \neq 0$, the irreducible $\dd \oplus \spd$-summands of $\sing V$ are pairwise non-isomorphic, hence homogeneous. When $\chi = 0$, however, non-homogeneous singular vectors only occur in some $f_{\alpha\beta}\fil^0 V,$ where $\alpha$ and $\beta$ are both non-zero. In this case, as $V$ is spanned over $H$ by its constant singular vectors, $V=f(V)$ is also $H$-spanned by $f_{\alpha\beta}\fil^0 V$, so that if $M\subset V$ is an $\Hdzero$-submodule containing $f_{\alpha\beta}\fil^0 V, \alpha \neq 0,$ then $M = V$. This proves the following claim
\begin{proposition}
Every proper submodule of a tensor module $\VPin, 1 \leq n \leq N,$ contains the homogenous components of all of its singular vectors.
\end{proposition}

\subsection{Classification of submodules of tensor modules}

Throughout this section, we will make repeated use of the standard homomorphism theorem in the following form: if $f\colon M \to N$ is a morphism of modules, and $M_0 \subset M$ is a submodule, then
$$f(M_0) \simeq (M_0 + \ker f) /\ker f \simeq M_0 / (M_0 \cap \ker f).$$


\begin{lemma}\label{singimd}
If $1 \leq n \leq N$, then
\begin{equation*}
\begin{split}
 \sing \, & \imdin \\ & =  \din \fil^0 \VPinminusone  \oplus \dindinminusonestar \fil^0 \VPichin.
\end{split}
\end{equation*}
\end{lemma}
\begin{proof}
We know that
$$\sing \VPinminusone \supset \fil^0 \VPinminusone \oplus \dinminusonestar \fil^0 \VPichin.$$
Applying $\din$ to both sides and using \leref{singtosing} one obtains
\begin{equation*}
\begin{split}
 \sing \, & \imdin \\ & \supset  \din \fil^0 \VPinminusone  \oplus \dindinminusonestar \fil^0 \VPichin.
\end{split}
\end{equation*}

In order to show that the above inclusion is indeed an equality, recall \thref{singclassification} and observe that $\imdin$ is a proper submodule of $\VPin$, so it cannot
contain constant singular vectors.
Moreover, when $n<N$, $\imdin$ cannot contain $\dinstar \fil^0 \VPinplusone$ as this is not killed by $\nextdin.$
\end{proof}

\begin{lemma}\label{singimd*}
If $0 \leq n < N$, then
\begin{equation*}
\begin{split}
 & \sing \, \imdinstar \\ & =  \dinstar \fil^0 \VPinplusone \oplus \dindinminusonestar \fil^0 \VPichin.
\end{split}
\end{equation*}
\end{lemma}
\begin{proof}
Similar, using \reref{uptomultiplication}.
\end{proof}

\begin{lemma}\label{uniquesingularsummand}
If $0<n\leq N$, then
$$\sing \imdindinminusonestar = \dindinminusonestar \fil^0 \VPichin.$$
\end{lemma}
\begin{proof}
The homomorphism $\dindinminusonestar$ is homogeneous of degree two, so the subspace $\sing \imdindinminusonestar$
can only contain elements of degree $\geq 2$. As $\dindinminusonestar$ is non-zero,
constant vectors are mapped non trivially to singular vectors of
degree two.
\end{proof}

\begin{proposition}\label{dd*irr}
$\imdindinminusonestar$ is irreducible for every $0<n\leq N$.
\end{proposition}
\begin{proof}
As $\imdindinminusonestar = \dindinminusonestar\VPichin$, it
is $H$-spanned by $\dindinminusonestar \fil^0 \VPichin$, which by previous lemma is an irreducible $\dd\oplus\spd$-summand containing all singular vectors. A non-zero submodule of $\imdindinminusonestar$ must then
contain all of its singular vectors, which $H$-span it.
\end{proof}

\begin{corollary}\label{imD}
$\im \DiPi$ is irreducible.
\end{corollary}
\begin{proof}
As $\DiPi$ coincides with $\diNdiNminusonestar$ up to a non-zero multiplicative constant, it follows from \prref{dd*irr}.
\end{proof}

\begin{proposition}\label{maximalp}
If $0 < n < N$, then
\begin{enumerate}
\item $\imdin \cap \imdinstar$ is maximal in $\imdin$;
\item $\imdin \cap \imdinstar$ is maximal in $\imdinstar$;
\item $\imdin$ is maximal in $\imdin + \imdinstar$;
\item $\imdinstar$ is maximal in $\imdin + \imdinstar$.
\end{enumerate}
\end{proposition}
\begin{proof}
One has
\begin{equation*}
\begin{split}
\imdinstar /( & \imdin \cap \imdinstar) = \\
(\imdinstar & + \imdin)/\imdin,
\end{split}
\end{equation*}
by standard homomorphism
properties. Moreover
\begin{equation*}
\begin{split}
\im \nextdin & \dinstar\\
& = \nextdin( \imdin + \imdinstar),
\end{split}
\end{equation*}
and this equals
\begin{equation*}
\begin{split}
(& \imdin + \imdinstar)/\ker \nextdin\\
& = ( \imdin + \imdinstar)/\imdin.
\end{split}
\end{equation*}
By \prref{dd*irr}, $\im \nextdin \dinstar$ is irreducible, hence so are all
the above quotients, proving (2) and (3); (1) and (4) are proved similarly.
\end{proof}

\begin{theorem}\label{submodN}
$\VPin \varsupsetneq \imdin \varsupsetneq \imDiPi$ are the only non-zero
submodules of $\VPiN$.
\end{theorem}
\begin{proof}
First of all,
$$\imDiPi = \ker \nextDiPi, \qquad \ker \DiPiminuschi = \imdiN$$
by exactness of the conformally symplectic pseudo de Rham complex. The submodule $\imDiPi $ is irreducible by \coref{imD} and is therefore minimal in $\VPiN$. Similarly, $\DiPiminuschi \VPiN$ is irreducible.

We show now that $\imdiN$ is maximal in $\VPiN$. Indeed
$$\VPiN/\imdiN = \VPiN/\ker\DiPiminuschi \simeq \im \DiPiminuschi,$$
which is irreducible.

Finally, $\imDiPi$ is a submodule of $\imdiN $ as $\DiPi$ equals $\diNdiNminusonestar$ up to a non-zero multiplicative constant; it is indeed maximal
since 
\begin{equation*}
\begin{split}
\imdiN& /\imDiPi\\
& \simeq \nextDiPi \imdiN\\
& = \nextDiPidiNalt \VPiNminusone,
\end{split}
\end{equation*}
as $\ker \nextDiPi = \imDiPi$.
We have proved so far that the submodule $\imdiN$ is maximal, that $\imDiPi=
\imdiNdiNminusonestar$ is minimal, and that $\imDiPi$ is maximal inside $\imdiN$.

Every proper submodule contains non-constant homogeneous singular vectors, and is therefore
contained in $\imdiN$. If it furthermore contains singular
vectors of degree one, then it coincides with $\imdiN$, as this
is generated by its singular vectors of degree one. If instead it does not contain
singular vectors of degree one, then it lies between $\imDiPi$ and
$\imdiN$, and must then coincide with the latter by maximality.
\end{proof}

\begin{lemma}\label{shapeN-1}
The submodule
$$\diNminusone \VPiNminusoneminusone + \diNminusonestarbla \VPiNminusoneplusone \subset \VPiNNminusone$$ is maximal  and
\begin{equation*}
\begin{split}
\diNminusone &\diNminustwostar \VPichiNminusone\\
& = 
\diNminusone \VPiNminusoneminusone) \cap \diNminusonestarbla \VPiNminusoneplusone.
\end{split}
\end{equation*}
\end{lemma}
\begin{proof}
By \thref{submodN} we know that the quotient
$$\imdiN/\diNminusone \diNminustwostar \VPichiNminusone$$ is irreducible.
However this equals
\begin{equation*}
\begin{split}
 \imdiN/\diN&(\im\diNminusone +\im\diNminusonestarbla ) \\
& = \VPiNminusone/(\im\diNminusone +\im\diNminusonestarbla ),
\end{split}
\end{equation*}
thus proving the first claim. Notice now that
$$\im \diNminusone \diNminustwostarchihalf = \im \diNminusonestarbla \diNchihalf$$
is contained in
$\im \diNminusone \cap \im \diNminusonestarbla,$
and that the quotient
$$\left(\im\diNminusone \cap \im\diNminusonestarbla \right)/\im\diNminusone \diNminustwostarchihalf$$
injects inside
\begin{equation*}
\begin{split}
\im \diNminusonestarbla/&\im \diNminusonestarbla \diNchihalf \\
= &\VPiNminusoneminusone/(\im \diNchihalf + \ker \diNminusonestarbla)\\
= &\VPiNminusoneminusone/\im\diNchihalf,
\end{split}
\end{equation*}
which is irreducible. Therefore
$$\im \diNminusone \cap \im \diNminusonestarbla  = \im\diNminusone \diNminustwostarchihalf,$$
since
$$\im \diNminusone \cap \im \diNminusonestarbla \neq
\im \diNminusonestarbla.$$
\end{proof}

\begin{proposition}\label{pmaximalsubmodule}
If $1 \leq n < N$, then $\imdin+\imdinstar$ is maximal in $\VPin$, and
$$\imdindinminusonestar = \imdin\cap\imdinstar.$$
\end{proposition}
\begin{proof}
By induction on $N-n$, the basis of induction being
\leref{shapeN-1}.

The quotient $$\VPin/(\imdin+\imdinstar)$$ is isomorphic to
$$\dinplusoneminuschi \VPin / \im \dinplusoneminuschi \dinstar$$
which equals
$$\im \dinplusoneminuschi/(\im \dinplusoneminuschi \cap \im \dinplusonestarminuschi$$
by inductive hypothesis.
However this is irreducible by \prref{maximalp}.

As for the second claim, notice that
$$\imdindinminusonestar \subset \imdin\cap\imdinstar$$
is irreducible by \prref{dd*irr}, and is therefore
minimal. Now proceed as in previous lemma: the quotient
$$(\imdin\cap\imdinstar)/\imdindinminusonestar$$
injects inside
\begin{equation*}
\begin{split}
& \imdinstar /\im \dinstar \dinplusone\\
& = \VPinplusone/(\im \dinplusone + \im \dinplusonestar),
\end{split}
\end{equation*}
which is irreducible by inductive
hypothesis. Thus
$\imdin \cap \imdinstar$
equals
$\imdindinminusonestar$, otherwise
$$\imdin \cap \imdinstar =\imdinstar,$$
whence
$$\imdinstar \subset \imdin,$$ a
contradiction.
\end{proof}

\begin{corollary}\label{submodulesofimd}
$\imdin \cap \imdinstar$ is the only non-zero submodule of both $\imdin$ and $\imdinstar$.
\end{corollary}
\begin{proof}
We have seen in \leref{singimd} and \leref{singimd*} that the only non-constant homogeneous singular vectors in both modules are contained in
$$\dindinminusonestar \fil^0 \VPichin),$$
which $H$-span
$\imdindinminusonestar =\imdin \cap \imdinstar.$

Any proper submodule must then contain this submodule, which is however maximal in both $\imdin$ and $\imdinstar$ by \prref{maximalp}.
\end{proof}

\begin{corollary}\label{supmodulesofimd}
If $M$ is a proper submodule of $\VPin$ strictly containing either $\imdin$ or $\imdinstar$ then necessarily
$$M = \imdin + \imdinstar.$$
\end{corollary}
\begin{proof}
Let us treat the case $M \supsetneq \imdinstar$, the other case being totally analogous.
Submodules lying between $\imdinstar$ and $\VPin$ are in one-to-one correspondence with submodules of
$$\VPin/\imdinstar \simeq \dinminusonestarminuschi \VPin$$
so that we may use \coref{submodulesofimd} to show that there is a single possibility for $M$, which must be $\imdin + \imdinstar.$
\end{proof}

\begin{proposition}\label{peitheror}
A proper non-zero submodule of $\VPin$ either contains, or is contained in, either $\imdin$ or $\imdinstar$.
\end{proposition}
\begin{proof}
Assume by contradiction that there is a submodule $M$ which neither contains nor is contained in the above two submodules; in particular, $M$ intersects both modules in $\imdindinminusonestar$. Then $M + \imdin$ (resp. $M + \imdinstar$ strictly contains $\imdin$ (resp. $\imdinstar)$). However
$$M + \imdin = \VPin = M + \imdinstar$$
cannot both hold. If so, then
\begin{equation}
\begin{split}
M/\imdindinminusonestar = M/(M \cap \imdin)\\ \simeq (M+\imdin)/\imdin\\
 \simeq \VPin/\imdin
\end{split}
\end{equation}
and similarly
\begin{equation}
\begin{split}
M/\imdindinminusonestar = M/(M \cap \imdinstar)\\ \simeq (M+\imdinstar)/\imdinstar\\
 \simeq \VPin/\imdinstar.
\end{split}
\end{equation}
This would force
\begin{equation}
\begin{split}
\dinplusoneminuschi \VPin& \simeq
\VPin/\imdin\\
& \simeq
\VPin/\imdinstar\\
& \simeq \im \dinminusonestarminuschi,
\end{split}
\end{equation}
which yields an easy contradiction by comparing the structure of singular vectors.
Let us thus assume that
$$M + \imdin = \imdin + \imdinstar,$$
as the other case is done similarly. Then
$$\dinplusoneminuschi M = \im \dinplusoneminuschi
\dinstar.$$
However, $\ker \dinplusoneminuschi \cap M$ is a submodule of $\imdin \cap M$, which contains 
$\imdin \cap \imdinstar$. Since $M$ does not contain $\imdin$, it must equal $\imdin \cap \imdinstar$.

Then
\begin{equation}
\begin{split}
& M/(\imdin \cap \imdinstar) =\\
& M/(\ker \dinplusoneminuschi \cap M) \simeq  \dinplusoneminuschi M = \im \dinplusoneminuschi \dinstar\\
& \simeq \imdinstar/(\ker \dinplusoneminuschi \cap \imdinstar)\\
& = \imdinstar/(\imdin \cap \imdinstar),
\end{split}
\end{equation}
thus showing $M = \imdinstar$, a contradiction.
\end{proof}

Results of this section may be now summarized as follows:

\begin{theorem}
Let\/ $V=\V(\Pi, U)$ be a tensor module over\/ $\Hd$, where $\Pi$ and\/ $U$ are finite-dimensional irreducible representations of\/ $\dd$ and\/ $\spd$, respectively. A complete
list of\/ $\Hd$-submodules of $V$ is as follows:
\begin{itemize}
\item $V \supset \imdizerostar \supset 0$ if $U \simeq \kk$;

\item $V \supset \imdin+\imdinstar$
$\supset
\imdin$, $\imdinstar
\supset
\imdin\cap\imdinstar
\supset
0$ if
$U \simeq R(\pi_n), 1 \leq n < N$;

\item $V \supset \imdiN \supset \imDiPi \supset 0$ if $U
= R(\pi_N)$;

\item $V \supset 0$ in all other cases.
\end{itemize}
\end{theorem}

\section{Submodules of tensor modules, non-graded case}\label{ssubmtmnongr}\label{ten}

Throughout this section, $\chi = 0, \omega = \di \bla$, so that $\dd' = \ddbla \oplus \kk c$. Let $\Pi'$ be a finite-dimensional irreducible $\dd'$-module. We have already seen that $\sing \VPiprimen \cap \fil^1 \VPiprimen$ contains canonical irreducible $\dd' \oplus \spd$-summands isomorphic to $\Pi' \boxtimes R(\pi_{n-1})$ and $\Pi' \boxtimes R(\pi_{n+1})$. Recall that $\VPiprimen$ was constructed as an induced module, so that we may use its universal property in order to construct $\Hdzeta$-module homomorphisms
\begin{equation*}
\begin{split}
\DIn\colon& \VPiprimenminusone \to \VPiprimen,\qquad n = 1, \dots, N,\\
\DInstar\colon& \VPiprimenplusone \to \VPiprimen,\qquad n = 0, \dots, N-1,
\end{split}
\end{equation*}
that are, by Schur's Lemma, only determined up to multiplication by a nonzero scalar. Notice that when the $\dd'$-module $\Pi'$ factors through a $\dd$-module $\Pi$, i.e., when $c$ acts trivially and $\Pi' = \pi^* \Pi$, then one has $\DIn = \dinchizero, \DInstar = \dinchizerostar$, up to the above undetermined scalar multiple. However, throughout the rest of this section, $c$ will act via multiplication by a scalar $0 \neq \lambda \in \kk$.

We also set $\DImostar = \DIzero = 0$. Then $\DInplusone \DIn = 0$, for all $0 \leq n \leq N-1$ as $\sing \VPiprimenplusone$ contains no $\dd' \oplus \spd$-irreducible summand isomorphic to $\Pi' \boxtimes R(\pi_{n-1})$. Similarly, $\DInminusonestar \DInstar = 0$, for all $0 \leq n \leq N-1$. Recall that 
\begin{equation*}
\begin{split}
\sing \VPiprimen \cap \fil^1 \VPiprimen = \fil^0 & \VPiprimen \oplus \,\DIn \fil^0\VPiprimenminusone\\
& \oplus\DInstar \fil^0\VPiprimenplusone,
\end{split}
\end{equation*}
when $1 \leq n \leq N-1$, whereas in the case $n = N$, the last direct summand is missing.
Notice that when $1 \leq n \leq N$, $\sing \VPiprimen$ also contains an irreducible $\dd'\oplus\spd$-summand isomorphic to $\Pi' \boxtimes R(\pi_n)$ in degree two, which is however not canonically determined as any complement to $\fil^0 \VPiprimen$ in the $\Pi'\boxtimes R(\pi_n)$-isotypical component of $\sing \VPiprimen$ will do. Recall also that, as $\lambda\neq 0$, we have no well-defined grading to chose a {\em homogeneous} canonical complement. However, we denote by $S_\lambda$ the space of singular vectors of degree two as constructed in \seref{lambdavectors} from homogeneous singular vectors.

Before computing the structure of singular vectors in the $\Hdzeta$-modules $\im \DIn$, $\im \DInstar$, we need some preliminary results.

\begin{lemma}
Let $M \subset \V(\Pi', U)$ be a proper $\Hdzeta$-submodule. If $M$ contains a vector $v \in \fil^1 \V(\Pi', U)$, then $v$ is singular.
\end{lemma}
\begin{proof}
We know by \eqref{h1fp} that $\P_2 \fil^1 \V(\Pi', U) = 0$. If $v$ is not singular, then $\P_1 v$ contains a non-zero constant vector, which must lie in $M$. This yields a contradiction as proper submodules do not contain non-zero constant vectors.
\end{proof}

\begin{remark}
The above lemma follows from the general fact that a vector of minimal degree in a proper submodule is always singular. However, we do not need this generality.
\end{remark}

\begin{lemma}\label{decomposegrd}
Let $1 \leq n \leq N-1$. The tensor product of $\spd$-modules $S^d(\dd) \otimes R(\pi_n)$ decomposes as
$$R(d\pi_1 + \pi_n) \oplus R((d-1)\pi_1 + \pi_{n-1}) \oplus R((d-1)\pi_1 + \pi_{n+1}) \oplus R((d-2)\pi_1 + \pi_n)).$$

Similarly, $S^d(\dd) \otimes R(\pi_N)$ decomposes as
$$R(d\pi_1 + \pi_N) \oplus R((d-1)\pi_1 + \pi_{N-1}) \oplus R((d-2)\pi_1 + \pi_N)).$$
\end{lemma}
\begin{proof}
See, e.g, the reference tables at the end of \cite{OV}.
\end{proof}

Recall that the quotient $\P_0/\P_1$ acts on the space of singular vectors. However, the Lie subalgebra $\spd \simeq \p_0\subset \P_0$ is a section of this quotient. It acts on each $\Hdzeta$-module and extends the $\spd$-action on singular vectors. As $\chi = 0$, and we are disregarding the $\dd'$-action, we have already seen that the $\spd$-action on $\gr^d \V(\Pi', U)$ is isomorphic to $\Pi' \otimes (S^d \dd \otimes U)$, where we endow $\Pi'$ with a trivial $\dd$-action. Every homomorphism of $\Hdzeta$-modules is also an $\p_0 \simeq \spd$-module homomorphism.

\begin{proposition}\label{psisingular}
Let $M \subset \V(\Pi', U)$ be a proper $\Hdzeta$-submodule. If $M$ contains a vector $v \in \fil^2 \V(\Pi', U)$ such that $v \equiv \psi(u) \mod \fil^1\V(\Pi', U)$; then $v$ is singular.
\end{proposition}
\begin{proof}
We may assume without loss of generality that the action of $\spd$ on $U$ is non-trivial. Then chose $1 \leq n\leq N$ so that $U \simeq R(\pi_n)$. The vector $v$ lies in $\fil^1 \V(\Pi', R(\pi_n)) + S_\lambda$, which is a $\p_0\simeq \spd$-submodule of $\V(\Pi', R(\pi_n))$ isomorphic to $\Pi' \otimes(R(\pi_n) \oplus (R(\pi_{n-1}) \oplus R(\pi_{n+1}) \oplus R(\pi_1 + \pi_n)) \oplus R(\pi_n))$, where we set $R(\pi_{N+1}) = 0$ if $n=N$. As $M$ is a $\p_0$-submodule of $\V(\Pi', R(\pi_n))$, it must also contain the projection of $v$ to the $R(\pi_1 + \pi_n)$-isotypical component, which only contains vectors of degree one that are non singular. By previous Lemma, this projection must vanish, hence $v$ lies in the sum of the other components, which only contain singular vectors.
\end{proof}

\begin{lemma}
When $1 \leq n \leq N$, the $\spd$-module $\sing \im \DIn$ decomposes as the direct sum of $\DIn \fil^0 \VPiprimenminusone$ and an irreducible summand of singular vectors of degree two.

Similarly, when $1 \leq n \leq N-1$, the $\spd$-module $\sing \im \DInstar$ decomposes as the direct sum of $\DInstar \fil^0 \VPiprimenplusone$ and an irreducible summand of singular vectors of degree two.
\end{lemma}
\begin{proof}
We will only prove the first statement, as the other one is completely analogous.
The map $\DIn: \VPiprimenminusone \to \VPiprimen$ is a homomorphism of $\Hdzeta$-modules, so it commutes with the action of $\p_0$. We already know that
\begin{equation*}\label{singvpin}
\begin{split}
\sing \VPiprimen = \fil^0 & \VPiprimen \oplus \DIn \fil^0 \VPiprimenminusone\\
& \oplus \DInstar \fil^0 \VPiprimenplusone \oplus S_\lambda,
\end{split}
\end{equation*}
where the first summand in the right-hand side is $\fil^0 \VPiprimen$, the next two summands are singular vectors of degree one, and the last summand consists of singular vectors of degree two as constructed in \seref{lambdavectors} starting from homogenous vectors.

The third summand does not lie in $\im \DIn$, as by \leref{decomposegrd} the irreducible representation $R(\pi_{n+1})$ does not show up as a direct summand in any $S^d \dd \otimes R(\pi_{n-1})$. In particular, $\im \DIn$ is a proper submodule of $\VPiprimen$, hence it does not contain constant vectors either. However, it contains $\DIn \fil^0 \V(\Pi', R(\pi_{n-1}))$, i.e., the second summand, by construction.

If $\Pi' \boxtimes U$ is an irreducible representation of $\dd' \oplus \spd$, and the action of $\spd$ is not trivial, we have seen in \coref{psiinm} that every submodule $M \subset \V(\Pi', U)$ contains a vector coinciding with $\psi(u)$ modulo $\fil^1 \V(\Pi', U)$, which must be singular due to \prref{psisingular}. Thus $\im \DIn$ contains an irreducible $\spd$-summand of degree two singular vectors.
\end{proof}

\begin{proposition}
For each $1 \leq n \leq N-1$, the $\Hdzeta$-module $\VPiprimen$ decomposes as the direct sum of $\im \DIn$ and $\im \DInstar$. Moreover, $\im \DIn$ is an irreducible $\Hdzeta$-module for all $1 \leq n \leq N$.
\end{proposition}
\begin{proof}
Let us treat the case $n=1$ first. We already know from \prref{irreduciblecnonzero} that $\VPiprimezero$ is an irreducible $\Hdzeta$-module. As $\DIzero$ is non-zero, it is injective by irreducibility of $\VPiprimezero$, and its image is then isomorphic to $\VPiprimezero$, hence it is irreducible.

The intersection $\im \DIzero \cap \im \DIzerostar$ is an $\Hdzeta$-submodule of $\im \DIzero$, which is irreducible, and comparing singular vectors shows that $\im \DIzero$ is not contained in $\im \DIzerostar$. Thus the intersection is trivial. As for the sum $\im \DIzero + \im \DIzerostar$, it contains two irreducible summands of singular vectors of the degree two, which must be distinct as we already know the intersection to be trivial. Then $\im \DIzero + \im \DIzerostar$ contains constant vectors in $\VPiprimezero$, hence it coincides with $\VPiprimezero$.

Let us now proceed by induction on $n>1$. Assume we know that $\VPiprimen = \im \DIn \oplus \im \DInstar$; then applying $\DInplusone$, we obtain
$$\im \DInplusone = \DInplusone \VPiprimen =\DInplusone \DInstar \VPiprimenplusone.$$
We know that each tensor module is generated by its constant singular vectors, and the same is true of the image of a tensor module via a $\Hdzeta$-homomorphism. Consequently, $\im\DInplusone$ is spanned over $H$ by $\DInplusone \fil^0 \VPiprimen$, but also by $\DInplusone \DInstar \fil^0 \VPiprimenplusone$, as they are the two $\dd'\oplus \spd$-irreducible summands of $\sing \im \DInplusone$. As every $\Hdzeta$-submodule of $\im \DInplusone$ must contain non-zero singular vectors, this shows that $\im \DInplusone$ has no proper non-zero submodules, i.e., it is irreducible.

Once we know that $\im \DInplusone$ is $\Hdzeta$-irreducible, and $n+1 \neq N$, then we may argue as in the case $n = 1$ that $\im \DInplusone \cap \im \DInplusonestar = 0$ by irreducibility of $\im \DInplusone$ and
$\VPiprimenplusone = \im \DInplusone + \im \DInplusonestar$ because it contains non-zero constant vectors. In other words, $\VPiprimenplusone = \im \DInplusone \oplus \im \DInplusonestar$.
\end{proof}

\begin{proposition}
For each $0 \leq n \leq N-2$, the $\Hdzeta$-module $\im \DInstar$ is irreducible.
\end{proposition}
\begin{proof}
The case $n=0$ is immediate, as $\im \DIzerostar = \VPizerochizero$.

If $n>0$, we know that
$\im \DInstar = \DInstar \VPiprimenplusone = \DInstar (\im \DInplusone \oplus \im \DInplusonestar). $
Thus
$$\im \DInstar = \DInstar \VPiprimenplusone = \DInstar\DInplusone \VPiprimen.$$
Then we may argue, as in the proof of the previous proposition, that $\im \DInstar$ is $H$-linearly generated by each of the two $\dd'\oplus\spd$-irreducible summands of singular vectors it contains. Thus, it is irreducible.
\end{proof}

We have seen above that $\im \DIN \subset \VPiprimeN$ is irreducible, and we know that its singular vectors contain a degree two summand. Let $\eta: \VPiprimeN \to \VPiprimeN$ the $\Hdzeta$-homomorphism mapping nontrivially $\fil^0 \VPiprimeN$ to this summand. Then $\im \eta$ is contained in $\im \DIN$, and coincides with it by irreducibility of $\im \DIN$. In particular, $\eta$ is not injective, as it is not injective on $\sing \VPiprimeN$.

Then $\ker \eta$ is a proper $\Hdzeta$-submodule of $\VPiprimeN$, which cannot contain constant vectors. 
\begin{lemma}
The $\Hdzeta$-submodule $\ker \eta \subset \VPiprimeN$ does not contain singular vectors of degree one.
\end{lemma}
\begin{proof}
As $\eta$ is an $\Hdzeta$-homomorphism, it commutes with the action of $\p_0 \simeq \spd$, which provides a section of $\P_0/\P_1$ in $\P_0$. We have seen that $\gr^d \VPiprimeN \simeq \Pi' \otimes (S^d(\dd) \otimes R(\pi_N))$ as $\spd$-modules, where $\Pi'$ is considered as trivial. However,
$$S^d(\dd) \otimes R(\pi_N)\simeq R(d\pi_1 + \pi_N) \oplus R((d-1)\pi_1 + \pi_{N-1}) \oplus R((d-2)\pi_1 + \pi_N)$$
contains an irreducible summand isomorphic to $R(\pi_{N-1})$ only when $d = 1$, and this summand corresponds to singular vectors of degree one in $\VPiprimeN$.

We know that singular vectors of degree one in $\VPiprimeN$ lie in the image of $\eta$; however, they can only be the image via $\eta$ of singular vectors of degree one, which therefore do not lie in $\ker \eta$.
\end{proof}

As $\ker \eta$ is a proper submodule of $\VPiprimeN$, it must contain nontrivial singular vectors, that can only have degree two. We will denote by $\DIR$ the $\Hdzeta$-endomorphism of $\VPiprimeN$ which maps constant singular vectors to the summand of degree two singular vectors contained in $\ker \eta$. Notice that when $\lambda = 0$, $\DIR$ coincides with $\diru_{\Pi_\phi}$, as usual up to a nonzero multiplicative constant.

\begin{proposition}
The image of $\DIR: \VPiprimeN \to \VPiprimeN$ is irreducible. Moreover,
\begin{enumerate}
\item
$\VPiprimeN = \im \DIR \oplus \im \DIN$;
\item
$\ker \DIR = \im \DIN$;
\item
$\im \DIN \simeq \im \DINminusonestar$;
\item
$\im \DINminusonestar$ is $\Hdzeta$-irreducible.
\end{enumerate}
\end{proposition}
\begin{proof}
The image of $\DIR$ is contained in $\ker \eta$, so also $\sing \im \DIR \subset \sing \ker \eta$. As $\sing \ker \eta$ is $\dd' \oplus \spd$-irreducible, the same holds for $\sing \im \DIR$. However, $\im \DIR = \DIR \VPiprimeN$ is spanned over $H$ by $\DIR \fil^0 \VPiprimeN$ so it is certainly irreducible, as any submodule must contain nontrivial singular vectors, hence all of $\DIR\fil^0 \VPiprimeN$. Let us take care of the remaining claims.
\begin{enumerate}
\item
The direct sum decomposition clearly holds once we know that the intersection is trivial, as the sum must then contain non-zero constant vectors. However, as $\sing \im \DIR$ does not contain singular vectors of degree one, $\im \DIR \cap \im \DIN \neq 0$ would imply $\im \DIR \subsetneq \im \DIN$, which is impossible as $\im \DIN$ is irreducible.
\item 
As $\im \DIN$ and $\im \DIR$ are non-isomorphic $\Hdzeta$-irreducibles, and the image of $\DIR$ is (tautologically) isomorphic to $\im \DIR$, then $\DIR$ must map $\im \DIN$ to $0$.
\item
As $\VPiprimeN = \im \DIR \oplus \im \DIN$ is the direct sum of two irreducible summands, the non-zero image of $\DINminusonestar: \VPiprimeN \to \VPiprimeNminusone$ is isomorphic to either one of the summands or to their sum. However, we know that $\sing \im \DINminusonestar$ has exactly two $\dd'\oplus\spd$-irreducible summands, therefore $\im \DINminusonestar$ is necessarily isomorphic to $\im \DIN$.
\item
Follows from irreducibility of $\im \DIN$.
\end{enumerate}
\end{proof}

We are ready to summarize the results of this Section in the following 
statement.
\begin{theorem}\label{exactderhamnontrivial}
Let $\Pi'$ denote a finite-dimensional irreducible $\dd'$-module with a nontrivial action of $c \in \dd'$. Then the complex of $\Hdzeta$-modules
\begin{equation}\label{splitexactcomplex}
\begin{split}
0 & \to \V(\Pi',R(\pi_0)) \xrightarrow{\DIone} \V(\Pi',R(\pi_1))
\xrightarrow{\DItwo} \cdots \xrightarrow{\DIN} \V(\Pi',R(\pi_N))
\\
&\xrightarrow{\DIR}
\V(\Pi',R(\pi_N)) \xrightarrow{\DINminusonestar} 
\V(\Pi',R(\pi_{N-1})) 
\\
&\xrightarrow{\DINminustwostar} \cdots \xrightarrow{\DIzerostar} 
\V(\Pi',R(\pi_0)) \to 0 \,,
\end{split}
\end{equation}
is exact and $\im \DInstar \simeq \im \DInplusone$ is irreducible for all $n$, as well as $\im \DIR$.
One has $\VPiprimen = \im \DIn \oplus \im \DInstar$ when $0 \leq n \leq N-1$ and $\VPiprimeN = \im \DIN \oplus \im \DIR$.
\end{theorem}
\begin{proof}
Exactness whenever $\DIR$ is not involved follows from the fact that $\im \DInplusone$ is irreducible, hence $\ker \DIn$ is maximal and contains $\im \DIn$. The isomorphism $\im \DInstar \simeq \im \DInplusone$ follows from the fact that $\DInplusone$ acts trivially on $\im\DIn$ and faithfully on $\im\DInstar$. Everything else has already been proved.
\end{proof}

\section{Classification of finite irreducible modules over $\Hd$}

\begin{theorem}\label{tclass}
A complete list of non-trivial pairwise non-isomorphic irreducible finite\/ $\Hd$-modules is as follows:
\begin{itemize}
\item tensor modules\/ $\V(\Pi', U)$ where\/ $\Pi' \boxtimes U$ is a finite-dimensional irreducible\/ $\dd' \oplus \spd$-module and\/ $U$ is not isomorphic to any\/ $R(\pi_n)$, $0 \leq n \leq N$;
\item $\imdindinminusonestar, 1 \leq n \leq N$, where\/ $\Pi$ is a finite-dimensional irreducible\/ $\dd$-module;
\item $\im \DIn, 1 \leq n \leq N$, where $\Pi'$ is a finite-dimensional irreducible \/ $\dd'$-module with a non-trivial action of $c \in \dd'$;
\item $\im \DIR$, where $\Pi'$ is a finite-dimensional irreducible \/ $\dd'$-module with a non-trivial action of $c \in \dd'$;
\end{itemize}
\end{theorem}
\begin{proof}
Irreducible $\Hd$-modules are all obtained as quotients of a tensor module $\V(\Pi', U)$, where $\Pi' \boxtimes U$ is a finite-dimensional irreducible $\dd'\oplus \spd$ representation, by a maximal submodule. If $U$ is not isomorphic to any $R(\pi_n), 0 \leq n \leq N$, then \thref{irrcriterion} shows that there are no non-zero submodules; the last two sections compute maximal submodules in all other cases. Indeed, when $c$ acts nontrivially on $\Pi'$, the description of irreducible modules follows from \thref{exactderhamnontrivial}.

When $c$ acts trivially on $\Pi'$, then $\Pi' = \pi^* \Pi$ for some $\dd$-module $\Pi$, where  $\pi: \dd' \to \dd'/\kk c \simeq \dd$ denotes the canonical projection. Henceforth, we will use the notation introduced in \deref{dvmodw}(iii).

Thus the unique maximal submodule of $\VPin, 1 \leq n <  N,$ is $M = \imdin + \imdinstar$, which clearly lies in the kernel of the homomorphisms
$$\dindinminusonestar , \qquad\dinstar \dinplusone$$
which map $\VPichin$ to $\VPin$ and coincide up to multiplication by a non-zero scalar.

As the above maps are non-zero, their kernel must coincide with $M$, hence
$$\imdindinminusonestar\simeq \VPin/M$$
is the unique irreducible quotient of $\VPin$.

As for $\VPiN$, its unique maximal submodule is $M = \imdiN$, which lies in the kernel of the non-zero map
$\DiPi: \VPiN \to \V(\Pi_{-\chi}, R(\pi_N)),$ so that
$$\imDiPi \simeq \VPiN/M$$ is the unique irreducible quotient of $\VPiN$.

Finally, the only proper non-zero submodule of $\VPizero$ is  $\imdizerostar$, which is of finite $\kk$-codimension. Its $\dd \oplus \spd$-irreducible quotient is thus a torsion $H$-module, which is thence endowed with a trivial pseudoaction.
\end{proof}

As representations of the Lie pseudoalgebra $\Hd$ are in one-to-one correspondence with conformal (in the sense of the definition given in \seref{spsanih}) representations of its extended annihilation algebra $\dd \sd \P$, one may use the above results to deduce a classification of irreducible (topological, discrete) $P_{2N}$-modules and of singular vectors in Verma modules induced from irreducible $\sp_{2N}$ representations as follows.

Let $\dd$ be a $2N$-dimensional Lie algebra endowed with a Frobenius structure $\omega = \di \bla$; for instance, one may choose $\dd$ to be the direct sum of $N$ copies of the non-abelian $2$-dimensional Lie algebra. The extended annihilation algebra of $\Hdzeta = \dd \sd \P$ decomposes as a direct sum of Lie algebras $\dd^\bla \oplus \P$. Indeed, as $\chi = 0$, elements $\tilde{\d}$ act on $\P$ by inner derivations, and as $\omega = \di \bla$, the central extension yielding $\dd'$ may be trivialized to $\dd' = \dd^\bla + \kk c$. Let $U$ be a representation of $\spd$. Using $\spd \simeq \P_0/\P_1$, we endow $U$ with an action of $\P_0$ and extend it to $\N_\P$ by letting $\dd^\bla$ act trivially on $U$ and $c$ act via multiplication by a scalar $\lambda \in \kk$. Then the induced module $\Ind_{\N_\P}^{\tilde \P}$ is isomorphic to the tensor module $\V(\kk_\lambda, U)$, where $\kk_\lambda$ denotes the $1$-dimensional $\dd'$-module where $c$ acts as $\lambda$ and $\dd^\bla$ acts trivially. Notice that elements in $\dd^\bla$ act trivially on all of $\V(\kk_\lambda, U)$, so that the submodule lattice of $\V(\kk_\lambda, U)$ as a $\Hdzeta$-module only depends on its structure as a representation of $\P$. Also, when $\lambda = 0$, the action of $\P$ factors through its quotient $\H$.

We thus recover an old result of Rudakov as claim $(i)$ and $(ii)$ of the following proposition.
\begin{proposition}[\cite{Ru2}]
\label{Rudakov}
\mbox{ }
\begin{itemize}
\item[\textit{(i)}]
Every nonconstant homogeneous singular vector in the $\H$-module $V = \V(\kk, U)$ has degree one or two. The space $S$ of such singular vectors is an $\spd$-module and the quotient of $V$ by the $\H$-submodule generated by $S$ is an irreducible $\H$-module. All singular vectors of degree one are listed in cases (a), (b) below, while all singular vectors of degree two are listed in case (c):
\begin{itemize}
\item[\textrm{(a)}] $U = R(\pi_p), \qquad S = R(\pi_{p+1}), \quad 0 \leq p \leq N-1.$
\item[\textrm{(b)}] $U = R(\pi_p), \qquad S = R(\pi_{p-1}), \quad 1 \leq p \leq N.$
\item[\textrm{(c)}] $U = R(\pi_p), \qquad S = R(\pi_p), \quad 1 \leq p \leq N.$
\end{itemize}
\item[\textit{(ii)}]
If the $\spd$-module $U$ is infinite-dimensional irreducible, then $\V(\kk, U)$ does not contain non-constant singular vectors.
\item[\textit{(iii)}]
If a $\H$-module $\V(\kk, U)$ is not irreducible, then its unique irreducible quotient is isomorphic to the topological dual of the kernel of the differential of a member of the Eastwood complex over formal power series.
\end{itemize}
\end{proposition}
The next proposition deals with conformal representations of $\P\simeq P_{2N}$ with non-trivial action of the center.
\begin{proposition}
\mbox{ }
\begin{itemize}
\item[\textit{(i)}]
Let $0 \neq \lambda \in \kk$. Every nonconstant homogeneous singular vector in the $\P$-module $V = \V(\kk_\lambda, U)$ lies in $\fil^2 \V(\kk_\lambda, U)$. The space $S$ of such singular vectors is an $\spd$-module and each nonconstant $\spd$-irreducible summand lying in $\fil^1 \V(\kk_\lambda, U)$ generates an irreducible $\P$-module.
The description of the subspaces of singular vectors it the same as in \prref{Rudakov} $(i)$
\item[\textit{(ii)}]
If the $\spd$-module $U$ is infinite-dimensional irreducible, then $\V(\kk_\lambda, U)$ does not contain non-constant singular vectors.
\item[\textit{(iii)}]
If a $\P$-module $\V(\kk_\lambda, U)$ is not irreducible, then it decomposes as a direct sum of its two irreducible submodules.
\end{itemize}
\end{proposition}

\bibliographystyle{amsalpha}


\end{document}